\documentclass[table]{amsart}
\usepackage[margin=0.9in]{geometry}
\usepackage{amsmath,amssymb,amsthm,mathtools,mathrsfs, tikz}
\usepackage{enumitem}
\usepackage{pgfplots}
\pgfplotsset{compat=1.18}
\usepackage{xcolor}
\usepackage{hyperref}
\usepackage{orcidlink}

\definecolor{darkgoldenrod}{rgb}{0.72,0.53,0.04}
\definecolor{goldmetallic}{rgb}{0.83,0.69,0.22}

\hypersetup{
 colorlinks=true,
 linkcolor=darkgoldenrod,
 filecolor=brown,
 urlcolor=goldmetallic,
 citecolor=darkgoldenrod,
}

\numberwithin{equation}{section}

\newtheorem{theorem}{Theorem}[section]
\newtheorem{proposition}[theorem]{Proposition}
\newtheorem{lemma}[theorem]{Lemma}
\newtheorem{corollary}[theorem]{Corollary}
\newtheorem{assumption}[theorem]{Assumption}

\newtheorem{remark}[theorem]{Remark}
\theoremstyle{definition}
\newtheorem{definition}[theorem]{Definition}
\newtheorem{example}[theorem]{Example}

\newtheorem{lconj}{Conjecture}

\newcommand{\A}{\mathbb A}
\newcommand{\Gm}{\mathbb G_m}
\newcommand{\Z}{\mathbb Z}
\newcommand{\Q}{\mathbb Q}
\newcommand{\R}{\mathbb R}
\newcommand{\C}{\mathbb C}
\newcommand{\cB}{\mathcal B}
\newcommand{\cY}{\mathscr Y}
\newcommand{\vis}{\operatorname{vis}}
\newcommand{\bad}{\operatorname{bad}}
\newcommand{\GL}{\operatorname{GL}}

\title[Local-global visibility]{Local-global principles for visibility of lattice points on parameterized curves}
\author[S.~Chaubey]{Sneha Chaubey}
\address{Department of Mathematics,
Indraprastha Institute of Information Technology Delhi,
New Delhi 110020, India}
\email{sneha@iiitd.ac.in}

\author[A.~Ray]{Anwesh Ray\, \orcidlink{0000-0001-6946-1559}}
\address{Chennai Mathematical Institute, Chennai, India}
\email{anwesh@cmi.ac.in}


\begin{document}

\begin{abstract}
We develop a local-global theory for visibility of lattice points on families of parameterized curves.  We introduce a notion of \(p\)-adic visibility and ask whether a lattice point is globally visible precisely when it is visible at every prime.  We prove that this local-global principle holds for a broad class of families whose parametrizations are homogeneous with respect to positive weights.  When the points on these curves fill the entire positive integer lattice, we compute the local and global densities of visible points and show that the global density is the product of the local densities.  We then consider polynomial families of the form \(y=qP(x)\), with \(q\in\Q_{>0}\), and show that visibility can be detected prime by prime exactly when \(P\) is a monomial.  For non-monomial polynomials the local-global principle can fail, but the set of points where it fails has density zero; for separable polynomials we also obtain a quantitative bound for the number of non-visible points, improving the previously known bound. We further consider families whose lattice points lie on a proper lower-dimensional algebraic subset of the ambient space and show that their visibility densities can behave differently from those of the full lattice.  Finally, we extend the theory from visibility from the origin to visibility from one lattice point to another and show that the corresponding local-global principle continues to hold for weighted homogeneous families.
\end{abstract}
\maketitle

\setcounter{tocdepth}{1}
\tableofcontents
\section{Introduction}
\subsection{Motivation and background}

Local-global principles ask whether a global arithmetic property over \(\Q\)
can be detected over \(\R\) and over \(\Q_p\) for every prime \(p\). The
basic example is the Hasse--Minkowski theorem, which states that a quadratic form over \(\Q\) has
a nontrivial rational zero if and only if it has a nontrivial zero over \(\R\)
and over every \(\Q_p\); cf. \cite[p.~41, Chapter~IV]{Serre}. For
general varieties this implication can fail. The Brauer--Manin obstruction is
one systematic way to explain such failures, see
\cite{SkorobogatovTorsors} and \cite{PoonenRationalPoints} for further details. A similar local-global philosophy appears in strong approximation for
algebraic groups and in thin-orbit problems; see
\cite[Chapter~7]{PlatonovRapinchuk}.  For integral Apollonian circle packings, 
Graham--Lagarias--Mallows--Wilks--Yan \cite{GLMW} and Fuchs-Sanden \cite{FS} introduced the local-global conjecture. Haag--Kertzer--Rickards--Stange showed that this conjecture fails in general by producing primitive packings with missed admissible curvatures \cite{HaagKertzerRickardsStange}. The present paper studies a
visibility analogue: we ask when visibility of lattice points on parameterized
curves is detected prime by prime, and when the exceptional set has density
zero.

Given an integer $n\geq 2$, the classical problem of lattice-point visibility asks which points of
\(\Z^n\) can be seen from the origin along straight lines.  A point $a=(a_1,\dots,a_n)\in \Z^n$ is visible from the origin precisely when the line segment from the origin to \(a\)
contains no other lattice point.  Equivalently, $\gcd(a_1,\dots,a_n)=1$. Thus, classical visibility is already a local condition: \(a\) is visible if
and only if, for every prime \(p\), not all coordinates of \(a\) are divisible
by \(p\). Given $\mathcal{S}\subseteq \mathbb{Z}_{\geq 1}^n$, its natural density is defined as the following limit
\[\mathfrak{d}(\mathcal{S}):=\lim_{X\rightarrow \infty} \frac{|\mathcal{S}\cap [1,X]^n|}{X^n},\] provided that this limit exists. It is easy to see that the density of visible points in \(\Z_{\geq 1}^n\) is
\[
        \prod_p\left(1-\frac1{p^n}\right)
        =
        \frac1{\zeta(n)}.
\]

Classical questions about the distribution and
geometry of visible and invisible lattice points were studied by
Herzog--Stewart \cite{HerzogStewart1971}, and later refinements include
visibility in the plane \cite{AdhikariGranville2009}, index and joint
visibility \cite{ChaubeyTamazyanZaharescu2019}, and simultaneous visibility
from several prescribed points \cite{BerendKumarPollington2023}. The fine-scale distribution of visible lattice points has also been studied by
Boca--Cobeli-Zaharescu, who proved limiting results for the angular spacings
of lattice points visible from the origin \cite{BocaCobeliZaharescu2000}. More recently, visibility questions with congruence restrictions have been studied by
Shparlinski and Winterhof, who obtained estimates for visible points on
multidimensional modular hyperbolas \cite{ShparlinskiWinterhof2008}. Visibility has also been studied along nonlinear families of curves. Goins--Harris-Kubik-Mbirika \cite{GoinsHarrisKubikMbirika2018} introduced \(b\)-visibility from the origin: a
point \((r,s)\in\Z^2\) is visible along a generalized line of sight if it lies
on a curve $y=qx^b$ where $q\in\Q$ and no earlier lattice point lies on the same curve between the origin and
\((r,s)\). They showed that the density of such visible points is $\frac1{\zeta(b+1)}$. Harris and Omar \cite{HarrisOmar2018} extended this circle of questions to
power functions with rational exponents. Polynomial lines of sight provide a particularly important test case. In work of the first author with Pandey and Regavim \cite{ChaubeyPandeyRegavim2026}, the authors consider families $y=qP(x)$ where \(P\in\Z[x]\) has nonnegative coefficients and \(P(0)=0\).  For monomials \(P(x)=x^b\), the theory
reduces to weighted visibility and gives the density \(1/\zeta(b+1)\). However
, for genuinely non-monomial polynomials, they \cite{ChaubeyPandeyRegavim2026} prove that 
visible points have density one.

In this article, we work in a general framework. We consider parameterized curves contained within a real affine variety $X\subseteq \mathbb{R}^n$
defined over \(\Q\), and assume that \(0\in X\). All local visibility notions below are relative to this fixed affine embedding and its coordinate functions. We consider a collection of parameterized curves $\cY=\{(U_\alpha,\phi_\alpha)\}_{\alpha\in I}$. Here \(I\) is an index set, \(U_\alpha\subseteq\A^1_{\Q}\) is a Zariski open subset, and $\phi_\alpha:U_\alpha\longrightarrow X$ is a rational parametrization of a curve. We write $U_\alpha^+:=U_\alpha(\R)\cap \R_{>0}$. The order on each curve is the order of the parameter, i.e., the point
\(\phi_\alpha(s)\) precedes \(\phi_\alpha(t)\) if \(0<s<t\). We define the branch locus
\begin{equation}\label{cBY}\cB(\cY)
        :=
        \{\,\phi_\alpha(t)\in X(\Z)\mid \alpha\in I,\ t\in U_\alpha^+\,\}.
\end{equation} This is called a \emph{visibility datum} if each branch passes through the origin only at \(t=0\) and distinct branches do not meet away from the origin. We refer to Definition \ref{def:visibility-datum} for further details.

We say that a point in this branch locus is globally visible if there is no earlier
integral point on the same parameterized curve. Denote the set of globally visible points by $\cB^{\vis}(\cY)\subseteq \cB(\cY)$.

There is also a local version of visibility at each prime. For a point $P=(P_1,\dots,P_n)\in \Z^n$,
we define its \(p\)-adic minimum valuation by
\[
        v_p(P):=\min_i v_p(P_i).
\]
We say that a point in the branch locus is \(p\)-adically visible if no earlier integral
point on the same curve has strictly smaller \(p\)-adic minimum valuation. We denote the set of \(p\)-adically visible points by $\cB_p^{\vis}(\cY)\subseteq \cB(\cY)$.

Global visibility always implies \(p\)-adic visibility for every prime, and hence
\[
        \cB^{\vis}(\cY)
        \subseteq
        \bigcap_p \cB_p^{\vis}(\cY).
\]
The basic local-global question is whether the converse inclusion holds.  We
say that visibility is locally detectable for \(\cY\) if
\begin{equation}\label{eq:local-detectability} \cB^{\vis}(\cY)
        =
        \bigcap_p \cB_p^{\vis}(\cY).
\end{equation}
The possible failure of this
principle is measured by the local defect set
\begin{equation}\label{defect set}\cB^{\bad}(\cY)
        :=
        \left(\bigcap_p \cB_p^{\vis}(\cY)\right)
        \setminus \cB^{\vis}(\cY).
\end{equation}
Therefore, visibility is locally detectable precisely when $\cB^{\bad}(\cY)=\varnothing$.
\par For $P=(P_1,\dots,P_n)$, we define its height by
\begin{equation}\label{defn of H}H(P):=\max_{1\leq i\leq n}|P_i|.
\end{equation}
For \(X>0\), we then set
\[
        \cB(\cY;X)
        :=
        \{P\in\cB(\cY):H(P)\leq X\}.
\]
Given a subset \(\mathcal{S}\subseteq \cB(\cY)\), we define its relative density in the branch
locus by
\[
        \mathfrak d_{\cB(\cY)}(\mathcal{S})
        :=
        \lim_{X\to\infty}
        \frac{\#\bigl(\mathcal{S}\cap \cB(\cY;X)\bigr)}
             {\#\cB(\cY;X)},
\]
provided the limit exists. We say that \(\mathcal{S}\) has density zero relative to \(\cB(\cY)\) if $\mathfrak d_{\cB(\cY)}(\mathcal{S})=0$. We set \[\mathfrak{d}^{\vis}(\cY):=\mathfrak{d}_{\cB(\cY)}\left(\cB^{\vis}(\cY)\right)\quad \text{and}\quad \mathfrak{d}^{\bad}(\cY):=\mathfrak{d}_{\cB(\cY)}\left(\cB^{\bad}(\cY)\right),\]provided they exist.
The key examples studied in this paper show that $\mathfrak{d}^{\vis}(\cY)$ decomposes into a product of local terms that are interpreted as densities of $p$-adically visible points and that $\mathfrak{d}^{\bad}(\cY)=0$.

\subsection{Main results}
The first main result gives a geometric criterion for local
detectability.  Let \(\mathbb G_m\) denote the multiplicative group.  Given
positive integers \(w_1,\dots,w_n\), we let \(\mathbb G_m\) act on
\(\mathbb{R}^n\) by
\[
        d\cdot (x_1,\dots,x_n)
        =
        (d^{w_1}x_1,\dots,d^{w_n}x_n).
\]
We call this the positive-weight action with weights \(w_1,\dots,w_n\).

\begin{theorem}\label{thm:local-detectability}
Let \(\cY=\{(U_\alpha,\phi_\alpha)\}_{\alpha\in I}\) be a visibility datum in
\(X\subseteq \mathbb{R}^n\). Writing $\phi_\alpha
        =
        (\phi_{\alpha,1},\dots,\phi_{\alpha,n})$, assume that there are
positive integers \(w_1,\dots,w_n\) such that, for every \(\alpha\in I\), every
\(t\in U_\alpha\), and every \(d>0\) for which both sides are defined, one has
\[
        \phi_\alpha(dt)
        =
        d\cdot \phi_\alpha(t)
        =
        \bigl(d^{w_1}\phi_{\alpha,1}(t),\dots,
              d^{w_n}\phi_{\alpha,n}(t)\bigr).
\]
Then visibility is locally detectable for \(\cY\), that is, \eqref{eq:local-detectability} holds.
\end{theorem}
\begin{remark}\label{remark1.2}The theorem applies to weighted homogeneous parametrizations
\begin{equation}\label{weightedfamilies} \phi_\alpha(t)
        =
        \bigl(a_1(\alpha)t^{w_1},\dots,
              a_n(\alpha)t^{w_n}\bigr),
\end{equation}
where \(w_i\in\Z_{>0}\) and \(a_i:I\to\Q_{>0}\). For example, taking $w_1=1$, $w_2=b\in\Z_{\geq 2}$, $I=\Q_{>0}$
with $a_1(\alpha)=1$ and $a_2(\alpha)=\alpha$ gives us the family $\phi_\alpha(t)=(t,\alpha t^b)$. The corresponding branches are the curves $y=\alpha x^b$ in the positive quadrant \(\R_{\geq 0}^2\).

The weighted-projective 
interpretation of these families and its relation to the
local-global principle for visibility is discussed in
 Section~\ref{subsec:weighted-projective}.
\end{remark}

In many natural examples, the branch locus is the entire positive
lattice \(\mathbb Z_{\geq1}^n\). In that case, the local criterion
also gives an explicit density formula.

\begin{theorem}\label{thm:full-lattice-density}
Let \(n\geq 2\), and let \(\cY\) be a weighted homogenous family as in \eqref{weightedfamilies} whose branches are
defined for all \(t>0\). Set $W=w_1+\cdots+w_n$; assume that 
\begin{enumerate}
    \item[(i)] $\gcd(w_1,\dots,w_n)=1$, and
    \item[(ii)]$\cB(\cY)=\Z_{\geq 1}^n$.
\end{enumerate}
 Then, for every prime \(p\), the relative density of \(p\)-adically visible points is
\[
        \mathfrak d_p^{\operatorname{vis}}(\cY)
        =
        1-\frac1{p^W}.
\]
Moreover, the relative density of globally visible points equals
\[
        \mathfrak d^{\operatorname{vis}}(\cY)=
        \frac1{\zeta(W)}.
\]
\end{theorem}
\begin{remark}
    It follows from the above Theorem that the global density equals the product of local densities, that is, 
\[\mathfrak d^{\operatorname{vis}}(\cY)=\prod_p \mathfrak d_p^{\operatorname{vis}}(\cY).\]
\end{remark}

We also study nonhomogeneous polynomial families. Let
\[
P(T)=T^d+a_{d-1}T^{d-1}+\cdots+a_1T\in\Z[T],
\]
where $a_i\geq 0$ and consider the polynomial visibility datum
\[
        \cY_P:=\{(\A^1,\phi_q):q\in\Q_{>0}\},
\]
where $\phi_q(t)=(t,qP(t))$. For a fixed branch \(q=A/B\), the integral points on that branch are governed
by the congruence condition $B\mid P(x)$.

In the monomial case \(P(x)=x^d\), the parametrization
\[
        \phi_q(t)=(t,qt^d)
\]
is homogeneous for the action with weights \((1,d)\), and hence Theorem
\ref{thm:local-detectability} applies.  The following
result shows that among the polynomial families
\(y=qP(x)\), the local-global principle for visibility holds
only in the monomial case.

\begin{theorem}\label{thm:polynomial-local-global}
Let
\[
        P(T)=T^d+a_{d-1}T^{d-1}+\cdots+a_1T\in\Z[T],
\]
with $a_i\geq 0$ for $i=1, \dots, d-1$. Consider the family
\[
        \phi_q(t)=(t,qP(t)),
\]
where $q\in\Q_{>0}$. Then visibility is locally detectable for this family if and only if $P(x)=x^d$.
\end{theorem}
By the aforementioned result in
\cite{ChaubeyPandeyRegavim2026}, the non-visible locus has density zero for
non-monomial \(P\). Consequently, $\mathfrak{d}^{\vis}(\cY_P)=1$. Moreover, since
\[
\cB^{\bad}(\cY_P)
\subseteq
\cB(\cY_P)\setminus\cB^{\vis}(\cY_P),
\]
the defect set also has density zero (cf. Theorem \ref{thm:polynomial-defect-density}).

\par In Section \ref{s 6}, we also obtain quantitative results for polynomial families for which
visibility is not locally detectable. Write \(U(P;X)\) for the set of
invisible points of $\cY_P$ lying in \([1,X]^2\). It is shown in \cite{ChaubeyPandeyRegavim2026} that for separable \(P\) of degree \(d\),
the upper bound
\(O_{P,\varepsilon}(X^{2-1/(2d-1)+\varepsilon})\) holds.
We obtain the following estimate, which is stronger for every
degree \(d\geq3\).

\begin{theorem}\label{thm:intro-separable-polynomial}
Let \(P\) be as above and suppose that \(P\) is separable of degree
\(d\geq2\).  Then, for every \(\varepsilon>0\),
\[
        \#U(P;X)
        \ll_{P,\varepsilon}
        X^{\frac32+\frac1{2d}+\varepsilon}.
\]
In particular, the proportion of invisible points in \([1,X]^2\) is
\(O_{P,\varepsilon}
(X^{-1/2+1/(2d)+\varepsilon})\).\end{theorem}
\par We also study the quantitative size of the local defect set itself. For a
polynomial \(P\), write
\[
        D_P(X)
        :=
        \#\bigl(
        \cB^{\bad}(\cY_P)\cap[1,X]^2
        \bigr).
\]
\begin{theorem}\label{thm:intro-quantitative-defect}
Assume that $P(T)\in \Z[T]$ is a monic polynomial of degree $d\geq 2$ with non-negative coefficients and that $P(T)\neq T^d$. Then we have the general lower bound
\[D_P(X)\gg_P X.\]
Furthermore, if \(P(T)=T^r(m+nT)\), where \(r\geq2\),
\(m,n\geq1\), and \(\gcd(m,n)=1\), then, for every
\(\varepsilon>0\),
\[
        X\log X
        \ll_P
        D_P(X)
        \ll_{P,\varepsilon}
        X^{7/4+\varepsilon}.
\]
If \(r=2\), so that \(P(T)=T^2(m+nT)\) is a nonseparable cubic, then
the upper bound improves to
\(D_P(X)\ll_{P,\varepsilon}X^{5/3+\varepsilon}\).
\end{theorem}
\par Finally, in Section~\ref{sec:random-pairs}, we extend the preceding theory to
visibility between two lattice points.  The essential point is that visibility
of \(B\) from \(A\) depends only on the displacement \(B-A\), so that the
local-global principle is unchanged by translation.  To treat arbitrary
displacement vectors, including those with negative or zero coordinates, we
introduce signed weighted homogeneous families.  These are the natural
extensions of our positive-weight homogeneous families which cover the
whole lattice rather than only the positive orthant.  For these families, we
obtain the following two-point analog of our earlier density theorem.

\begin{theorem}
\label{thm:signed-weighted-pair-visibility}
Fix positive integer weights with greatest common divisor \(1\), and let \(W\)
be their sum.  For the corresponding signed weighted homogeneous family, a
lattice point \(B\) is visible from a lattice point \(A\) if and only if it is
\(p\)-adically visible from \(A\) for every prime \(p\).  Moreover, the density
of ordered pairs for which this visibility condition holds is $\frac{1}{\zeta(W)}$. The local density at a prime \(p\) is \(1-p^{-W}\), so the
global density is the product of the local densities.
\end{theorem}
\noindent For a more precise formulation of this result, see Theorem \ref{thm:signed-weighted-pair-visibility}.

\par We also study in Section~\ref{sec:sparse} what happens when the branch
locus is \emph{sparse}, which means that it is contained in a
lower-dimensional algebraic subvariety of the ambient affine space.
 We consider two different situations.  In the first, the
sparse set is still described by freely choosing several integral
coordinates and using homogeneous polynomials to determine the remaining
coordinates.  More precisely, starting from homogeneous polynomials
\(F_1,\dots,F_s\) in \(r\) variables, we consider points of the form
\[
 (x_1,\dots,x_r,F_1(x_1,\dots,x_r),\dots,F_s(x_1,\dots,x_r)).
\]
We call these \emph{sparse radial graph families}; see
Section~\ref{sec:sparse}. Under a natural
condition on the shape of the region obtained by imposing a height
bound, Theorem~\ref{thm:sparse-Davenport} shows that visibility is
equivalent to \(\gcd(x_1,\dots,x_r)=1\), that the local density at a
prime \(p\) is \(1-p^{-r}\), and that the global density is
\(1/\zeta(r)\).  We work this out explicitly for the family
\(\phi_a(t)=(t,at,a^2t^2)\) in
Example~\ref{prop:first-sparse-family}, as well as for a
higher-dimensional example in Example~\ref{ex:davenport-sparse-family}.
The hypothesis in Theorem~\ref{thm:sparse-Davenport} is not merely a
technical convenience: Example~\ref{ex:sparse-cusp}, based on
\(\Psi(x,y)=(x,y,x^dy)\), shows that when it fails the density can
change from \(1/\zeta(2)\) to \(1/\zeta(d+1)\).

\par The second sparse situation is quite different.  Rather than describing
the branch locus as a polynomial graph over a set of independent
integral coordinates, we start with a homogeneous affine variety and
take the branches to be the rational rays through the origin.  Thus,
for example, on the Pythagorean cone \(X^2+Y^2=Z^2\), each branch is
simply a ray through a rational point of the cone.  In this setting a
point is visible precisely when its ambient coordinates are relatively
prime, and this condition is detected prime by prime; this is
Proposition~\ref{prop:cone-local-global}.  The density question is then
reduced to understanding how many primitive integral points of bounded
height lie on the variety.  Proposition~\ref{prop:cone-density} shows
that the exponent governing the growth of these primitive points also
governs the resulting local and global visibility densities.  We
illustrate this principle with the Pythagorean cone in
Proposition~\ref{prop:Pythagorean-visibility}, the quadratic cone
\(XZ=Y^2\) in Proposition~\ref{prop:quadratic-cone-density}, the rank-one cone
 \(2\times2\) matrices in
Proposition~\ref{prop:rank-one-density}, and the quadratic Veronese
cone in Proposition~\ref{prop:Veronese-density}.  These examples show
in particular that there is no single zeta exponent determined by the
dimension of the sparse variety: the global density can be zero, can
equal \(1/\zeta(2)\), and in the Veronese example is
\(1/\zeta(3/2)\). Thus, even a nonintegral exponent arises naturally
from the growth of integral points with respect to the chosen height.
\subsection{Open questions}

The above results suggest that visibility can be viewed as a local-global
phenomenon.  Exact local detectability asks for the equality
\[
        \cB^{\operatorname{vis}}(\cY)
        =
        \bigcap_p \cB_p^{\operatorname{vis}}(\cY).
\]
This equality holds for positive-weight homogeneous families, but Theorem \ref{thm:polynomial-local-global} shows that it fails
for non-monomial polynomial families. Nevertheless, the defect set $\cB^{\bad}(\cY_P)$ has density $0$. These observations naturally lead us to make the following conjecture.

\begin{lconj}
Let \(\cY\) be a visibility datum such that \(\#\cB(\cY;X)\) has polynomial growth, every local visibility condition has a relative density, and the corresponding large-prime tail satisfies a uniform estimate analogous to \eqref{eq:large-prime-tail}. Then $\mathfrak{d}^{\bad}(\cY)=0$.
\end{lconj}

The examples in Section 
\ref{sec:sparse} suggest the following conjecture.

\begin{lconj}
Let \(n\geq 2\), and let \(\cY\) be a visibility datum whose branches are
defined for all \(t>0\) and have the form
\[
        \phi_\alpha(t)
        =
        \bigl(a_1(\alpha)t^{w_1},\dots,
              a_n(\alpha)t^{w_n}\bigr),
\]
where \(w_1,\dots,w_n\in\Z_{\geq 1}\) and
\(a_i(\alpha)\in\Q_{>0}\). Assume that $\gcd(w_1,\dots,w_n)=1$ and suppose that \(\cB(\cY)\) is sparse. Then, for every prime \(p\), the relative density \(\mathfrak d_p^{\operatorname{vis}}(\cY)\) should exist, the Euler product of the local densities should converge, and the relative density of globally visible
points should exist and satisfy
\[
        \mathfrak d^{\operatorname{vis}}(\cY)
        =
        \prod_p \mathfrak d_p^{\operatorname{vis}}(\cY),
\]
where both the global density and the local densities are taken relative to
\(\cB(\cY)\).
\end{lconj}
\noindent Note that this Conjecture holds when conditions (i) and (ii) of Theorem \ref{thm:full-lattice-density} are satisfied.
\par For all the sparse examples considered in this article, it turns out that the global density is of the form $\frac{1}{\zeta(n)}$ for some rational number $n>1$. It is natural to ask if this is the case for all weighted homogenous families \eqref{weightedfamilies}?

\subsection*{Acknowledgements}
SC gratefully acknowledges partial support from the Core Research Grant 
CRG/2023/001743 from ANRF, Department of Science and Technology (DST), GoI. SC and AR thank the International Center for Theoretical Sciences, Bengaluru,
where this project was initiated during the discussion meeting
\emph{The Classical Circle Method and the Large Sieve}, held in May,
2026. AR also thanks the Department of Mathematics at
IIIT Delhi for its hospitality during a visit during which part of this work was
carried out. SC and AR thank John Voight for valuable discussions during the preparation of this article.
\section{Ordered branches and visibility}

Throughout, we shall fix an affine variety
$X\subseteq \mathbb{A}_{\mathbb{R}}^n$ defined over $\Q$ and assume that $0\in X$. Recall that a rational map $\phi:\A^1_{\mathbb{R}}\dashrightarrow X$ is defined on a Zariski open subset
$W\subseteq \A^1_{\R}$ and, on real points $t\in W(\R)$, is given by
\[
\phi(t)
=
\bigl(
f_1(t)/g_1(t),\dots,f_n(t)/g_n(t)
\bigr),
\]
where $f_i$ and $g_i$ are polynomials with coefficients in $\Q$ and $g_i$
are all non-vanishing on $W$. We set $X(\Q):=X\cap \Q^n$ and $X(\Z):=X\cap \Z^n$.

\begin{definition}
An ordered branch on \(X\) is a pair \((U,\phi)\), where
\(U\subset \mathbb A_{\mathbb R}^1\) is a nonempty Zariski open subset defined over
\(\mathbb Q\) with \(0\in U(\Q)\), and $\phi:U\longrightarrow X$ is the restriction of a rational map \(\mathbb A^1_{\R}\dashrightarrow X\) which
is regular on \(U\). We set $U^+:=U(\R)\cap \mathbb R_{>0}$ and assume that the restriction of \(\phi\) to \(U^+\) is injective.
For \(t,t'\in U^+\), \(\phi(t)\) is said to precede \(\phi(t')\) on the
branch if \(t<t'\).
\end{definition}

\begin{definition}\label{def:visibility-datum}
A visibility datum is a collection
\[
\cY=\{(U_\alpha,\phi_\alpha)\}_{\alpha\in I}
\]
of ordered branches such that the following conditions hold:
\begin{enumerate}
    \item $\phi_\alpha(0)=0$ for all $\alpha\in I$ and, if $t\in U_\alpha(\R)$ and $\phi_\alpha(t)=0$, then $t=0$.
    \item For $\alpha\neq \beta$, we have that \[\phi_\alpha(U_\alpha^+)\cap \phi_\beta(U_\beta^+)=\emptyset.\] 
\end{enumerate}
\end{definition}
\noindent For $\alpha\in I$ and $t\in U_\alpha^+$ we set $P_{\alpha,t}:=\phi_\alpha(t)$ and denote by $\cB(\cY)$ the set of points $P_{\alpha,t}\in X(\Z)$ where $t\in U_\alpha^+$, as in \eqref{cBY}.
\begin{definition}
A branch-point $P_{\alpha,t}\in \cB(\cY)$ is globally visible if there is no
$s\in U_\alpha^+$ such that $0<s<t$ and $P_{\alpha,s}\in \cB(\cY)\cap\mathbb{Z}^n$. Equivalently, $P_{\alpha, t}$ is globally visible if it is the first integral
point on its branch.
\end{definition}
Let $p$ be a prime number. For $x\in \Q$, we write $v_p(x)$ for the
$p$-adic valuation normalized by $v_p(p)=1$, and we set
$v_p(0)=+\infty$. For a vector $P=(P_1,\dots,P_n)\in \Z^n$, define
\[
v_p(P):=\min_{1\leq i\leq n}v_p(P_i).
\]

\begin{definition}
Let $p$ be a prime. A branch-point $P_{\alpha,t}\in \cB(\cY)$ is
$p$-adically visible if, for every $s\in U_\alpha^+$ such that $0<s<t$ and $P_{\alpha,s}\in \Z^n$, one has
\[
v_p(P_{\alpha,t})\leq v_p(P_{\alpha,s}).
\]
\end{definition}Recall that $\cB^{\vis}(\cY)$ is the set of globally visible points in $\cB(\cY)$ and $\cB_p^{\vis}(\cY)$ for the set of $p$-adically visible branch-points, also that the defect set $\cB^{\bad}(\cY)$ is given by \eqref{defect set}. We use the abbreviation $\mathfrak d^{\vis}(\cY):=\mathfrak d_{\cB(\cY)}\left(\cB^{\vis}(\cY)\right)$ for the density of visible points.

\par For a branch-point
$P_{\alpha,t}\in \cB(\cY)$, define the set of earlier integral parameters
\[
E_\alpha(t)
:=
\left\{
s\in U_\alpha:
0<s<t,\;
\phi_\alpha(s)\in \Z^n
\right\}.
\]
\begin{proposition}\label{prop:obstruction}
Let $P_{\alpha, t}\in \cB(\cY)$. Then $P_{\alpha, t}\in \cB^{\bad}(\cY)$ if and only if the following two conditions hold:
\begin{enumerate}[label=(\roman*)]
\item $E_\alpha(t)$ is nonempty;
\item for every $s\in E_\alpha(t)$ and every prime $p$, one has
\[
v_p(P_{\alpha,t})\leq v_p(P_{\alpha,s}).
\]
\end{enumerate}
\end{proposition}

\begin{proof}
By definition, the branch-point $P_{\alpha, t}$ is not globally visible precisely when
$E_\alpha(t)$ is nonempty.  On the other hand, it is $p$-adically visible if
and only if
\[
v_p(P_{\alpha,t})\leq v_p(P_{\alpha,s})
\]
for every $s\in E_\alpha(t)$.  Requiring this for every prime $p$ gives the
result.
\end{proof}

This gives a useful sufficient condition for local detectability.

\section{Positive-weight homogeneous families}

In this section, we consider a specialized setting in which every locally visible point is globally visible.

\subsection{Density results for visibility}
\begin{definition}
Let $w_1,\dots,w_n$ be positive integers. The associated positive-weight
action of $\Gm$ on $\mathbb{R}^n$ is $\rho:\Gm\longrightarrow \GL_n$, given by $\rho(d)=\operatorname{diag}(d^{w_1},\dots,d^{w_n})$. We use its restriction to $d\in\R_{>0}$ when comparing the order of parameters on a branch.
\end{definition}

\begin{definition}
A visibility datum $\cY=\{(U_\alpha,\phi_\alpha)\}_{\alpha\in I}$ is positive-weight homogeneous if there are positive integers
$w_1,\dots,w_n$ such that, for every $\alpha$, every $t\in U_\alpha$, and
every $d\in \mathbb{R}_{>0}$ with $dt\in U_\alpha$, one has
\[
\phi_\alpha(dt)=\rho(d)\phi_\alpha(t),
\]
where $\rho(d)=\operatorname{diag}(d^{w_1},\dots,d^{w_n})$.
\end{definition}

\begin{proof}[Proof of Theorem \ref{thm:local-detectability}]
Assume that \(P_{\alpha,t}\) is not globally visible.  We show that
\(P_{\alpha,t}\) is not \(p\)-adically visible for some prime \(p\).  Since
\(P_{\alpha,t}\) is not globally visible, there exists
\(s\in U_\alpha^+\) with \(0<s<t\) such that
\[
        Q:=\phi_\alpha(s)\in \Z^n.
\]
Set $P:=\phi_\alpha(t)$ and $d:=\frac ts$. Then \(d>1\), and by positive-weight homogeneity,
\[
        P
        =
        \phi_\alpha(t)
        =
        \phi_\alpha(ds)
        =
        \rho(d)\phi_\alpha(s)
        =
        \rho(d)Q.
\]
Writing $P=(P_1,\dots,P_n)$ and $Q=(Q_1,\dots,Q_n)$,
we have $P_i=d^{w_i}Q_i$ for all $i$. The point \(Q\) is nonzero, since \(s>0\) and by Definition
\ref{def:visibility-datum} the only parameter mapping to \(0\) is \(0\). Setting
\[
        J:=\{i:Q_i\neq 0\},
\]
we have that \(J\neq\varnothing\).  For \(i\in J\), we have \(P_i\neq 0\), and since
\(P_i,Q_i\in\Z\), the relation
\[
        P_i=d^{w_i}Q_i
\]
implies
$
        d^{w_i}=\frac{P_i}{Q_i}\in\Q_{>0}.
$
Letting $h:=\gcd\{w_i:i\in J\}$, choose integers \(c_i\) such that
$
        \sum_{i\in J} c_iw_i=h.
$
We deduce that
\[
        d^h
        =
        \prod_{i\in J}(d^{w_i})^{c_i}
        \in\Q_{>0}.
\]

Set $D:=d^h$; since \(d>1\), we have \(D>1\).  Write \(D=M/N\) in lowest terms, with
\(M,N\in\Z_{\geq 1}\).  Since \(D>1\), we have \(M>N\), so \(M>1\).  Choose a
prime \(p\mid M\).  Then $v_p(D)>0$. For \(i\in J\), the integer \(h\) divides \(w_i\), and hence $d^{w_i}=D^{w_i/h}$. Therefore
\[
        v_p(P_i)
        =
        v_p(Q_i)+v_p(d^{w_i})
        =
        v_p(Q_i)+\frac{w_i}{h}v_p(D)
        >
        v_p(Q_i).
\]
For \(i\notin J\), one has \(Q_i=0\), and hence \(P_i=0\), so both
corresponding valuations are \(+\infty\).  For every $i\in J$, the difference \(v_p(P_i)-v_p(Q_i)\) is a positive integer, while the coordinates outside \(J\) have valuation $+\infty$. Thus, taking minima over all
coordinates,
\[
        v_p(P)\geq v_p(Q)+1>v_p(Q).
\]
Therefore \(Q\) is an earlier integral point on the same branch with strictly
smaller \(p\)-adic valuation.  Hence \(P_{\alpha,t}\) is not
\(p\)-adically visible.  This proves local detectability.
\end{proof}

The theorem applies to the weighted homogeneous parametrizations
\begin{equation}\label{eq:weighted-parametrization}\phi_\alpha(t)
=
(a_1(\alpha)t^{w_1},\dots,a_n(\alpha)t^{w_n}),
\end{equation}
where $w_i\in \Z_{>0}$ and $a_i:I\longrightarrow \Q_{>0}$, provided that the conditions of Definitions \ref{def:visibility-datum} hold for $\cY$. We note that positivity of the weights ensures that the parametrization is regular at $0$ and vanishes there.

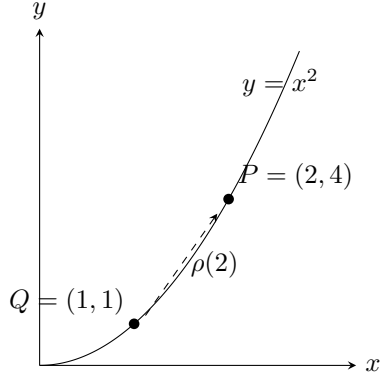
\begin{figure}[t]
\centering
\begin{tikzpicture}[x=1.25cm,y=.55cm,>=stealth]
  \draw[->] (0,0) -- (3.35,0) node[right] {$x$};
  \draw[->] (0,0) -- (0,8.1) node[above] {$y$};

  \draw[domain=0:2.75,samples=80,smooth,variable=\t]
        plot ({\t},{\t*\t});

  \fill (1,1) circle (2pt)
        node[above left] {$Q=(1,1)$};
  \fill (2,4) circle (2pt)
        node[above right] {$P=(2,4)$};

  \draw[->,dashed] (1.12,1.2) -- (1.88,3.65)
        node[midway,right] {$\rho(2)$};

  \node at (2.55,6.8) {$y=x^2$};
\end{tikzpicture}
\caption{Weighted scaling on the branch \(y=x^2\).  The dilation with
weights \((1,2)\) sends \(Q\) to \(P\): the first coordinate is
multiplied by \(2\), while the second is multiplied by \(2^2\).}
\label{fig:weighted-scaling}
\end{figure}
We now prove a general result for studying visibility densities for positive-weight homogeneous families. We consider the special case where $\cB(\cY)=\Z_{\geq 1}^n$. This gives us a way to pass from a finite set of local conditions to infinitely many
local conditions. For each prime \(p\), let \(U_p\subseteq \cB(\cY)\) be a local
condition. In our applications, $U_p$ will consist of $p$-adically visible points in $\cB(\cY)$. 
\begin{assumption}\label{ass:local-congruences} We assume that \(U_p\) is defined by congruences modulo a fixed
power of \(p\), that is, there exists an integer \(k_p\geq 1\) and a subset $R_p\subseteq (\mathbb Z/p^{k_p}\mathbb Z)^n$ such that
\[
U_p
=
\{a\in \mathbb Z_{\geq 1}^n : a\bmod p^{k_p}\in R_p\}.
\]
\end{assumption}
We set \[U_p(X):=U_p\cap[-X, X]^n.\] 
\begin{lemma}
\label{lem:single-local-density}
The density
\[
\mathfrak d_p
:=
\lim_{X\to\infty}
\frac{|U_p(X)|}{|\cB(\cY;X)|}
\]
exists and is given by $\mathfrak d_p=\frac{|R_p|}{p^{nk_p}}$.
\end{lemma}

\begin{proof}
With \(M=p^{k_p}\) and \(N=\lfloor X\rfloor\), one has \[|U_p(X)|=\frac{|R_p|}{M^n}N^n+O_M(N^{n-1})\] and \(|\cB(\cY;X)|=N^n\).
Hence, letting \(X\to\infty\), \[\mathfrak d_p=|R_p|/M^n=|R_p|/p^{nk_p}.\]
\end{proof}

\begin{assumption}\label{ass:euler-product}
    Assume that the product $\prod_p \mathfrak d_p$ converges.
\end{assumption}

For a finite set of primes $S$, set
\[
U_S:=\bigcap_{p\in S}U_p \quad \text{and}\quad
U:=\bigcap_p U_p.
\]
\begin{lemma}
\label{lem:finite-local-density}
The density \(\mathfrak d_{\cB(\cY)}(U_S)\) exists and is given by
\[
\mathfrak d_{\cB(\cY)}(U_S)
=
\prod_{p\in S}\frac{|R_p|}{p^{nk_p}}
=
\prod_{p\in S}\mathfrak d_p.
\]
\end{lemma}

\begin{proof}
This is an easy consequence of the Chinese remainder theorem.
\end{proof}

Let $U_p'$ be the complement of $U_p$ and set $U_p'(X):=U_p'\cap[-X, X]^n$. Given $Y>0$, let 
\[\mathcal{W}_Y:=\bigcup_{p>Y} U_p'=\left\{a\in \Z^n_{\geq 1}\mid a\notin U_p\quad \text{for some}\quad p>Y\right\},\]
and set $U_Y:=\bigcap_{p\leq Y} U_p$.
\begin{theorem}\label{thm:density}
With respect to notation above, assume that 
\begin{enumerate}
    \item for each prime $p$, there are $k_p\geq1$ and $R_p\subseteq(\Z/p^{k_p}\Z)^n$ such that
\[
U_p=\{a\in\Z_{\geq1}^n:a\bmod p^{k_p}\in R_p\}, 
\]
\item the infinite product $\prod_p\mathfrak d_p$ converges,
\item and that one has the following limit:\begin{equation}\label{eq:large-prime-tail}\lim_{Y\to\infty}\limsup_{X\to\infty}\frac{|\mathcal{W}_Y(X)|}{|\cB(\cY;X)|}=0.
\end{equation}
\end{enumerate}
Then $\mathfrak d_{\cB(\cY)}(U)=\prod_p\mathfrak{d}_p$.
\end{theorem}
\begin{proof}
For \(Y>0\), we have $U\subseteq U_Y$ and $U_Y\setminus U\subseteq \mathcal W_Y$. Hence
\[
|U_Y(X)|-|\mathcal W_Y(X)|
\leq |U(X)|\leq |U_Y(X)|.
\]
By Lemma~\ref{lem:finite-local-density},
\[
\lim_{X\to\infty}\frac{|U_Y(X)|}{|\cB(\cY;X)|}
=\prod_{p\leq Y}\mathfrak d_p.
\]
Thus
\[
\prod_{p\leq Y}\mathfrak d_p-
\limsup_{X\to\infty}\frac{|\mathcal W_Y(X)|}{|\cB(\cY;X)|}
\leq
\liminf_{X\to\infty}\frac{|U(X)|}{|\cB(\cY;X)|}
\]
and
\[
\limsup_{X\to\infty}\frac{|U(X)|}{|\cB(\cY;X)|}
\leq \prod_{p\leq Y}\mathfrak d_p.
\]
Letting \(Y\to\infty\), assumptions (ii) and (iii) show that the
\(\liminf\) and \(\limsup\) both equal \(\prod_p\mathfrak d_p\).
\end{proof}

\begin{proof}[Proof of Theorem \ref{thm:full-lattice-density}]
Fix a prime \(p\) and write \(P=\phi_\alpha(t)=(x_1,\dots,x_n)\). We first claim that \(P\) is not \(p\)-adically visible if and only if $p^{w_i}\mid x_i$ for every $i$. Indeed, if these divisibilities hold, then
\(\phi_\alpha(t/p)=(x_1/p^{w_1},\dots,x_n/p^{w_n})\) is an earlier integral
point whose \(p\)-adic minimum valuation is strictly smaller.

Conversely, suppose that \(Q=\phi_\alpha(s)\) is an earlier integral point
with \(v_p(Q)<v_p(P)\), and put \(d=t/s>1\). Then
\(x_i=d^{w_i}Q_i\), so \(d^{w_i}\in\Q\) for every \(i\). Since
\(\gcd(w_1,\dots,w_n)=1\), Bézout's identity gives \(d\in\Q\). Writing
\(e=v_p(d)\), we have \(v_p(x_i)=v_p(Q_i)+w_i e\). If \(e=0\), then
\(v_p(P)=v_p(Q)\), while if \(e<0\), choosing a coordinate realizing
\(v_p(Q)\) gives \(v_p(P)<v_p(Q)\). Both are impossible, so \(e\geq1\),
and hence \(p^{w_i}\mid x_i\) for every \(i\).

Thus the complement of the \(p\)-adically visible locus is given by the
simultaneous divisibility conditions \(p^{w_i}\mid x_i\). Its density is
\(p^{-W}\), where \(W=\sum_iw_i\), and therefore
\(\mathfrak d_p^{\operatorname{vis}}(\cY)=1-p^{-W}\). By Theorem
\ref{thm:local-detectability}, global visibility is the intersection of
these local conditions.

It remains only to check the large-prime tail. For
\(\Omega_X=\{1,\dots,\lfloor X\rfloor\}^n\), the union bound gives
\[
 \frac{\#\bigl(\Omega_X\cap\bigcup_{p>Y}U_p'\bigr)}{\#\Omega_X}
 \leq \sum_{p>Y}p^{-W}.
\]
Since \(W>1\), the right-hand side tends to \(0\) as \(Y\to\infty\).
Theorem \ref{thm:density} therefore applies and yields
\[
 \mathfrak d^{\operatorname{vis}}(\cY)
 =\prod_p(1-p^{-W})=\frac1{\zeta(W)}.
\]
\end{proof}

\subsection{Some examples}
We now return to the affine visibility problem and illustrate the
preceding results with some basic weighted homogeneous families.
The first two examples recover classical visibility and weighted
lines of sight, while the last one explains the role of the
coprimality assumption on the weights in Theorem~1.3.

\noindent\textbf{(1) Classical visibility.}
Let \(n\geq 2\). Consider the index set
\[
I=
\left\{
\alpha=(m_2,\dots,m_n):
m_i\in \Q_{>0}\text{ for }2\leq i\leq n
\right\}.
\]
For \(\alpha\in I\), define
\[
\phi_\alpha(t)=(t,m_2t,\dots,m_nt).
\]
Set $Y_\alpha=\phi_\alpha(\mathbb R)$ and $Y_\alpha^+=\phi_\alpha(\mathbb R_{>0})$. Thus, \(Y_\alpha\) is a line passing through the origin and \(Y_\alpha^+\)
is its positive ray. Moreover, $\bigcup_{\alpha\in I}Y_\alpha^+$ contains \(\Q_{>0}^n\) and is dense in \(\mathbb R_{>0}^n\). This is a special case of the weighted homogeneous family
\eqref{eq:weighted-parametrization} with $w_1=\cdots=w_n=1$. Hence, visibility is locally detectable by Theorem
\ref{thm:local-detectability}.

Let $a=(a_1,\dots,a_n)\in \Z_{\geq 1}^n$. Then \(a\) lies on the branch with $m_i=a_i/a_1$ for $2\leq i\leq n$. Let $g=\gcd(a_1,\dots,a_n)$. Then $a=g(a_1/g,\dots,a_n/g)$ and $(a_1/g,\dots,a_n/g)$ is the primitive integral point on the same ray. Therefore, $a$ is visible if and only if $\gcd(a_1,\dots,a_n)=1$.

For a prime \(p\), the point \(a\) fails to be \(p\)-adically visible if and
only if $p\mid a_i$ for every $1\leq i\leq n$. Thus the \(p\)-adic visible locus is
\[
U_p=
\{a\in \Z_{\geq 1}^n:\text{there exists }i\text{ such that }p\nmid a_i\}.
\]
Its complement is
\[
U_p'=
\{a\in \Z_{\geq 1}^n:p\mid a_i\text{ for every }i\}.
\] We verify the conditions of Theorem
\ref{thm:density}. Let $N=\lfloor X\rfloor$ and
\[
\Omega_X=\{a=(a_1,\dots,a_n)\in \Z_{\geq 1}^n:
1\leq a_i\leq N\text{ for every }i\}.
\]
The number of points in \(\Omega_X\) that fail the -adic visibility of \(p\) is $\lfloor N/p\rfloor^n$.
Therefore,
\[
\mathfrak d_p
=
\lim_{X\to\infty}
\frac{|U_p\cap \Omega_X|}{|\Omega_X|}
=
1-1/p^n.
\]
For the large-prime tail, we have
\[
\left|\bigcup_{p>Y}U_p'\cap \Omega_X\right|
\leq
\sum_{p>Y}\lfloor N/p\rfloor^n.
\]
Dividing by \(N^n\) and taking \(\limsup\) as \(X\to\infty\), we obtain
\[
\limsup_{X\to\infty}
\frac{\left|\bigcup_{p>Y}U_p'\cap \Omega_X\right|}
{|\Omega_X|}
\leq
\sum_{p>Y}1/p^n.
\]
Since \(n\geq 2\), this tends to \(0\) as \(Y\to\infty\). Hence Theorem
\ref{thm:density} applies and gives
\[
\mathfrak d^{\operatorname{vis}}
=
\prod_p(1-1/p^n)
=
1/\zeta(n).
\]
\medskip

\noindent\textbf{(2) Weighted lines of sight.}
\textbf{Weighted lines of sight.}
We generalize the previous example. Fix positive integers $w_1,\dots,w_n\in \Z_{>0}$ such that $w_1=1$ and let $I=\Q_{>0}^{n-1}$. For $\alpha=(a_2,\dots,a_n)\in I$, define
\[
\phi_\alpha(t)
:=
(t^{w_1},a_2 t^{w_2}, \dots,a_nt^{w_n})=(t,a_2 t^{w_2}, \dots,a_nt^{w_n}).
\]
We assume \(a_i>0\) for every \(i\), so the image lies in the positive
orthant for \(t>0\). This is homogeneous for the positive-weight action
\[
d\cdot (x_1,\dots,x_n)
=
(d^{w_1}x_1,\dots,d^{w_n}x_n).
\]
Hence, visibility is locally detectable by Theorem
\ref{thm:local-detectability}. Since it is assumed that $w_1=1$, we have $\cB(\cY)=\Z_{\geq 1}^n$.

By Theorem \ref{thm:full-lattice-density}, for $x=(x_1,\dots,x_n)\in \Z_{\geq 1}^n$, the point $x$ is visible if and only if there is no prime $p$ such that $p^{w_i}\mid x_i$ for every $1\leq i\leq n$. The \(p\)-adic visible locus is
\[
U_p=
\{x\in \Z_{\geq 1}^n:
\text{there exists }i\text{ such that }p^{w_i}\nmid x_i\}.
\]
Its complement is
\[
U_p'=
\{x\in \Z_{\geq 1}^n:
p^{w_i}\mid x_i\text{ for every }1\leq i\leq n\}.
\]
Let $W=w_1+\cdots+w_n$ and $N=\lfloor X\rfloor$. For
\[
\Omega_X=
\{x=(x_1,\dots,x_n)\in \Z_{\geq 1}^n:
1\leq x_i\leq N\text{ for every }i\},
\]
the number of points which fail the \(p\)-adic visibility condition is
\[
\prod_{i=1}^n \lfloor N/p^{w_i}\rfloor.
\]
Therefore
\[
\mathfrak d_p
=
\lim_{X\to\infty}
\frac{|U_p\cap \Omega_X|}{|\Omega_X|}
=
1-1/p^W.
\]
The large-prime tail is bounded by
\[
\left|\bigcup_{p>Y}U_p'\cap \Omega_X\right|
\leq
\sum_{p>Y}\prod_{i=1}^n \lfloor N/p^{w_i}\rfloor.
\]
Dividing by \(N^n\) and taking \(\limsup\) as \(X\to\infty\), we get
\[
\limsup_{X\to\infty}
\frac{\left|\bigcup_{p>Y}U_p'\cap \Omega_X\right|}
{|\Omega_X|}
\leq
\sum_{p>Y}1/p^W.
\]
Since \(W>1\), this tends to \(0\) as \(Y\to\infty\). Hence Theorem
\ref{thm:density} applies and gives
\[
\mathfrak d^{\operatorname{vis}}
=
\prod_p(1-1/p^W)
=
1/\zeta(W).
\]
This example specializes to classical visibility when $w_1=\cdots=w_n=1$. It specializes to generalized lines of sight in the plane when $n=2, w_1=1$ and $w_2=b$.
\medskip

\noindent\textbf{(3) A non-example.} 
Consider the family $\phi_m(t)=(t^2,mt^2)$ where $m\in\Q_{>0}$. This is a positive-weight homogeneous family with weights $w_1=w_2=2$.
Since $\gcd(w_1,w_2)=2$, hypothesis $\gcd(w_1,w_2)=1$ of Theorem \ref{thm:full-lattice-density} is not satisfied.
It is easy to see that $\cB(\cY)=\Z_{\geq 1}^2$.

However, the following naive local condition suggested by the proof of Theorem \ref{thm:full-lattice-density}, $p^2\mid x$ and $p^2\mid y$, is not the correct condition. For example, the point $P=(p,p)$
lies on the branch \(m=1\) with parameter \(t=\sqrt p\). On the same branch,
the point $Q=(1,1)$
occurs earlier, with parameter \(t=1\). Hence \(P\) is not globally visible.
Moreover, $v_p(P)=1$ and $v_p(Q)=0$ so \(P\) is not \(p\)-adically visible. Hence, the above condition does not detect this failure. In this example, $(x,y)$ is visible if and only if $\gcd(x,y)=1$.
Consequently, $\mathfrak d^{\operatorname{vis}}(\cY)
        =
        \frac1{\zeta(2)}$ and not $\frac1{\zeta(4)}$.
Thus, condition \(\gcd(w_1,\dots,w_n)=1\) cannot be omitted in general from the density
formula with exponent \(W=w_1+\cdots+w_n\) in Theorem \ref{thm:full-lattice-density}.

\subsection{Weighted homogeneous families and weighted projective stacks}
\label{subsec:weighted-projective}
We now give a geometric perspective on visibility questions for weighted homogeneous families and explain the relationship between visibility and weighted projective stacks.
We work over \(\Q\).  For background on weighted projective stacks and
their arithmetic, see
\cite{DardaWeightedProjectiveStacks,ChanLoughranRome}.

Let \(\mathbf w=(w_1,\ldots,w_n)\in\Z_{>0}^n\).  The
\emph{weighted projective stack} of weights \(\mathbf w\) is the quotient
stack
\[
 \mathscr P(\mathbf w)
 :=
 \bigl[(\A^n\setminus\{0\})/\Gm\bigr],
\]
where $\lambda\cdot(x_1,\ldots,x_n)
 =
 (\lambda^{w_1}x_1,\ldots,\lambda^{w_n}x_n)$. The coarse moduli space of \(\mathscr P(\mathbf w)\) is the usual
weighted projective space $\operatorname{Proj}\Q[X_1,\ldots,X_n]$ where $\deg X_i=w_i$. For a geometric point
\(\mathbf x=(x_1,\ldots,x_n)\in\A^n\setminus\{0\}\), the
\emph{stabilizer} of \(\mathbf x\) under the weighted \(\Gm\)-action is
the subgroup
\[
\operatorname{Stab}_{\Gm}(\mathbf x)
=
\{\lambda\in\Gm:\lambda^{w_i}=1
\text{ for every }i\in J(\mathbf x)\}.
\]
where $J(\mathbf x):=\{i:x_i\neq0\}$. 
\par Let \(k\) be an algebraically closed field of characteristic \(0\), and let \(\mathbf x=(x_1,\ldots,x_n)\in k^n\setminus\{0\}\). Set
\[ h(\mathbf x):=\gcd\{w_i:i\in J(\mathbf x)\}.
\]
Then it is easy to see that the stabilizer of the point represented by \(\mathbf x\) in
\(\mathscr P(\mathbf w)\) is  \(\mu_{h(\mathbf x)}\). Here, $\mu_n$ denotes the group of $n$-th roots of unity. In particular, the
generic stabilizer is
\(\mu_{\gcd(w_1,\ldots,w_n)}\), and the weighted action is effective if
and only if $\gcd(w_1,\ldots,w_n)=1$. Thus, the gcd that occurs naturally in the proof of
Theorem~\ref{thm:local-detectability} has a stack-theoretic meaning: it
is the order of the stabilizer of the corresponding weighted-projective
point.
\par Now, let
\[
 \phi_\alpha(t)
 =
 \bigl(a_1(\alpha)t^{w_1},\ldots,
       a_n(\alpha)t^{w_n}\bigr)
\]
be a weighted homogeneous branch, and let
\(\pi:\A^n\setminus\{0\}\to\mathscr P(\mathbf w)\) be the quotient map.
Then the restriction of \(\pi\circ\phi_\alpha\) to \(\Gm\) is constant:
every nonzero point of the branch represents the same point
\[
 [a_1(\alpha):\cdots:a_n(\alpha)]_{\mathbf w}
 \in\mathscr P(\mathbf w).
\]
For instance, when \(w_1=\cdots=w_n=1\), a ray through the origin maps to a single point of
ordinary projective space, and its first integral point is the primitive
integral representative of that projective point.  The corresponding
notion for general weights is weighted primitivity.

\begin{definition}
For \(\mathbf x=(x_1,\ldots,x_n)\in\Z^n\setminus\{0\}\), define its
\emph{weighted gcd} by
\[
 \operatorname{wgcd}_{\mathbf w}(\mathbf x)
 :=
 \prod_p
 p^{\,\min_i\lfloor v_p(x_i)/w_i\rfloor}.
\]
We say that \(\mathbf x\) is \emph{weighted primitive} if
\(\operatorname{wgcd}_{\mathbf w}(\mathbf x)=1\).
\end{definition}

The relevance of this definition to the weighted projective stack is
that it measures the extent to which an integral representative can be
scaled backwards along its weighted \(\Gm\)-orbit while remaining
integral. Indeed, if \(m^{w_i}\mid x_i\) for every \(i\), then
\[
 \mathbf y=
 \left(\frac{x_1}{m^{w_1}},\ldots,
       \frac{x_n}{m^{w_n}}\right)\in\Z^n
\]
and \(\mathbf x=m\cdot\mathbf y\) under the weighted action, so
\(\mathbf x\) and \(\mathbf y\) determine the same point of
\(\mathscr P(\mathbf w)\).  Thus,
\(\operatorname{wgcd}_{\mathbf w}(\mathbf x)\) is the largest positive
integer that can be removed from \(\mathbf x\) in this way, and the
weighted primitivity means that the weighted nontrivial integral scaling cannot
 be removed.  This is a direct analog to choosing a primitive
integral representative of a rational point of ordinary projective
space.  The weighted gcd and its valuation formula are discussed in
\cite[Section~2, Proposition~1]{BeshajGutierrezShaska}; see also
\cite{DardaWeightedProjectiveStacks} for its relation to heights on
weighted projective stacks.

The following proposition identifies this notion exactly with
visibility.

\begin{proposition}
\label{prop:visibility-weighted-primitive}
Assume the hypotheses of Theorem~\ref{thm:full-lattice-density}. Let
\[
 P=(x_1,\ldots,x_n)=\phi_\alpha(t)
 \in\Z_{\geq1}^n.
\]
For every prime \(p\), the point \(P\) is \(p\)-adically visible precisely
when its weighted gcd is not divisible by \(p\). Consequently, \(P\) is
globally visible precisely when its weighted gcd is equal to \(1\).
\end{proposition}

\begin{proof}
By the definition of the weighted gcd,
\(p\mid\operatorname{wgcd}_{\mathbf w}(P)\) if and only if
\(p^{w_i}\mid x_i\) for every \(i\).  In the proof of
Theorem~\ref{thm:full-lattice-density}, we showed that this is
equivalent to \(P\) failing to be \(p\)-adically visible.  This proves
the first assertion.  The second now follows from
Theorem~\ref{thm:local-detectability}: \(P\) is globally visible if and
only if it is \(p\)-adically visible for every prime \(p\), equivalently
if and only if \(\operatorname{wgcd}_{\mathbf w}(P)=1\).
\end{proof}

Thus Theorem~\ref{thm:local-detectability}, in this setting, can be
viewed as the local--global principle
\[
 \operatorname{wgcd}_{\mathbf w}(P)=1
 \quad\Longleftrightarrow\quad
 p\nmid\operatorname{wgcd}_{\mathbf w}(P)
 \quad\text{for every prime }p.
\]
The arithmetic notion on the right is exactly the
\(p\)-adic visibility condition defined earlier. The density formula also acquires a simple weighted-projective
interpretation.

\begin{corollary}\label{cor:weighted-primitive-density}
Let \(W=w_1+\cdots+w_n\).  Under the hypotheses of
Theorem~\ref{thm:full-lattice-density}, the density of integral
representatives which are weighted primitive at \(p\) is
\(1-p^{-W}\), and the density of weighted-primitive integral
representatives is
\[
 \prod_p(1-p^{-W})=\frac{1}{\zeta(W)}.
\]
\end{corollary}

\begin{proof}
Failure of weighted primitivity at \(p\) means precisely
\(p^{w_i}\mid x_i\) for all \(i\).  Hence its local density is
\(p^{-W}\).  Moreover, for the large-prime tail the union bound gives
an upper bound \(\sum_{p>Y}p^{-W}\), which tends to zero since
\(W>1\). Applying Theorem~\ref{thm:density} gives the asserted Euler
product.
\end{proof}

Notice that \(W\), rather than the ordinary dimension \(n\), is the
relevant exponent because weighted multiplication by \(p\) has lattice
index \(p^W\).  For \(\mathbf w=(1,\ldots,1)\), this reduces to the
usual density \(1/\zeta(n)\) of primitive lattice points.

\section{Sparse families}\label{sec:sparse}

In the preceding section, we considered positive-weight homogeneous families whose branch locus is the whole positive lattice.  In that setting, the density problem can be treated by imposing congruence conditions directly on the ambient lattice.  We now turn to the opposite situation. Recall that we call a visibility datum \(\cY\) \emph{sparse} when the Zariski closure of \(\cB(\cY)\) has a dimension strictly smaller than the dimension of the ambient affine space. We first review the form of Davenport's lattice-point estimate, which we shall use.

\begin{lemma}[Davenport's lemma]\label{lem:davenport}
Let \(\mathcal R\subseteq\R^n\) be a bounded region.  Suppose that there is an integer \(h\geq1\) such that every intersection of \(\mathcal R\) with a line parallel to a coordinate axis is a union of at most \(h\) intervals, and that the same property holds for every coordinate projection of \(\mathcal R\).  For \(1\leq j\leq n-1\), let \(V_j(\mathcal R)\) denote the sum of the \(j\)-dimensional volumes of the orthogonal projections of \(\mathcal R\) onto the \(j\)-dimensional coordinate subspaces of \(\R^n\).  Then
\[
 \#(\mathcal R\cap\Z^n)
 =
 \operatorname{vol}_n(\mathcal R)
 +
 O_{n,h}\left(1+\sum_{j=1}^{n-1}V_j(\mathcal R)\right).
\]
\end{lemma}

\begin{proof}
This is Davenport's lattice-point estimate; see \cite{DavenportLipschitz}.  The dependence of the implied constant is only on the dimension and on the uniform bound for the number of intervals occurring in coordinate sections and coordinate projections.
\end{proof}
\subsection{Sparse radial graph families}We introduce the general class of sparse families to which Davenport's lemma will be applied. Fix integers \(r\geq2\) and \(s\geq1\).  For \(1\leq j\leq s\), let \(F_j\in\Z[X_1,\dots,X_r]\) be a nonzero homogeneous polynomial of degree \(\delta_j\geq1\), and assume that all coefficients of \(F_j\)'s are nonnegative.  Consider the polynomial graph map
\begin{equation}\label{Psi defn}
    \Psi(x_1,\dots,x_r)
 =
 \bigl(x_1,\dots,x_r,F_1(x_1,\dots,x_r),\dots,F_s(x_1,\dots,x_r)\bigr).
\end{equation}
For \(\boldsymbol a=(a_2,\dots,a_r)\in\Q_{>0}^{r-1}\), define an ordered branch by setting \[\phi_{\boldsymbol a}(t):=\Psi(t,a_2t,\dots,a_rt)\] for \(t>0\), and let \(\cY_\Psi\) be the collection of all such branches. We refer such a family \(\cY_\Psi\) as a \emph{sparse radial graph family}.

Homogeneity shows that the last \(s\) coordinates of \(\phi_{\boldsymbol a}(t)\) have the form \(F_j(1,a_2,\dots,a_r)t^{\delta_j}\).  Thus this is a positive-weight homogeneous family, with weight one on each of the first \(r\) coordinates and weight \(\delta_j\) on the coordinate determined by \(F_j\).  

We first describe the branch locus.  Suppose that \(\phi_{\boldsymbol a}(t)\) is integral.  Its first coordinate is \(t\), while its next \(r-1\) coordinates are \(a_2t,\dots,a_rt\). Thus the first \(r\) coordinates of $\Psi$ form a vector \(x=(x_1,\dots,x_r)\in\Z_{\geq1}^r\).  Since each \(F_j\) has integral coefficients, all remaining coordinates of $\Psi$ are then equal to the integers \(F_j(x)\).  Conversely, if \(x\in\Z_{\geq1}^r\), taking \(t=x_1\) and \(a_i=x_i/x_1\) produces the branch point \(\Psi(x)\).  We therefore have
\[
 \cB(\cY_\Psi)
 =
 \{\Psi(x):x\in\Z_{\geq1}^r\}.
\]
Moreover, the first \(r\) coordinates determine the branch uniquely, since \(a_i=x_i/x_1\) for \(2\leq i\leq r\).  Hence two distinct branches do not intersect away from the origin.

The branch locus lies on the graph cut out by the equations \(Y_{r+j}=F_j(Y_1,\dots,Y_r)\) for \(1\leq j\leq s\).  This graph is isomorphic to \(\A^r\), and hence has dimension \(r\), whereas the ambient affine space has dimension \(r+s\).  Thus every family \(\cY_\Psi\) constructed in this way is sparse.

We shall use the standard max-height on the ambient affine space.  For \(X\geq1\), let \(\mathcal R_X\) denote the set of \(u=(u_1,\dots,u_r)\in[1,\infty)^r\) for which \(H(\Psi(u))\leq X\), and put \(V(X):=\operatorname{vol}_r(\mathcal R_X)\).  Since the first \(r\) coordinates of \(\Psi(u)\) are \(u_1,\dots,u_r\) themselves, membership in \(\mathcal R_X\) implies \(u_i\leq X\) for every \(i\).  In particular, \(\mathcal R_X\subseteq[1,X]^r\), so \(\mathcal R_X\) is bounded.  This verifies the boundedness hypothesis in Davenport's lemma.
\par For a point \(P=(P_1,\dots,P_n)\in X(\Z)\), recall from \eqref{defn of H} that
\[
        H(P):=\max_{1\leq i\leq n}|P_i|.
\]
A subset \(\mathcal R\subseteq\R^n\) is called
\emph{semialgebraic} if it can be obtained from finitely many sets of
the form $\{x\in\R^n:f(x)=0\}$ and $\{x\in\R^n:g(x)>0\}$, where \(f,g\in\R[X_1,\dots,X_n]\), by taking finitely many unions,
intersections, and complements.

\begin{proposition}
\label{prop:Davenport-admissible}
Let \(r\geq 2\) and \(s\geq 1\), and let
\(F_1,\dots,F_s\in\Z[X_1,\dots,X_r]\) be homogeneous polynomials whose
coefficients are all nonnegative. Consider the associated map
\[
        \Psi(u_1,\dots,u_r)
        =
        \bigl(
        u_1,\dots,u_r,
        F_1(u_1,\dots,u_r),\dots,
        F_s(u_1,\dots,u_r)
        \bigr),
\]
as in \eqref{Psi defn}. For \(X\geq 1\), let
\[
        \mathcal R_X
        :=
        \left\{
        u\in[1,\infty)^r:
        H(\Psi(u))\leq X
        \right\}.
\]
Then the following assertions hold.
\begin{enumerate}[label=\textnormal{(\roman*)}]
    \item The region \(\mathcal R_X\) is bounded and semialgebraic.
    \item Every intersection of \(\mathcal R_X\) with a line parallel
    to a coordinate axis is either empty or a single interval. Further, every coordinate projection of \(\mathcal R_X\) has the same
    property: its intersection with a line parallel to one of its
    coordinate axes is either empty or a single interval.
\end{enumerate}
Consequently, the hypotheses of Davenport's lemma are
    satisfied for \(\mathcal R_X\) with \(h=1\), uniformly in \(X\).
\end{proposition}
\begin{proof}
Let $\mathbf{u}=(u_1, \cdots, u_r)$. Since the coordinates of \(\Psi(\mathbf u)\) are nonnegative on
\(\mathbb R_{\geq 0}^r\), we have
\[
R_X
=
\left\{
\mathbf u\in[1,X]^r:
F_j(\mathbf u)\leq X
\text{ for }1\leq j\leq s
\right\}.
\]
This implies that \(\mathcal R_X\) is described
by finitely many polynomial inequalities.
Thus, \(R_X\) is bounded and semialgebraic.

 Every monomial that occurs in \(F_j\) is
nondecreasing in each variable in \(\R_{\geq0}^r\).  
Consequently, if \(\mathbf u\in R_X\) and
\[
1\leq v_i\leq u_i
\qquad (1\leq i\leq r),
\]
then \(\mathbf v\in R_X\). 
Hence \(\mathcal R_X\) is \emph{coordinatewise downward closed} in $[0,\infty)^r$.

\par Fix an index \(i\), and fix values of all
coordinates other than the \(i\)-th coordinate.  Consider the line
parallel to the \(i\)-th coordinate axis obtained by allowing only the
\(i\)-th coordinate to vary.  Suppose that two points on this line,
with $x_i=a$ and $y_i=b$, belong to
\(\mathcal R_X\), where \(1\leq a<b\). If \(c\) satisfies
\(a\leq c\leq b\), then the point obtained by replacing the
\(i\)-th coordinate \(b\) by \(c\) is coordinatewise no larger than
the point with \(i\)-th coordinate \(b\).  By the downward-closed
property, this intermediate point also belongs to \(\mathcal R_X\). It follows that the intersection of \(\mathcal R_X\) with any line
parallel to a coordinate axis is convex as a subset of that line.
Since a convex subset of a line is an interval, every such
intersection is either empty or a single interval.

It remains to verify the corresponding assertion for coordinate
projections.  Let \(I\subseteq\{1,\dots,r\}\), and let
\(\pi_I:\R^r\to\R^I\) be the corresponding coordinate projection.  We
claim that \(\pi_I(\mathcal R_X)\) is again coordinatewise downward
closed. Indeed, let \(y=(y_i)_{i\in I}\in\pi_I(\mathcal R_X)\). This implies that there exists \(u\in\mathcal R_X\) whose coordinates
indexed by \(I\) are \(y_i\).  Suppose now that
\(z=(z_i)_{i\in I}\) satisfies \(1\leq z_i\leq y_i\) for every
\(i\in I\).  Define \(v\in[1,\infty)^r\) by replacing, for each
\(i\in I\), the coordinate \(u_i=y_i\) by \(z_i\) and leaving all
coordinates outside \(I\) unchanged. Then \(v\) is coordinate wise \(\le u\).  Since \(u\in\mathcal R_X\), the downward-closed
property gives \(v\in\mathcal R_X\).  Its projection is \(z\), and
therefore \(z\in\pi_I(\mathcal R_X)\).

Thus, every coordinate projection of \(\mathcal R_X\) is itself
coordinate wise downward closed.  Repeating the preceding line-section
argument inside the projected coordinate space shows that the
intersection of every coordinate projection with a line parallel to
one of its coordinate axes is either empty or a single interval.
Therefore, Davenport's lemma applies to \(R_X\) with \(h=1\), uniformly in
\(X\).
\end{proof}
For a subset \(\Omega\subseteq\cB(\cY)\), we shall henceforth write
\[
        \Omega_X
        :=
        \{P\in\Omega:H(P)\leq X\}.
\]
Thus, whenever the limit exists, its relative density in the branch
locus is
\[
        \mathfrak d_{\cB(\cY)}(\Omega)
        :=
        \lim_{X\to\infty}
        \frac{\#\Omega_X}{\#\cB(\cY;X)}.
\]
Define
\[
        D(X)
        :=
        \sup_{u=(u_1,\dots,u_r)\in\mathcal R_X}
        \min_{1\leq i\leq r}u_i.
\]
Thus \(D(X)\) measures the largest possible size of the smallest
coordinate of a point in \(\mathcal R_X\).  In particular, if
\(x\in\mathcal R_X\cap\Z^r\) and an integer \(d\) divides every
coordinate of \(x\), then \(d\leq D(X)\). We can now state the general result.

\begin{theorem}
\label{thm:sparse-Davenport}
Let \(\cY_\Psi\) be the sparse radial graph family associated to
homogeneous polynomials \(F_1,\dots,F_s\) with nonnegative integral
coefficients as above. Assume that \(V(X)\to\infty\) and
\(D(X)\to\infty\), and suppose that
\begin{equation}\label{eq:sparse-error-hypothesis}
 D(X)
 +V_1(\mathcal R_X)\log\!\bigl(2D(X)\bigr)
 +\sum_{j=2}^{r-1}V_j(\mathcal R_X)
 =
 o\bigl(V(X)\bigr),
\end{equation}
where the final sum is omitted when \(r=2\).

Then a branch point \(\Psi(x_1,\dots,x_r)\) is globally visible if and
only if
\(\gcd(x_1,\dots,x_r)=1\). For every prime \(p\), it is
\(p\)-adically visible if and only if \(p\) does not divide all of
\(x_1,\dots,x_r\). In particular, visibility is locally detectable.

Moreover,
\begin{equation}\label{eq:sparse-counting-asymptotics}
 \#\cB(\cY_\Psi;X)
 =
 V(X)+o(V(X))
 \quad\text{and}\quad
 \#\cB^{\operatorname{vis}}(\cY_\Psi;X)
 =
 \frac{V(X)}{\zeta(r)}+o(V(X)).
\end{equation}
For every prime \(p\),
\[
 \mathfrak d_p^{\operatorname{vis}}(\cY_\Psi)
 =
 1-\frac1{p^r},
\]
while
\[
 \mathfrak d^{\operatorname{vis}}(\cY_\Psi)
 =
 \frac1{\zeta(r)}.
\]
Consequently,
\[
 \mathfrak d^{\operatorname{vis}}(\cY_\Psi)
 =
 \prod_p
 \mathfrak d_p^{\operatorname{vis}}(\cY_\Psi).
\]
\end{theorem}

\begin{proof}
We first record the visibility criteria.  As observed above,
\(\cY_\Psi\) is a positive-weight homogeneous visibility datum, with
weights $1,\dots,1,\delta_1,\dots,\delta_s$.
Therefore visibility is locally detectable by
Remark~\ref{remark1.2}.

It remains useful for the counting argument to describe global and
\(p\)-adic visibility explicitly.  Let
\(P=\Psi(x_1,\dots,x_r)\).  If
\(g=\gcd(x_1,\dots,x_r)>1\), then
\(\Psi(x_1/g,\dots,x_r/g)\) is an integral point on the same branch
and occurs at the earlier parameter \(t/g\).  Hence \(P\) is not
globally visible.  Conversely, if an earlier integral point occurs on
the same branch, then its first \(r\) coordinates are
\(\lambda x_1,\dots,\lambda x_r\) for some \(0<\lambda<1\).
Writing \(\lambda=a/b\) in lowest terms, integrality implies
\(b\mid x_i\) for every \(i\), and \(b>1\).  Thus \(P\) is globally
visible precisely when
\(\gcd(x_1,\dots,x_r)=1\).

Fix now a prime \(p\), and put
\(m=\min_i v_p(x_i)\).  Since \(F_j\) is homogeneous of degree
\(\delta_j\) and has integral coefficients,
\(v_p(F_j(x))\geq\delta_jm\geq m\).  One of the first \(r\)
coordinates has valuation exactly \(m\), and hence
\[
        v_p\bigl(\Psi(x)\bigr)=m.
\]
If \(p\mid x_i\) for every \(i\), division of all the \(x_i\) by
\(p\) gives an earlier integral point on the same branch whose
\(p\)-adic minimum valuation is \(m-1\).  If \(p\) does not divide all
of the \(x_i\), then \(m=0\), and no integral point can have smaller
\(p\)-adic minimum valuation.  Thus \(P\) is \(p\)-adically visible
exactly when \(p\nmid\gcd(x_1,\dots,x_r)\).

We now turn to the counting argument.  By the description of the
branch locus, the map \(x\mapsto\Psi(x)\) gives a bijection
\[
        \mathcal R_X\cap\Z^r
        \longrightarrow
        \cB(\cY_\Psi;X).
\]
By Proposition~\ref{prop:Davenport-admissible}, the region
\(\mathcal R_X\) satisfies the hypotheses of Davenport's lemma with
\(h=1\).  Hence
\begin{equation}\label{eq:sparse-total-count}
 \#\cB(\cY_\Psi;X)
 =
 V(X)
 +
 O_r\left(
 1+\sum_{j=1}^{r-1}V_j(\mathcal R_X)
 \right).
\end{equation}
Condition~\eqref{eq:sparse-error-hypothesis} implies that
\(V_j(\mathcal R_X)=o(V(X))\) for every \(1\leq j\leq r-1\).
For \(j\geq2\) this is immediate.  For \(j=1\), it follows from
\(D(X)\to\infty\), since
\(\log(2D(X))\to\infty\).  As \(V(X)\to\infty\), the constant term is
also \(o(V(X))\).  Therefore
\[
        \#\cB(\cY_\Psi;X)
        =
        V(X)+o(V(X)),
\]
which proves the first assertion of
\eqref{eq:sparse-counting-asymptotics}.

For \(d\geq1\), let \(N_d(X)\) be the number of
\(x=(x_1,\dots,x_r)\in\mathcal R_X\cap\Z^r\) such that
\(d\mid x_i\) for every \(i\).  Writing \(x=dy\) gives a bijection
between these points and the integral points of
\(d^{-1}\mathcal R_X\).  Proposition~\ref{prop:Davenport-admissible}
also applies to every such dilated region with the same value \(h=1\).
Since dilation by \(d^{-1}\) multiplies \(r\)-dimensional volume by
\(d^{-r}\) and \(j\)-dimensional coordinate-projection volumes by
\(d^{-j}\), Davenport's lemma gives
\begin{equation}\label{eq:sparse-divisible-count}
 N_d(X)
 =
 \frac{V(X)}{d^r}
 +
 O_r\left(
 1+\sum_{j=1}^{r-1}
 \frac{V_j(\mathcal R_X)}{d^j}
 \right),
\end{equation}
uniformly in \(d\) and \(X\).

We first deduce the local densities.  For a fixed prime \(p\), the
points which fail \(p\)-adic visibility are precisely those counted by
\(N_p(X)\).  Since \(p\) is fixed, condition
\eqref{eq:sparse-error-hypothesis} and
\eqref{eq:sparse-divisible-count} give
\[
        N_p(X)
        =
        \frac{V(X)}{p^r}+o(V(X)).
\]
Together with \eqref{eq:sparse-total-count}, this shows that the
proportion of points which fail \(p\)-adic visibility tends to
\(p^{-r}\).  Therefore
\[
        \mathfrak d_p^{\operatorname{vis}}(\cY_\Psi)
        =
        1-p^{-r}.
\]

For the global count, the visibility criterion and M\"obius inversion
give
\begin{equation}\label{eq:sparse-mobius}
 \#\cB^{\operatorname{vis}}(\cY_\Psi;X)
 =
 \sum_{d\geq1}\mu(d)N_d(X).
\end{equation}
If \(N_d(X)\neq0\), then some
\(x\in\mathcal R_X\cap\Z^r\) has all of its coordinates divisible by
\(d\).  Hence \(d\leq\min_i x_i\leq D(X)\).  Thus the sum in
\eqref{eq:sparse-mobius} may be restricted to \(d\leq D(X)\).

Substituting \eqref{eq:sparse-divisible-count} into
\eqref{eq:sparse-mobius}, the contribution of the main terms is
\[
 V(X)\sum_{d\leq D(X)}\frac{\mu(d)}{d^r}
 =
 \frac{V(X)}{\zeta(r)}
 +
 O_r\left(V(X)D(X)^{1-r}\right),
\]
because \(r\geq2\) and
\(\sum_{d\geq1}\mu(d)d^{-r}=1/\zeta(r)\).

It remains to sum the Davenport errors.  The constant term contributes
\(O(D(X))\).  The \(j=1\) term contributes
\[
 O\left(
 V_1(\mathcal R_X)\log(2D(X))
 \right),
\]
while, for every \(j\geq2\), the convergence of
\(\sum_{d\geq1}d^{-j}\) gives a contribution
\(O(V_j(\mathcal R_X))\).  We therefore obtain
\begin{equation}\label{eq:sparse-visible-count}
 \#\cB^{\operatorname{vis}}(\cY_\Psi;X)
 =
 \frac{V(X)}{\zeta(r)}
 +
 O_r\left(
 V(X)D(X)^{1-r}
 +D(X)
 +V_1(\mathcal R_X)\log(2D(X))
 +\sum_{j=2}^{r-1}V_j(\mathcal R_X)
 \right).
\end{equation}
Since \(D(X)\to\infty\) and \(r\geq2\), the first error term on the
right is \(o(V(X))\).  All the remaining terms are
\(o(V(X))\) by
\eqref{eq:sparse-error-hypothesis}.  This proves the second assertion
of \eqref{eq:sparse-counting-asymptotics}.

Finally, division by the total counting asymptotic gives
\[
        \mathfrak d^{\operatorname{vis}}(\cY_\Psi)
        =
        \frac1{\zeta(r)}.
\]
Together with the local density formula, this yields
\[
        \mathfrak d^{\operatorname{vis}}(\cY_\Psi)
        =
        \prod_p\left(1-\frac1{p^r}\right)
        =
        \prod_p
        \mathfrak d_p^{\operatorname{vis}}(\cY_\Psi).
\]
\end{proof}

\begin{example}\label{ex:sparse-cusp}
The hypothesis \eqref{eq:sparse-error-hypothesis} in Theorem~\ref{thm:sparse-Davenport}
cannot in general be omitted, even for a single homogeneous polynomial
with nonnegative coefficients.  Let \(d\geq2\) and consider
\[
        \Psi(x,y)=(x,y,x^dy).
\]
The associated polynomial \(F(X,Y)=X^dY\) is homogeneous of degree
\(d+1\) and has nonnegative integral coefficients. We now compute the quantities entering
\eqref{eq:sparse-error-hypothesis} explicitly.  Since
\[
        \mathcal R_X
        =
        \left\{
        (x,y)\in[1,\infty)^2:x^dy\leq X
        \right\},
\]
the inequality \(x^dy\leq X\), together with \(y\geq1\), implies
\(1\leq x\leq X^{1/d}\). 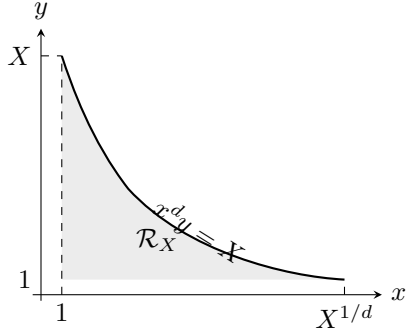
\begin{figure}[t]
\centering
\begin{tikzpicture}[x=1.1cm,y=.8cm,>=stealth]
  \fill[gray!15]
       (1,1) -- (4.4,1)
       .. controls (3.7,1.05) and (2.5,1.45) .. (1.8,2.5)
       .. controls (1.35,3.3) and (1.12,4.2) .. (1,4.7)
       -- cycle;

  \draw[thick]
       (4.4,1)
       .. controls (3.7,1.05) and (2.5,1.45) .. (1.8,2.5)
       .. controls (1.35,3.3) and (1.12,4.2) .. (1,4.7);

  \draw[->] (.65,.75) -- (4.85,.75) node[right] {$x$};
  \draw[->] (.75,.65) -- (.75,5.15) node[above] {$y$};

  \draw[dashed] (1,.75) -- (1,4.7);
  \draw[dashed] (4.4,.75) -- (4.4,1);
  \draw[dashed] (.75,4.7) -- (1,4.7);

  \node[below] at (1,.75) {$1$};
  \node[below] at (4.4,.75) {$X^{1/d}$};
  \node[left] at (.75,1) {$1$};
  \node[left] at (.75,4.7) {$X$};

  \node at (2.15,1.65) {$\mathcal R_X$};
  \node[rotate=-32] at (2.65,1.8) {$x^dy=X$};
\end{tikzpicture}
\caption{The parameter region for the sparse cusp
\(\Psi(x,y)=(x,y,x^dy)\).  The height condition is \(x^dy\leq X\).
The region has a \(y\)-projection of length comparable with \(X\),
whereas its \(x\)-projection has length only \(X^{1/d}\).}
\label{fig:sparse-cusp-region}
\end{figure} For a fixed
\(x\in[1,X^{1/d}]\), the variable \(y\) therefore ranges over
\[
        1\leq y\leq \frac{X}{x^d}.
\]
Consequently,
\begin{align*}
        V(X)
        &=
        \operatorname{vol}_2(\mathcal R_X)=
        \int_1^{X^{1/d}}
        \left(\frac{X}{x^d}-1\right)\,dx=\frac{X}{d-1}+O_d(X^{1/d})\\&\asymp_dX.
        \end{align*}
        Next, we compute \(V_1(\mathcal R_X)\).  
        Recall that this is the sum
of the lengths of the projections of \(\mathcal R_X\) onto the two
coordinate axes.  The projection onto the \(x\)-axis is
\([1,X^{1/d}]\), and therefore has length \(X^{1/d}-1\).  On the
other hand, taking \(x=1\) in the inequality \(x^dy\leq X\) shows
that every \(y\in[1,X]\) occurs.  Hence the projection onto the
\(y\)-axis is exactly \([1,X]\), and has length \(X-1\).  Thus
\[
        V_1(\mathcal R_X)
        =
        (X^{1/d}-1)+(X-1)
        =
        X+X^{1/d}-2,
\]
so in particular \(V_1(\mathcal R_X)\asymp X\).

For completeness, we also determine \(D(X)\).  If
\((x,y)\in\mathcal R_X\) and
\(T:=\min\{x,y\}\), then \(x\geq T\) and \(y\geq T\), and hence
\[
        X\geq x^dy\geq T^{d+1}.
\]
Thus \(T\leq X^{1/(d+1)}\).  Equality is attained at
\(x=y=X^{1/(d+1)}\), so $D(X)=X^{1/(d+1)}$.

Since \(r=2\), condition
\eqref{eq:sparse-error-hypothesis} reduces in this example to
\[
        D(X)
        +
        V_1(\mathcal R_X)\log\!\bigl(2D(X)\bigr)
        =
        o(V(X)).
\]
Therefore, as $X\rightarrow\infty$,
\begin{align*}
        \frac{
        V_1(\mathcal R_X)\log(2D(X))
        }{V(X)}\sim\bigl(X+X^{1/d}-2\bigr)
        \log\!\bigl(2X^{1/(d+1)}\bigr)\frac{d-1}{X}
        \sim
        \frac{d-1}{d+1}\log X
        \longrightarrow\infty.
\end{align*}
Thus, condition
\eqref{eq:sparse-error-hypothesis} not only fails, but its
\(V_1\)-term is larger than the main volume by an order
\(\log X\).

We now show that the failure of
\eqref{eq:sparse-error-hypothesis} changes the
visibility density.

Recall that the branch locus is
\[
        \cB(\cY)
        =
        \left\{
        (x,y,x^dy):
        x,y\in\Z_{\geq1}
        \right\}.
\]
Since \(x,y\geq1\), one has \(x^dy\geq x\) and \(x^dy\geq y\) and $H((x,y,x^dy))=x^dy$. Hence
\(\cB(\cY;X)\) is in bijection with the pairs
\((x,y)\in\Z_{\geq1}^2\) satisfying \(x^dy\leq X\).

We first count all branch points: 
\begin{equation}\label{eq:sparse-cusp-total-asymptotic}
        \#\cB(\cY;X)
        =
        \sum_{x\leq X^{1/d}}
        \left\lfloor\frac{X}{x^d}\right\rfloor =\zeta(d)X+O_d(X^{1/d}),
\end{equation}
and in particular
\[
        \#\cB(\cY;X)\sim\zeta(d)X.
\]

We next count the globally visible points.  By the visibility criterion
proved above, the point \((x,y,x^dy)\) is globally visible if and only
if $\gcd(x,y)=1$. This gives
\begin{equation}\label{eq:sparse-cusp-visible-sum}
        \#\cB^{\operatorname{vis}}(\cY;X)
        =
        \sum_{x\leq X^{1/d}}
        \#\left\{
        1\leq y\leq\frac{X}{x^d}:
        (x,y)=1
        \right\}.
\end{equation}
Now for fixed \(x\), the number of positive integers \(y\leq Y\) coprime
to \(x\) satisfies
\[
        \#\{1\leq y\leq Y:(x,y)=1\}
        =
        \frac{\varphi(x)}{x}Y
        +
        O(\tau(x)).
\]
This implies
\begin{align*}\label{eq:sparse-cusp-visible-sum}
\#\cB^{\operatorname{vis}}(\cY;X)
        &=
        X
        \sum_{x\leq X^{1/d}}
        \frac{\varphi(x)}{x^{d+1}}
        +
        O\left(
        \sum_{x\leq X^{1/d}}\tau(x)
        \right)\\&= \frac{\zeta(d)}{\zeta(d+1)}X
        +
        o(X),
\end{align*}
where we have used
the standard estimate
\(\sum_{n\leq T}\tau(n)\ll T\log(T)\) and $\sum_{n\geq1}\frac{\varphi(n)}{n^s}
        =
        \frac{\zeta(s-1)}{\zeta(s)}; \  \Re(s)>2.
$
Comparing
\eqref{eq:sparse-cusp-visible-sum} with
\eqref{eq:sparse-cusp-total-asymptotic}, we find
\[
        \mathfrak d^{\operatorname{vis}}(\cY)
        =
        \lim_{X\to\infty}
        \frac{\#\cB^{\operatorname{vis}}(\cY;X)}
             {\#\cB(\cY;X)}
        =
        \frac1{\zeta(d+1)}.
\]

We can also see the same exponent directly from the local densities.
Fix a prime \(p\).  A branch point fails to be \(p\)-adically visible
precisely when
$ p\mid x$ and $
        p\mid y.$
Writing \(x=pu\) and \(y=pv\), the height condition becomes
\[
        x^dy
        =
        p^{d+1}u^dv
        \leq X.
\]
Thus division by \(p\) in both radial coordinates gives a bijection
between the \(p\)-adically invisible points of height at most \(X\)
and the entire branch locus of height at most \(X/p^{d+1}\).  Hence
\[
        \#\bigl(
        \cB(\cY;X)\setminus
        \cB_p^{\operatorname{vis}}(\cY;X)
        \bigr)
        =
        \#\cB\left(\cY;\frac{X}{p^{d+1}}\right).
\]
Using \eqref{eq:sparse-cusp-total-asymptotic}, the proportion of
\(p\)-adically invisible points therefore tends to
\[
        \frac{
        \zeta(d)X/p^{d+1}
        }{
        \zeta(d)X
        }
        =
        \frac1{p^{d+1}}.
\]
Consequently
\[
        \mathfrak d_p^{\operatorname{vis}}(\cY)
        =
        1-\frac1{p^{d+1}}.
\]
Taking the product over all primes gives
\[
        \prod_p
        \mathfrak d_p^{\operatorname{vis}}(\cY)
        =
        \prod_p
        \left(1-\frac1{p^{d+1}}\right)
        =
        \frac1{\zeta(d+1)}
        =
        \mathfrak d^{\operatorname{vis}}(\cY).
\]

Thus visibility is still locally detectable and the global density
still factors as the product of the local densities.  
\end{example}

\begin{example}\label{prop:first-sparse-family}
Consider the family \(\phi_a(t)=(t,at,a^2t^2)\), where \(a\in\Q_{>0}\).  This is precisely the radial graph family associated to \(r=2\) and the homogeneous polynomial \(F_1(X,Y)=Y^2\).  Indeed, the corresponding graph map is \(\Psi(x,y)=(x,y,y^2)\), and \(\Psi(t,at)=\phi_a(t)\).  Theorem~\ref{thm:sparse-Davenport} therefore applies once its counting hypothesis is verified.

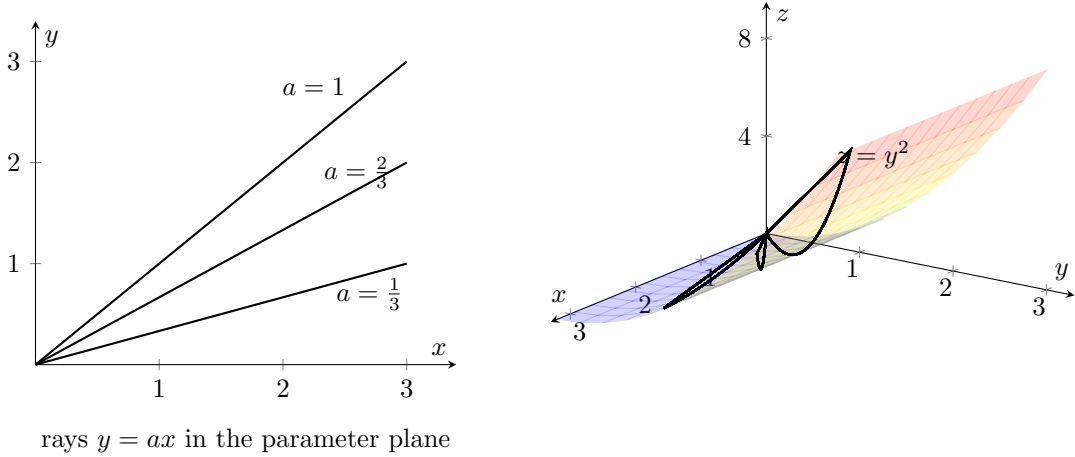
\begin{figure}[t]
\centering
\begin{tikzpicture}

\begin{axis}[
    name=plane,
    width=0.42\textwidth,
    height=0.36\textwidth,
    axis lines=middle,
    xmin=0,xmax=3.4,
    ymin=0,ymax=3.4,
    xlabel={$x$},
    ylabel={$y$},
    xtick={1,2,3},
    ytick={1,2,3},
    clip=false
]
    \addplot[thick,domain=0:3,samples=2] {x/3};
    \addplot[thick,domain=0:3,samples=2] {2*x/3};
    \addplot[thick,domain=0:3,samples=2] {x};

    \node at (axis cs:2.7,0.7) {$a=\frac13$};
    \node at (axis cs:2.6,1.9) {$a=\frac23$};
    \node at (axis cs:2.25,2.75) {$a=1$};

    \node[align=center] at (axis cs:1.7,-0.75)
    {rays \(y=ax\) in the parameter plane};
\end{axis}

\begin{axis}[
    at={(plane.east)},
    anchor=west,
    xshift=1.25cm,
    width=0.50\textwidth,
    height=0.39\textwidth,
    view={125}{25},
    axis lines=center,
    xmin=0,xmax=3.3,
    ymin=0,ymax=3.3,
    zmin=0,zmax=9.5,
    xlabel={$x$},
    ylabel={$y$},
    zlabel={$z$},
    xtick={1,2,3},
    ytick={1,2,3},
    ztick={0,4,8},
    clip=false
]
    \addplot3[
        surf,
        opacity=0.18,
        shader=flat,
        domain=0:3.2,
        y domain=0:3,
        samples=13,
        samples y=13
    ]
    {y^2};

    \addplot3[
        thick,
        domain=0:3,
        samples=60
    ]
    ({x},{x/3},{x^2/9});

    \addplot3[
        thick,
        domain=0:3,
        samples=60
    ]
    ({x},{2*x/3},{4*x^2/9});

    \addplot3[
        thick,
        domain=0:3,
        samples=60
    ]
    ({x},{x},{x^2});

    \node at (axis cs:2.8,3.1,8.7) {$z=y^2$};
\end{axis}

\end{tikzpicture}

\caption{The geometry of Example~\ref{prop:first-sparse-family}.
On the left, the branches are determined by the rays \(y=ax\) in
the parameter plane.  Under the graph embedding
\(\Psi(x,y)=(x,y,y^2)\), each ray lifts to the curve
\(\phi_a(t)=(t,at,a^2t^2)\) on the surface \(z=y^2\), shown on the
right.}
\label{fig:first-sparse-family}
\end{figure}

The branch locus is the graph \(\{(x,y,y^2):x,y\in\Z_{\geq1}\}\), which has dimension \(2\) inside \(\A^3\).  In particular, the family is sparse.  The general visibility part of the theorem already shows that \((x,y,y^2)\) is globally visible exactly when \(\gcd(x,y)=1\), and that it is \(p\)-adically visible exactly when \(p\) does not divide both \(x\) and \(y\).  Thus there is no need to establish these criteria separately.

It remains only to verify condition \eqref{eq:sparse-error-hypothesis}.  For \(X\geq1\), the height condition is equivalent to \(1\leq x\leq X\) and \(1\leq y\leq X^{1/2}\).  Hence \(\mathcal R_X\) is a rectangle.  Its area is \((X-1)(X^{1/2}-1)\), so \(V(X)=X^{3/2}+O(X)\).  The two one-dimensional coordinate projections have lengths \(X-1\) and \(X^{1/2}-1\), and therefore \(V_1(\mathcal R_X)=O(X)\).  Finally, the largest possible value of \(\min\{x,y\}\) in the region is \(X^{1/2}\), so \(D(X)=X^{1/2}\).

Consequently the error expression in \eqref{eq:sparse-error-hypothesis} is \(O(X\log X)\), whereas \(V(X)\asymp X^{3/2}\).  Hence the hypothesis of Theorem~\ref{thm:sparse-Davenport} is satisfied.

It follows from Theorem \ref{thm:sparse-Davenport}
\[
 \mathfrak d^{\operatorname{vis}}(\cY)
 =
 \frac1{\zeta(2)},
 \qquad
 \mathfrak d_p^{\operatorname{vis}}(\cY)
 =
 1-\frac1{p^2},
 \qquad
 \mathfrak d^{\operatorname{vis}}(\cY)
 =
 \prod_p\mathfrak d_p^{\operatorname{vis}}(\cY),
\]
and in fact
\[\#\cB^{\operatorname{vis}}(\cY;X)=V(X)/\zeta(2)+O(X\log X).\]
\end{example}
\begin{example}\label{ex:davenport-sparse-family}
Let
\[
 Q(X,Y,Z)=X^2+XY+Y^2+Z^2
 \qquad\text{and}\qquad
 F(X,Y,Z)=X^3+Y^3+Z^3+XYZ.
\]
For \(a,b\in\Q_{>0}\), consider
\[
 \phi_{a,b}(t)
 =
 \bigl(t,at,bt,Q(1,a,b)t^2,F(1,a,b)t^3\bigr),
 \qquad t>0.
\]
Since \(Q\) and \(F\) are homogeneous of degrees \(2\) and \(3\), respectively, this is exactly the radial graph family associated to
\[
 \Psi(x,y,z)
 =
 (x,y,z,Q(x,y,z),F(x,y,z)).
\]

The branch locus is therefore
\[
 \{\Psi(x,y,z):x,y,z\in\Z_{\geq1}\}.
\]
It lies on the codimension-two subvariety of \(\A^5\) defined by the two graph equations \(x_4=Q(x_1,x_2,x_3)\) and \(x_5=F(x_1,x_2,x_3)\).  Theorem~\ref{thm:sparse-Davenport} already gives the visibility criterion: a branch point is globally visible if and only if \(\gcd(x,y,z)=1\), and it is \(p\)-adically visible if and only if \(p\) does not divide all three of \(x,y,z\).

We verify the height and counting hypotheses carefully.  Suppose \(x,y,z\geq1\).  Then \(x^3\geq x^2\), \(y^3\geq y^2\), \(z^3\geq z^2\), and \(xyz\geq xy\).  Adding these inequalities shows that \(F(x,y,z)\geq Q(x,y,z)\).  We also have \(F(x,y,z)\geq x^3\geq x\), and similarly \(F(x,y,z)\geq y,z\).  Thus on the region relevant to the branch lattice the ambient max-height is simply \(F(x,y,z)\).

Consequently,
\[\mathcal R_X=\{(x,y,z)\in[1,\infty)^3: F(x,y,z)\le X\}\]
  It is bounded, since \(F\geq x^3,y^3,z^3\).

For the volume estimate it is convenient to compare \(\mathcal R_X\) with
\[
 \mathcal S_X
 :=
 \{(u,v,w)\in\R_{\geq0}^3:F(u,v,w)\leq X\}.
\]
Homogeneity gives \(\mathcal S_X=X^{1/3}\mathcal S_1\).  If
\(\kappa:=\operatorname{vol}_3(\mathcal S_1)\), it follows that
\(\operatorname{vol}_3(\mathcal S_X)=\kappa X\).

The difference between \(\mathcal S_X\) and \(\mathcal R_X\) is contained in the union of the three slabs in which at least one coordinate is less than \(1\).  Since \(\mathcal S_X\subseteq[0,X^{1/3}]^3\), each such slab has volume \(O(X^{2/3})\).  Hence
\(V(X)=\kappa X+O(X^{2/3})\).

The same comparison gives the required bounds for the coordinate projections.  Every two-dimensional coordinate projection of \(\mathcal R_X\) is contained in the corresponding projection of \(\mathcal S_X=X^{1/3}\mathcal S_1\), and therefore has area \(O(X^{2/3})\).  Likewise every one-dimensional coordinate projection has length \(O(X^{1/3})\).  Thus
\(V_2(\mathcal R_X)=O(X^{2/3})\) and
\(V_1(\mathcal R_X)=O(X^{1/3})\).

It remains to estimate \(D(X)\).  If \(m=\min\{x,y,z\}\), then
\(F(x,y,z)\geq x^3+y^3+z^3\geq3m^3\), so membership in
\(\mathcal R_X\) implies \(m\ll X^{1/3}\).  Conversely, taking \(x=y=z=t\) gives \(F(t,t,t)=4t^3\), so for \(t=(X/4)^{1/3}\) and sufficiently large \(X\) the point \((t,t,t)\) belongs to \(\mathcal R_X\).  Hence \(D(X)\asymp X^{1/3}\).

The error expression in \eqref{eq:sparse-error-hypothesis} is now
\(O(X^{2/3}+X^{1/3}\log X)\), which is \(o(X)\).  Since
\(V(X)\sim\kappa X\), all the hypotheses of
Theorem~\ref{thm:sparse-Davenport} are satisfied.

We therefore obtain
\[
 \#\cB(\cY;X)
 =
 \kappa X+O(X^{2/3}),
 \quad\text{and}\quad
 \#\cB^{\operatorname{vis}}(\cY;X)
 =
 \frac{\kappa}{\zeta(3)}X+O(X^{2/3}).
\]
In particular,
\[
 \mathfrak d^{\operatorname{vis}}(\cY)
 =
 \frac1{\zeta(3)},
 \qquad
 \mathfrak d_p^{\operatorname{vis}}(\cY)
 =
 1-\frac1{p^3},
 \qquad
 \mathfrak d^{\operatorname{vis}}(\cY)
 =
 \prod_p\mathfrak d_p^{\operatorname{vis}}(\cY).
\]
\end{example}
\subsection{Rational rays on an affine cone}
\label{subsec:homogeneous-cones}

The sparse families considered above were obtained from polynomial graph
maps of the form
\[
        \Psi(x_1,\dots,x_r)
        =
        \bigl(
        x_1,\dots,x_r,
        F_1(x),\dots,F_s(x)
        \bigr).
\]
In particular, the parameters \(x_1,\dots,x_r\) themselves occurred as coordinates, and the remaining coordinates were polynomial
functions of them. We consider a natural family which does not fit into the above framework. Let \(V\subseteq\A^n\) be an affine variety defined over \(\Q\) by
homogeneous equations.  Thus, if \(v\in V(\R)\) and \(t\in\R\), then
\(tv\in V(\R)\).  Suppose that \(V(\Q)\cap\Q_{>0}^n\) is nonempty.
For every positive rational ray
\[
        [v]\in
        \bigl(V(\Q)\cap\Q_{>0}^n\bigr)/\Q_{>0},
\]
choose a representative \(v\), and define the corresponding ordered
branch by \(\phi_{[v]}(t)=tv\) for \(t>0\).  Replacing \(v\) by a
positive rational multiple merely gives an order-preserving
reparametrization of the same branch.  We denote the resulting
visibility datum by \(\cY_V\).

Distinct rays meet only at the origin, and every positive integral
point of \(V\) lies on its rational ray.  Consequently the branch locus
is
\[
        \cB(\cY_V)
        =
        V(\Z)\cap\Z_{>0}^n.
\]

The following proposition gives the local-global principle in this
setting.

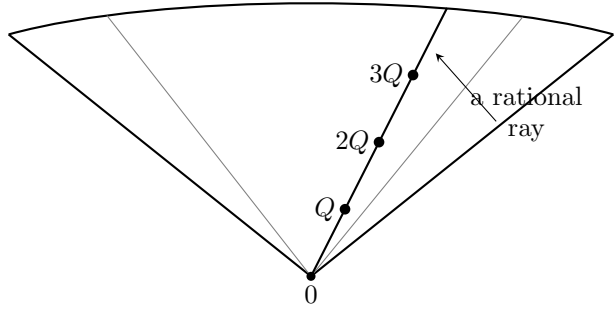
\begin{figure}[t]
\centering
\begin{tikzpicture}[x=1cm,y=1cm,>=stealth]
  \coordinate (O) at (0,0);

  \draw[thick] (-4,3.2) -- (O) -- (4,3.2);
  \draw[thick]
        (-4,3.2) .. controls (-2,3.75) and (2,3.75) .. (4,3.2);

  \draw[gray] (O) -- (-2.7,3.45);
  \draw[gray] (O) -- (2.8,3.43);
  \draw[thick] (O) -- (1.8,3.55);

  \fill (O) circle (1.7pt) node[below] {$0$};
  \fill (.45,.8875) circle (2pt) node[left] {$Q$};
  \fill (.9,1.775) circle (2pt) node[left] {$2Q$};
  \fill (1.35,2.6625) circle (2pt) node[left] {$3Q$};

  \node[align=center] at (2.85,2.15)
       {a rational\\ray};
  \draw[->] (2.45,2.05) -- (1.65,2.95);
\end{tikzpicture}
\caption{Rational-ray visibility on an affine cone.  On a rational ray
the primitive integral point \(Q\) is the first lattice point, while
\(2Q,3Q,\ldots\) occur later and are globally invisible.  Divisibility
of all ambient coordinates detects the same obstruction prime by
prime.}
\label{fig:cone-rays}
\end{figure}

\begin{proposition}\label{prop:cone-local-global}
Let \(V\subseteq\A^n\) be a homogeneous affine variety as above, and
let \(P=(P_1,\dots,P_n)\in\cB(\cY_V)\).  Then \(P\) is globally visible
if and only if
\(\gcd(P_1,\dots,P_n)=1\).  For every prime \(p\), the point \(P\) is
\(p\)-adically visible if and only if \(p\) does not divide all of the
coordinates \(P_i\).  In particular, visibility is locally detectable
for \(\cY_V\).
\end{proposition}

\begin{proof}
Suppose first that \(d:=\gcd(P_1,\dots,P_n)>1\).  Since the equations
defining \(V\) are homogeneous, the point \(P/d\) also lies on \(V\).
All of its coordinates are positive integers, so \(P/d\) is another
point of the branch locus.  Moreover, it lies on the same rational ray
as \(P\) and occurs earlier on that ray. Thus, \(P\) is not globally
visible.

Conversely, suppose that \(P\) is not globally visible.  Then there is
an earlier integral point \(Q\) on the same ray.  Hence \(Q=\lambda P\)
for some real number \(0<\lambda<1\).  Since \(P\) and \(Q\) have
nonzero integral coordinates, \(\lambda\) is rational.  Write
\(\lambda=a/b\) in lowest terms.  Since \(aP_i/b\) is an integer for
every \(i\) and \((a,b)=1\), the integer \(b\) divides every \(P_i\).
 Inequality \(\lambda<1\) gives \(b>1\), so the coordinates of \(P\)
have a nontrivial common divisor.

Now fix a prime \(p\).  If \(p\) divides every coordinate of \(P\), then
\(P/p\) is an earlier integral point on the same branch.  Furthermore,
\(v_p(P/p)=v_p(P)-1\). Thus, \(P\) is not \(p\)-adically visible.

If \(p\) does not divide every coordinate, then \(v_p(P)=0\).  Every
integral point has a nonnegative \(p\) -adic minimum valuation, so no
earlier integral point can have a valuation strictly smaller than zero.
Hence \(P\) is \(p\)-adically visible.

It follows that \(P\) is \(p\)-adically visible for every prime \(p\)
if and only if no prime divides all of its coordinates, which is
equivalent to \(\gcd(P_1,\dots,P_n)=1\).  The latter is exactly the
global visibility condition that was proved above.
\end{proof}

The density problem can also be reduced to primitive-point counting.
Write \(N_V(X):=\#\cB(\cY_V;X)\), and let \(P_V(X)\) denote the number
of primitive points, equivalently globally visible points in
\(\cB(\cY_V;X)\).

\begin{proposition}
\label{prop:cone-density}
Suppose that for some \(c>0\), \(\alpha\geq1\), and
\(\beta\geq0\), one has
\begin{equation}
        P_V(X)
        \sim
        cX^\alpha(\log X)^\beta.
        \label{PV}
\end{equation}
Then the following assertions hold.

If \(\alpha>1\), then
\[
        N_V(X)
        \sim
        c\zeta(\alpha)X^\alpha(\log X)^\beta,
\]
and consequently
\[
        \mathfrak d^{\vis}(\cY_V)
        =
        \frac{1}{\zeta(\alpha)}.
\]
For every prime \(p\),
\[
        \mathfrak d_p^{\vis}(\cY_V)
        =
        1-p^{-\alpha},
\]
so that
\[
        \mathfrak d^{\vis}(\cY_V)
        =
        \prod_p
        \mathfrak d_p^{\vis}(\cY_V).
\]

If \(\alpha=1\), then
\[
        N_V(X)
        \sim
        \frac{c}{\beta+1}
        X(\log X)^{\beta+1}.
\]
In this case, the globally visible points have relative density zero,
while for every prime \(p\) one has
\[
        \mathfrak d_p^{\vis}(\cY_V)
        =
        1-\frac1p.
\]
Thus, the Euler product of the local densities is zero and again agrees
with the global density.
\end{proposition}

\begin{proof}
Recall that \(P_V(X)\) denotes the number of primitive points of
\(\cB(\cY_V)\) having height at most \(X\), while \(N_V(X)\) denotes
the total number of points of \(\cB(\cY_V)\) having height at most
\(X\).

Every point
\[
        P=(P_1,\dots,P_n)\in\cB(\cY_V)
\]
has a unique decomposition
\[
        P=dQ,\] with 
$d=\gcd(P_1,\dots,P_n)
$
and $Q$ primitive.  Since \(V\) is an affine cone, \(Q\) again
belongs to \(V(\Z)\), and since all coordinates considered are
positive, \(Q\in\cB(\cY_V)\).  Conversely, if \(Q\) is a primitive
branch point and \(d\geq1\), then \(dQ\) is again a branch point.
The height is homogeneous under ordinary scalar multiplication:
$ H(dQ)=dH(Q).$
Consequently, for a fixed positive integer \(d\), the points of height
at most \(X\) whose coordinate gcd is exactly \(d\) are in bijection
with the primitive points \(Q\) satisfying \(H(Q)\leq X/d\).  It
follows that
\begin{equation}\label{eq:cone-total-from-primitive}
        N_V(X)
        =
        \sum_{d\le X}P_V(X/d).
\end{equation}
From \eqref{PV}, we write
\[P_V(Y)=cY^{\alpha}(\log Y)^{\beta}+E(Y),\]
where $E(Y)=o(1)$.
Partial summation on the first term gives: for $\alpha>1$ and $\beta>0$
\begin{align*}
    \sum_{d\le X}c(X/d)^{\alpha}(\log (X/d))^{\beta}=cX^{\alpha}\zeta(\alpha)(\log X)^{\beta}+O_{\alpha,\beta}((\log X)^{\beta-1}).
\end{align*}
For $\beta=0$ and $\alpha>1$, it is
\[c\sum_{d\le X}\frac{X^{\alpha}}{d^{\alpha}}=c\zeta(\alpha)X^{\alpha}+O(X).\]
For $d\asymp X$, $X/d$ is bounded, and so to estimate the remaining term $\sum_{d\le X}E(X/d)$, we split the sum over $d$ into the range $d\le X/T$ and $X/T<d\le X$ for a choice of $T<X$. Using the boundedness of $E$ in the range $d\le X/T$ for a suitable choice of $T$, and the sum in the remaining range going to $o(X)$, we can write the following
\[\sum_{d\le X}E(X/d)=O(X).\]
Putting all of this together, we have: for $\beta>0$,
\begin{equation*}
N_V(X)=cX^{\alpha}\zeta(\alpha)(\log X)^{\beta}+O(X)
\end{equation*}
Hence, for $\alpha>1$ and $\beta\ge 0$,
\begin{equation} \label{eq:cone-total-alpha-large}
N_V(X)\sim cX^{\alpha}\zeta(\alpha)(\log X)^{\beta}. 
\end{equation}
Thus, we obtain the claimed asymptotic of Proposition \ref{prop:cone-density}.
By Proposition~\ref{prop:cone-local-global}, the globally visible
points are precisely the primitive points. Therefore,
from \eqref{PV} and
\eqref{eq:cone-total-alpha-large}, we obtain
\[
 \mathfrak d^{\vis}(\cY_V)
 =
 \lim_{X\to\infty}\frac{P_V(X)}{N_V(X)}=\frac{1}{\zeta(\alpha)}.
\]

Next, we compute the local density.  Fix a prime \(p\).  By
Proposition~\ref{prop:cone-local-global}, a branch point fails to be
\(p\)-adically visible if and only if all its coordinates are
divisible by \(p\). The division by \(p\) gives a bijection
\[
 \left\{
 P\in\cB(\cY_V;X):
 p\mid P_i\text{ for every }i
 \right\}
 \longleftrightarrow
 \cB(\cY_V;X/p).
\]
Indeed, if \(P=pQ\), the homogeneity of the equations defining \(V\)
shows that \(Q\in V(\Z)\), the positivity is preserved, and
\(H(Q)=H(P)/p\). Conversely, if \(Q\) has height at most \(X/p\),
then \(pQ\) has height at most \(X\) and all its coordinates are
divisible by \(p\).

Thus, the number of points of height at most \(X\) that fail
\(p\)-adic visibility is exactly \(N_V(X/p)\).  By
\eqref{eq:cone-total-alpha-large},
\begin{align*}
 \frac{N_V(X/p)}{N_V(X)}
 &\sim
 \frac{
 (X/p)^\alpha
 \bigl(\log(X/p)\bigr)^\beta
 }{
 X^\alpha(\log X)^\beta
 }\\
 &=
 p^{-\alpha}
 \left(
 1-\frac{\log p}{\log X}
 \right)^\beta
 \longrightarrow
 p^{-\alpha}.
\end{align*}
Consequently, the proportion of \(p\)-adically visible points tends to
\[
        \mathfrak d_p^{\vis}(\cY_V)
        =
        1-p^{-\alpha}.
\]
Finally, since \(\alpha>1\), Euler's product for the zeta function
converges absolutely and gives
\[
        \prod_p\mathfrak d_p^{\vis}(\cY_V)
        =
        \prod_p(1-p^{-\alpha})
        =
        \frac1{\zeta(\alpha)}
        =
        \mathfrak d^{\vis}(\cY_V).
\]

\smallskip
\noindent
\emph{Case \(\alpha=1\).}

We now assume
\[
        P_V(Y)
        \sim
        cY(\log Y)^\beta.
\]
Hence, $P_Y(Y)=cY(\log Y)^{\beta}+E(Y)$, where $E(Y)=o(1)$ as $Y\rightarrow\infty$. 
Applying partial summation formula to the first term gives
\begin{align}
    cX\sum_{d\le X}\frac{\log (X/d)^{\beta}}{d}=\frac{X(\log X)^{\beta+1}}{\beta+1}+O\left(X(\log X)^{\beta}\right).
\end{align}
For the remainder term $E(X)$, as before, we have the partial sum contribution $O(X)$.
Therefore, for $\alpha=1$, we obtain
\begin{equation}\label{eq:cone-critical-total}
        N_V(X)
        \sim
        \frac{c}{\beta+1}
        X(\log X)^{\beta+1}.
\end{equation}
The globally visible points are exactly the primitive points, so
\[
 \frac{P_V(X)}{N_V(X)}
 \sim
 \frac{
 cX(\log X)^\beta
 }{
 \frac{c}{\beta+1}X(\log X)^{\beta+1}
 }
 =
 \frac{\beta+1}{\log X},
\]
which tends to zero.  Hence
\[
        \mathfrak d^{\vis}(\cY_V)=0.
\]

Finally, fix a prime \(p\).  As before, the number of points which fail
\(p\)-adic visibility is exactly \(N_V(X/p)\).  From
\eqref{eq:cone-critical-total},
\begin{align*}
 \frac{N_V(X/p)}{N_V(X)}
 &\sim
 \frac{
 \frac{X}{p}
 \bigl(\log(X/p)\bigr)^{\beta+1}
 }{
 X(\log X)^{\beta+1}
 }\\
 &=
 \frac1p
 \left(
 1-\frac{\log p}{\log X}
 \right)^{\beta+1}
 \longrightarrow
 \frac1p.
\end{align*}
Therefore
\[
        \mathfrak d_p^{\vis}(\cY_V)
        =
        1-\frac1p.
\]

The corresponding Euler product vanishes.  Indeed,
\[
        \prod_p\left(1-\frac1p\right)=0.
\]
For example, if the product is taken over \(p\leq Y\), then
\[
 \log\prod_{p\leq Y}\left(1-\frac1p\right)
 =
 \sum_{p\leq Y}\log\left(1-\frac1p\right)
 \leq
 -\sum_{p\leq Y}\frac1p,
\]
and the latter tends to \(-\infty\) because
\(\sum_p p^{-1}\) diverges.  Thus the partial products tend to zero.
Consequently,
\[
        \mathfrak d^{\vis}(\cY_V)
        =
        0
        =
        \prod_p\mathfrak d_p^{\vis}(\cY_V).
\]
This completes the proof.
\end{proof}

\subsubsection{Pythagorean triples}

Consider the quadratic cone
\[
        V_{\mathrm{Pyth}}
        :
        X^2+Y^2=Z^2
        \subseteq\A^3
\]
and the rational-ray visibility datum on its positive part.  Its branch
locus consists of the positive integral Pythagorean triples.

This variety is not a polynomial graph of the type considered earlier.
For example, projection onto the first two coordinates would require
the third coordinate to be
\(\sqrt{x^2+y^2}\), rather than a polynomial in \(x\) and \(y\).

\begin{figure}[t]
\centering
\begin{tikzpicture}
\begin{axis}[
    width=0.72\textwidth,
    height=0.62\textwidth,
    view={45}{24},
    axis lines=center,
    xlabel={$X$},
    ylabel={$Y$},
    zlabel={$Z$},
    xmin=-14,xmax=14,
    ymin=-14,ymax=14,
    zmin=0,zmax=14,
    xtick={-10,0,10},
    ytick={-10,0,10},
    ztick={0,5,10},
    unit vector ratio*=1 1 1,
    clip=false
]

\addplot3[
    surf,
    opacity=0.16,
    shader=interp,
    domain=0:13,
    y domain=0:360,
    samples=18,
    samples y=40
]
({x*cos(y)},{x*sin(y)},{x});

\addplot3[
    thick,
    domain=0:360,
    samples=100
]
({13*cos(x)},{13*sin(x)},{13});

\addplot3[
    very thick,
    domain=0:2.4,
    samples=2
]
({3*x},{4*x},{5*x});

\addplot3[
    only marks,
    mark=*,
    mark size=2.2pt
]
coordinates {
    (3,4,5)
    (6,8,10)
};

\node at (axis cs:3.3,4.2,5.8)
      {$P=(3,4,5)$};

\node at (axis cs:6.3,8.2,10.8)
      {$2P=(6,8,10)$};

\addplot3[
    very thick,
    dashed,
    domain=0:1,
    samples=2
]
({5*x},{12*x},{13*x});

\addplot3[
    only marks,
    mark=*,
    mark size=2pt
]
coordinates {
    (5,12,13)
};

\node at (axis cs:5.6,11.4,13.3)
      {$(5,12,13)$};

\addplot3[
    only marks,
    mark=*,
    mark size=1.8pt
]
coordinates {(0,0,0)};

\node at (axis cs:-1,-1,0.6) {$0$};

\end{axis}
\end{tikzpicture}

\caption{Rational-ray visibility on the Pythagorean cone
\(X^2+Y^2=Z^2\).  The full upper cone is shown to emphasize the
ambient geometry, while the arithmetic visibility datum uses its
positive part.  The primitive triple \(P=(3,4,5)\) is the first
integral point on its rational ray, whereas \(2P=(6,8,10)\) is a
later integral point and is therefore not globally visible.  The
dashed line illustrates a second rational ray through the primitive
triple \((5,12,13)\).}
\label{fig:pythagorean-rays}
\end{figure}
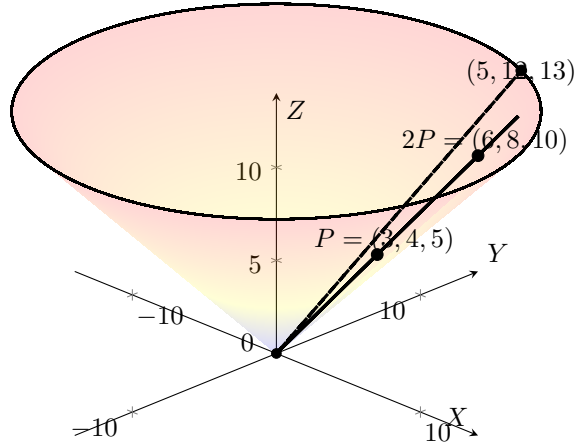

\begin{proposition}\label{prop:Pythagorean-visibility}
For the Pythagorean visibility datum one has
\[
        \mathfrak d^{\vis}(\cY_{\mathrm{Pyth}})=0,
        \qquad
        \mathfrak d_p^{\vis}(\cY_{\mathrm{Pyth}})
        =
        1-\frac1p
\]
for every prime \(p\).  Visibility is locally detectable, and
\[
        \mathfrak d^{\vis}(\cY_{\mathrm{Pyth}})
        =
        \prod_p
        \mathfrak d_p^{\vis}(\cY_{\mathrm{Pyth}}).
\]
\end{proposition}

\begin{proof}
By Proposition~\ref{prop:cone-local-global}, a point of the
Pythagorean cone is globally visible if and only if its three
coordinates are relatively prime.  Thus the globally visible points
are precisely the primitive positive integral solutions of
\(x^2+y^2=z^2\).  The same proposition also shows that visibility is
locally detectable. It remains to determine the growth of the primitive
points.

We recall the classical parametrization of primitive Pythagorean
triples. Let \(m>n>0\) be integers satisfying \(\gcd(m,n)=1\) and such
that \(m\) and \(n\) have opposite parity. Then
\(m^2-n^2, 2mn, m^2+n^2\) are positive integers satisfying
\((m^2-n^2)^2+(2mn)^2=(m^2+n^2)^2\), and their greatest common divisor
is \(1\). Conversely, every primitive positive integral solution of
\(x^2+y^2=z^2\) arises uniquely in one of the two forms
\[
        (x,y,z)
        =
        (m^2-n^2,2mn,m^2+n^2)
\]
or
\[
        (x,y,z)
        =
        (2mn,m^2-n^2,m^2+n^2),
\]
for a unique pair \(m>n>0\) satisfying the coprimality and parity
conditions above.

The occurrence of these two ordered triples is important in our
setting, since the first and second coordinates of the ambient affine
space are distinguished. Interchanging them therefore gives a
different point unless they are equal. In the primitive parametrization
they cannot be equal: the equation \(m^2-n^2=2mn\) would imply
\((m/n)^2-2(m/n)-1=0\), and hence \(m/n=1+\sqrt2\), which is impossible
for rational \(m/n\). Thus each admissible parameter pair \((m,n)\)
gives exactly two distinct primitive points of the branch locus.

Since \(m^2+n^2>m^2-n^2\) and \(m^2+n^2>2mn\), the max-height of
either corresponding point is \(H(x,y,z)=m^2+n^2\). Therefore the
primitive points of height at most \(X\) correspond to the integer
pairs
\[
        (m,n)\in\Z^2,
        \qquad
        m>n>0,\qquad
        m^2+n^2\leq X,
\]
which are coprime and of opposite parity, with each such pair counted
twice.

We now count these parameter pairs. Let
\[
        \mathcal S_X
        :=
        \left\{
        (u,v)\in\R^2:
        u>v>0,\;
        u^2+v^2\leq X
        \right\}.
\]
In polar coordinates, the inequalities \(u>v>0\) correspond to angles
\(0<\theta<\pi/4\), while \(u^2+v^2\leq X\) corresponds to radius at
most \(\sqrt X\). Hence \(\mathcal S_X\) is a circular sector of angle
\(\pi/4\), and its area is
\[
        \operatorname{area}(\mathcal S_X)
        =
        \frac12\cdot\frac{\pi}{4}\cdot X
        =
        \frac{\pi X}{8}.
\]

We next determine the density of integer pairs \((m,n)\) satisfying
the two arithmetic conditions. First consider the condition
\(\gcd(m,n)=1\). For every odd prime \(p\), the only forbidden residue
class condition is that both \(m\) and \(n\) be divisible by \(p\).
Thus the local factor at \(p\) is \(1-p^{-2}\). At the prime \(2\),
the requirement that \(m\) and \(n\) have opposite parity means that
the residue class of \((m,n)\) modulo \(2\) must be either \((1,0)\)
or \((0,1)\). There are four residue classes modulo \(2\) in total,
so the local factor at \(2\) is \(2/4=1/2\).

It follows that the density of coprime pairs of opposite parity in
\(\Z^2\) is
\[
        \frac12
        \prod_{p\ \mathrm{odd}}
        \left(1-\frac1{p^2}\right).
\]
Using \(\prod_p(1-p^{-2})=1/\zeta(2)=6/\pi^2\), and separating the
factor at \(p=2\), we obtain
\[
        \prod_{p\ \mathrm{odd}}
        \left(1-\frac1{p^2}\right)
        =
        \frac{6/\pi^2}{3/4}
        =
        \frac8{\pi^2}.
\]
Hence the density is \(4/\pi^2\).

A standard M\"obius-inversion lattice-point argument, applied to the
expanding sector \(\mathcal S_X\), therefore gives
\[
        \#\left\{
        (m,n)\in\mathcal S_X\cap\Z^2:
        \gcd(m,n)=1,\;
        m\not\equiv n\pmod2
        \right\}
        =
        \frac4{\pi^2}\operatorname{area}(\mathcal S_X)
        +o(X).
\]
Since \(\operatorname{area}(\mathcal S_X)=\pi X/8\), the right-hand
side is \(X/(2\pi)+o(X)\).

Each admissible parameter pair gives the two distinct ordered primitive
points described above. Consequently,
\[
        P_{V_{\mathrm{Pyth}}}(X)
        =
        2\left(
        \frac{X}{2\pi}+o(X)
        \right)
        =
        \frac{X}{\pi}+o(X),
\]
and therefore \(P_{V_{\mathrm{Pyth}}}(X)\sim X/\pi\).

We may now apply Proposition~\ref{prop:cone-density} with
\(\alpha=1\), \(\beta=0\), and \(c=1/\pi\). It follows that the total
number of positive integral points on the Pythagorean cone satisfies
\[
        N_{V_{\mathrm{Pyth}}}(X)
        \sim
        \frac{X}{\pi}\log X.
\]
Hence
\[
        \frac{
        P_{V_{\mathrm{Pyth}}}(X)
        }{
        N_{V_{\mathrm{Pyth}}}(X)
        }
        \sim
        \frac1{\log X}
        \longrightarrow0.
\]
Therefore \(\mathfrak d^{\vis}(\cY_{\mathrm{Pyth}})=0\).

For a fixed prime \(p\), Proposition~\ref{prop:cone-density} gives
\(\mathfrak d_p^{\vis}(\cY_{\mathrm{Pyth}})=1-1/p\).
Equivalently, this follows directly by observing that a point fails
\(p\)-adic visibility precisely when all three of its coordinates are
divisible by \(p\). Division by \(p\) then identifies such points of
height at most \(X\) with all Pythagorean points of height at most
\(X/p\). Since
\(N_{V_{\mathrm{Pyth}}}(X)\sim (X/\pi)\log X\), one has
\[
        \frac{
        N_{V_{\mathrm{Pyth}}}(X/p)
        }{
        N_{V_{\mathrm{Pyth}}}(X)
        }
        \longrightarrow
        \frac1p.
\]
Thus the proportion which is \(p\)-adically visible tends to
\(1-1/p\).

Finally,
\[
        \prod_p
        \mathfrak d_p^{\vis}(\cY_{\mathrm{Pyth}})
        =
        \prod_p\left(1-\frac1p\right)
        =
        0,
\]
because \(\sum_p1/p\) diverges. This agrees with the global density
computed above.
\end{proof}

Thus primitive triples form a zero-density subset of all integral
Pythagorean triples.
The reason is that summing the positive integer multiples of primitive
triples introduces a harmonic series.

\subsubsection{The quadratic cone \(XZ=Y^2\)}

We begin with one of the simplest examples which is genuinely different
from the polynomial graph families considered above.  Let
\(V_2\subseteq\A^3\) be the affine surface defined by \(XZ=Y^2\), and
consider its positive integral points.  The variety \(V_2\) is an
ordinary affine cone: if \(P=(x,y,z)\in V_2\) and \(\lambda>0\), then
\(\lambda P\in V_2\), since
\((\lambda x)(\lambda z)=(\lambda y)^2\).  Thus the rational rays
through the origin define a visibility datum of the type considered in
Proposition~\ref{prop:cone-local-global}.

This example is not covered by the polynomial graph construction of
the preceding subsection.  Indeed, if one tries to regard \(x\) and
\(y\) as free coordinates, then the third coordinate is forced to be
\(z=y^2/x\).  An arbitrary positive pair \((x,y)\in\Z_{>0}^2\) does
not therefore determine an integral point on \(V_2\); one must impose
the additional arithmetic condition \(x\mid y^2\).  In particular,
the integral points do not arise as the graph of an integral polynomial
over an unrestricted positive lattice.

The primitive points nevertheless admit a very simple description,
and this allows the visibility density to be computed completely.

\begin{proposition}\label{prop:quadratic-cone-density}
For the rational-ray visibility datum on \(V_2\), one has
\[
        \mathfrak d^{\vis}(\cY_{V_2})=0,
        \qquad
        \mathfrak d_p^{\vis}(\cY_{V_2})=1-\frac1p
\]
for every prime \(p\).  Visibility is locally detectable and the Euler
product of the local densities agrees with the global density.
\end{proposition}

\begin{proof}
By Proposition~\ref{prop:cone-local-global}, global visibility is
equivalent to primitivity of the ambient coordinate vector.  We
therefore begin by determining all primitive positive integral points
on \(V_2\).

Let \((x,y,z)\in\Z_{>0}^3\) satisfy \(xz=y^2\) and
\(\gcd(x,y,z)=1\).  We first claim that \(x\) and \(z\) are relatively
prime.  Suppose otherwise that a prime \(p\) divides both \(x\) and
\(z\).  Then \(p^2\mid xz=y^2\), and hence \(p\mid y\).  Thus \(p\)
would divide \(x,y,z\), contradicting the assumed primitivity.
Consequently \(\gcd(x,z)=1\).

Since \(xz=y^2\) is a square and \(x\) and \(z\) are relatively prime,
each of \(x\) and \(z\) must itself be a square.  Indeed, let \(p\) be
a prime divisor of \(x\).  Since \(p\nmid z\), the exponent of \(p\)
in the factorization of \(xz\) is exactly the exponent of \(p\) in
\(x\).  But \(xz=y^2\) is a square, so this exponent is even.  The
same argument applies to every prime divisor of \(x\), and hence
\(x=a^2\) for a uniquely determined positive integer \(a\).  Similarly,
\(z=b^2\) for a uniquely determined positive integer \(b\).  The
identity \(y^2=a^2b^2\), together with positivity of \(y\), then gives
\(y=ab\).

Moreover, \(a\) and \(b\) are relatively prime.  A common prime divisor
of \(a\) and \(b\) would divide \(a^2,ab,b^2\), again contradicting
the primitivity of \((x,y,z)\).  Thus every primitive positive point is
uniquely of the form
\[
        (a^2,ab,b^2),
        \qquad
        a,b\in\Z_{>0},\quad \gcd(a,b)=1.
\]
Conversely, every coprime positive pair \((a,b)\) gives a primitive
point of \(V_2\).  The equation \(a^2b^2=(ab)^2\) shows that the point
lies on \(V_2\), while a prime dividing all three coordinates would
divide both \(a\) and \(b\).

We next translate the height condition into a condition on the
parameters \(a\) and \(b\).  The max-height of
\((a^2,ab,b^2)\) is \(\max\{a,b\}^2\).  Indeed, \(ab\leq
\max\{a,b\}^2\), while one of \(a^2\) and \(b^2\) is exactly
\(\max\{a,b\}^2\).  Hence the primitive points of height at most \(X\)
are in bijection with the coprime pairs \(a,b\leq X^{1/2}\).

We now count these pairs explicitly.  Put \(T=X^{1/2}\).  By the
M\"obius identity
\(\mathbf 1_{\gcd(a,b)=1}=\sum_{q\mid a,\ q\mid b}\mu(q)\), their
number is
\[
        \sum_{q\leq T}
        \mu(q)
        \left\lfloor\frac{T}{q}\right\rfloor^2.
\]
Using \(\lfloor T/q\rfloor=T/q+O(1)\), this equals
\(T^2\sum_{q\leq T}\mu(q)q^{-2}+O(T\sum_{q\leq T}q^{-1})+O(T)\).
The second term is \(O(T\log T)\).  On the other hand,
\(\sum_{q\geq1}\mu(q)q^{-2}=1/\zeta(2)\), and the tail beyond \(T\)
is \(O(T^{-1})\).  We therefore obtain
\[
        P_{V_2}(X)
        =
        \frac{X}{\zeta(2)}
        +
        O(X^{1/2}\log X).
\]
In particular, \(P_{V_2}(X)\sim X/\zeta(2)\).

Proposition~\ref{prop:cone-density} applies with \(\alpha=1\),
\(\beta=0\), and \(c=1/\zeta(2)\).  Hence
\[
        N_{V_2}(X)
        \sim
        \frac{X\log X}{\zeta(2)}.
\]
The globally visible points are precisely the primitive points, so the
relative visible proportion is asymptotic to \(1/\log X\) and therefore
tends to zero.  This proves
\(\mathfrak d^{\vis}(\cY_{V_2})=0\).

For a fixed prime \(p\), Proposition~\ref{prop:cone-density} gives
\(\mathfrak d_p^{\vis}(\cY_{V_2})=1-1/p\).  It is also useful to see
the local calculation directly.  By
Proposition~\ref{prop:cone-local-global}, a point fails \(p\)-adic
visibility exactly when all three of its coordinates are divisible by
\(p\).  Since \(V_2\) is an ordinary cone, division by \(p\) sends
such a point to another positive integral point of \(V_2\), and the
max-height is divided by \(p\).  Thus the \(p\)-adically invisible
points of height at most \(X\) are in bijection with all branch points
of height at most \(X/p\).  Since \(N_{V_2}(X)\) has growth
\(X\log X\), the quotient \(N_{V_2}(X/p)/N_{V_2}(X)\) tends to \(1/p\).

Exact local detectability follows from
Proposition~\ref{prop:cone-local-global}.  Finally,
\(\prod_p(1-1/p)=0\), so the Euler product of the local densities
agrees with the global density.
\end{proof}

The preceding parametrization is the affine version of the quadratic
Veronese parametrization of the projective conic.  It is worth
emphasizing why this does not reduce the example to the polynomial graph
case.  The map \((a,b)\mapsto(a^2,ab,b^2)\) is polynomial, but neither
parameter occurs as an ambient coordinate.  More importantly, the same
rational ray contains infinitely many integral points obtained by
ordinary scalar multiplication in the ambient coordinates, whereas
multiplication of the parameters by an integer \(q\) multiplies all
three coordinates by \(q^2\).  Thus the relation between parameter
scaling and ambient radial scaling is different from the graph families
treated by Theorem~\ref{thm:sparse-Davenport}.  The primitive
parametrization removes this ambiguity, and Proposition~\ref{prop:cone-density}
then recovers the full counting function by summing over ordinary
ambient dilations.

\subsubsection{Rank-one \(2\times2\) matrices}

We next consider a determinantal example in which the primitive counting
function has a different growth exponent.  Let \(V_{\det}\subseteq
\A^4\) be the hypersurface defined by \(X_1X_4=X_2X_3\).  Writing a
point as
\[
        M=
        \begin{pmatrix}
        x_1&x_2\\
        x_3&x_4
        \end{pmatrix},
\]
the defining equation is exactly the condition \(\det(M)=0\).  Since
we restrict to positive points, no such matrix is zero, and hence
\(V_{\det}\) parametrizes positive rank-one \(2\times2\) matrices.

This is again an ordinary affine cone, since scalar multiplication of
all four matrix entries preserves the rank-one condition.  It is not,
however, a polynomial graph over three unrestricted integral
coordinates.  For example, on the locus \(x_1\neq0\), the final entry
is determined by \(x_4=x_2x_3/x_1\).  Thus an arbitrary triple
\((x_1,x_2,x_3)\) does not give an integral matrix on the variety:
one must additionally require the divisibility condition
\(x_1\mid x_2x_3\).

The rank-one structure gives a more natural parametrization.  Every
rank-one matrix is an outer product of two vectors, and primitivity of
the matrix translates into primitivity of both vectors.  The resulting
count is no longer of order \(X\); instead a logarithm appears because
the height factors as a product of the two vector heights.

\begin{proposition}\label{prop:rank-one-density}
For the positive rank-one determinantal cone,
\[
        \mathfrak d^{\vis}(\cY_{V_{\det}})
        =
        \frac1{\zeta(2)},
        \qquad
        \mathfrak d_p^{\vis}(\cY_{V_{\det}})
        =
        1-\frac1{p^2}.
\]
Visibility is locally detectable and the global density is the Euler
product of the local densities.
\end{proposition}

\begin{proof}
We first establish the parametrization of the primitive points.  Let
\(M\) be a positive integral rank-one matrix.  Its two columns span a
one-dimensional rational subspace of \(\Q^2\).  There is a unique
primitive positive vector \(u=(a,b)^{\mathrm t}\in\Z_{>0}^2\)
spanning this line.

Every integral vector lying on this rational line is in fact an
integral multiple of \(u\).  To see this, suppose
\((r,s)=\lambda(a,b)\in\Z^2\) with \(\lambda\in\Q\).  Since
\(\gcd(a,b)=1\), there exist integers \(A,B\) such that \(Aa+Bb=1\).
It follows that \(\lambda=A(\lambda a)+B(\lambda b)=Ar+Bs\), which is
an integer.  Hence the two columns of \(M\) are \(cu\) and \(du\) for
some positive integers \(c,d\).

Writing \(v=(c,d)^{\mathrm t}\), we therefore have \(M=uv^{\mathrm t}\).
Since \(\gcd(a,b)=1\), the greatest common divisor of the four entries
\(ac,ad,bc,bd\) is precisely \(\gcd(c,d)\).  Thus \(M\) is primitive
if and only if \(v\) is primitive.  We have consequently obtained a
bijection between primitive positive rank-one matrices and ordered
pairs of primitive positive vectors \(u,v\in\Z_{>0}^2\).  The
factorization is unique because the positive primitive generator of the
column space is unique.

Let \(A(T)\) denote the number of primitive positive pairs
\((a,b)\) satisfying \(\max\{a,b\}\leq T\).  M\"obius
inversion gives
\[
        A(T)
        =
        \frac{T^2}{\zeta(2)}
        +
        O(T\log(2T)).
\]

It will be convenient to group these vectors according to their exact
max-height.  Let \(a(m)\) denote the number of primitive positive pairs
\((a,b)\) with \(\max\{a,b\}=m\).  When \(m=1\), the only pair is
\((1,1)\), so \(a(1)=1\).  Suppose \(m\geq2\). This pair is either
of the form \((m,k)\) with \(1\leq k<m\) and \(\gcd(k,m)=1\), or of
the form \((k,m)\) with the same conditions.  There are
\(\varphi(m)\) possibilities for each type and the two types are
disjoint.  Hence \(a(m)=2\varphi(m)\) for \(m\geq2\).

If \(u=(a,b)^{\mathrm t}\) and \(v=(c,d)^{\mathrm t}\), the four
entries in \(uv^{\mathrm t}\) are \(ac,ad,bc,bd\).  Since all
coordinates are positive, their maximum is
\(\|u\|_\infty\|v\|_\infty\).  Grouping the first primitive vector
according to \(m=\|u\|_\infty\) therefore gives the exact identity
\begin{equation}\label{eq:rank-one-primitive-count}
        P_{V_{\det}}(X)
        =
        \sum_{m\leq X}a(m)A(X/m).
\end{equation}

We now evaluate the right-hand side.  Substitution of the asymptotic
for \(A(X/m)\) gives the main contribution
\(X^2\zeta(2)^{-1}\sum_{m\leq X}a(m)m^{-2}\).  It remains to estimate
this weighted sum of \(a(m)\).

We use the classical estimate
\(\sum_{m\leq T}\varphi(m)=T^2/(2\zeta(2))+O(T\log T)\).
For $m\ge 2$, partial summation gives
\[
        \sum_{m\leq X}\frac{a(m)}{m^2}
        =
        \frac{2}{\zeta(2)}\log X+O(1).
\]
We must also check that the errors in the estimate for \(A(X/m)\) do
not affect the leading term.  Their total contribution is bounded by a
constant multiple of
\(X\sum_{m\leq X}a(m)m^{-1}\log(2X/m)\).  Since \(a(m)\leq2m\) for
\(m\geq2\), this is at most a constant multiple of
\(X\sum_{m\leq X}\log(2X/m)\).  By Stirling's formula,
\(\sum_{m\leq X}\log(2X/m)=O(X)\), and hence the accumulated error is
\(O(X^2)\).

These estimates yield
\[
        P_{V_{\det}}(X)
        =
        \frac{2}{\zeta(2)^2}X^2\log X
        +
        O(X^2).
\]
Thus Proposition~\ref{prop:cone-density} applies with \(\alpha=2\),
\(\beta=1\), and \(c=2/\zeta(2)^2\).  We obtain
\[
        N_{V_{\det}}(X)
        \sim
        \frac{2}{\zeta(2)}X^2\log X,
\]
and division of the primitive count by the total count gives
\(\mathfrak d^{\vis}(\cY_{V_{\det}})=1/\zeta(2)\).

For every prime \(p\), Proposition~\ref{prop:cone-density} gives
\(\mathfrak d_p^{\vis}(\cY_{V_{\det}})=1-p^{-2}\).  Exact local
detectability follows from Proposition~\ref{prop:cone-local-global},
and Euler's product gives
\(\prod_p(1-p^{-2})=1/\zeta(2)\), as required.
\end{proof}

This example already shows that the dimension of the sparse variety
does not determine the Euler factor.  The hypersurface \(V_{\det}\)
has dimension \(3\), but the local factor is \(1-p^{-2}\) and the
global density is \(1/\zeta(2)\), rather than \(1-p^{-3}\) and
\(1/\zeta(3)\).  The exponent \(2\) comes instead from the growth of
the primitive counting function.  The two primitive vectors \(u\) and
\(v\) each contribute two-dimensional lattice-point growth, while the
multiplicative height constraint
\(\|u\|_\infty\|v\|_\infty\leq X\) produces the additional logarithmic
factor.  The power of \(X\), rather than the logarithmic factor, is
what enters the local Euler factor.

\subsubsection{The quadratic Veronese cone in three variables}

The preceding quadratic cone arose from the Veronese embedding of
\(\mathbf P^1\).  We now consider the corresponding quadratic
Veronese construction in three variables.  The projective quadratic
Veronese map sends \([a:b:c]\) to the six quadratic monomials
\([a^2:b^2:c^2:ab:ac:bc]\).  Its affine cone can be described as the
space of symmetric \(3\times3\) matrices of rank one.

Thus we identify a point with
\[
        M=
        \begin{pmatrix}
        x_{11}&x_{12}&x_{13}\\
        x_{12}&x_{22}&x_{23}\\
        x_{13}&x_{23}&x_{33}
        \end{pmatrix},
\]
and impose the condition that all \(2\times2\) minors vanish.  These
minor equations cut out the rank-one symmetric cone in \(\A^6\).
Again, this is not a graph over an unrestricted subset of the ambient
coordinates.  The rank-one equations impose several nonlinear
divisibility relations among the six integral entries.

The advantage of the symmetric rank-one description is that primitive
integral points are controlled by primitive integral vectors in
\(\Z^3\).  We record this explicitly.

\begin{lemma}\label{lem:primitive-Veronese}
A primitive positive integral point on the quadratic Veronese cone is
uniquely of the form
\[
        \nu_2(a,b,c)
        =
        (a^2,b^2,c^2,ab,ac,bc),
\]
where \(a,b,c\in\Z_{>0}\) and \(\gcd(a,b,c)=1\).
\end{lemma}

\begin{proof}
Let \(M\) be a positive primitive integral symmetric matrix of rank
one.  Its column space is a one-dimensional rational subspace of
\(\Q^3\).  There is therefore a unique primitive positive integral
vector \(u=(a,b,c)^{\mathrm t}\) spanning this column space.

Since \(M\) has rank one, it can be written as \(uv^{\mathrm t}\) for
some rational vector \(v\).  Symmetry forces the row space and column
space to be the same one-dimensional subspace.  Hence \(v\) is a
rational multiple of \(u\), and consequently \(M=\lambda uu^{\mathrm
t}\) for some \(\lambda\in\Q_{>0}\).

Write \(\lambda=A/B\) in lowest terms.  Integrality of \(M\) implies
that \(B\) divides each of the six integers
\(a^2,b^2,c^2,ab,ac,bc\). Moreover, $\gcd(\nu_2((a,b,c))=1$
Indeed, if a prime \(p\) divided all six, it would in particular divide
\(a^2,b^2,c^2\), and hence divide \(a,b,c\), contradicting
\(\gcd(a,b,c)=1\).  It follows that \(B=1\), so \(\lambda\) is a
positive integer.

The same argument shows that the entries of \(uu^{\mathrm t}\) have gcd one. Therefore, the greatest common divisor of
the entries of \(M=\lambda uu^{\mathrm t}\) is \(\lambda\).  Since
\(M\) is primitive, we must have \(\lambda=1\).  This proves that
\(M=uu^{\mathrm t}\), which gives the six coordinates stated.

Conversely, suppose \(\gcd(a,b,c)=1\).  The matrix \(uu^{\mathrm t}\)
has rank one entries and positive integral entries.  If a prime divided all
six entries, it would divide \(a^2,b^2,c^2\), and hence \(a,b,c\),
which is impossible. Thus, the corresponding point is primitive.
Uniqueness follows from the uniqueness of the positive primitive
generator of the column space.
\end{proof}

\begin{proposition}\label{prop:Veronese-density}
For the quadratic Veronese cone in \(\A^6\),
\[
        \mathfrak d^{\vis}(\cY_{\nu_2})
        =
        \frac1{\zeta(3/2)},
        \qquad
        \mathfrak d_p^{\vis}(\cY_{\nu_2})
        =
        1-p^{-3/2}.
\]
Visibility is locally detectable, and the global density is the Euler
product of the local densities.
\end{proposition}

\begin{proof}
By Lemma~\ref{lem:primitive-Veronese}, primitive positive integral
points are parametrized bijectively by primitive positive triples
\((a,b,c)\).  Put \(T=\max\{a,b,c\}\).  Each of the six coordinates
\(a^2,b^2,c^2,ab,ac,bc\) is at most \(T^2\), while one of the first
three coordinates is exactly \(T^2\). Hence, the maximum height of
\(\nu_2(a,b,c)\) is exactly \(T^2\).

It follows that the primitive points of height at most \(X\) correspond
precisely to primitive positive triples satisfying
\(a,b,c\leq X^{1/2}\).  We count these by M\"obius inversion.  For a
positive real number \(T\), the number of primitive triples in the box
\([1,T]^3\) is
\[
        \sum_{q\leq T}
        \mu(q)
        \left\lfloor\frac{T}{q}\right\rfloor^3.
\]
Expanding the floor gives
\(T^3\sum_{q\leq T}\mu(q)q^{-3}+O(T^2\sum_{q\leq T}q^{-2})+
O(T\log T)+O(T)\).  Since \(\sum q^{-2}\) converges, the total error is
\(O(T^2)\).  Moreover,
\(\sum_{q\geq1}\mu(q)q^{-3}=1/\zeta(3)\), while the tail beyond \(T\)
contributes only \(O(T)\) after multiplication by \(T^3\).  Hence the
number of primitive positive triples is
\(T^3/\zeta(3)+O(T^2)\).

Taking \(T=X^{1/2}\), we obtain
\[
        P_{\nu_2}(X)
        =
        \frac{X^{3/2}}{\zeta(3)}
        +
        O(X).
\]
Thus Proposition~\ref{prop:cone-density} applies with
\(\alpha=3/2\), \(\beta=0\), and \(c=1/\zeta(3)\).  It follows that
\[
        N_{\nu_2}(X)
        \sim
        \frac{\zeta(3/2)}{\zeta(3)}X^{3/2}.
\]
Consequently the relative density of primitive, and therefore globally
visible, points is \(1/\zeta(3/2)\).

For a fixed prime \(p\), the same proposition gives
\(\mathfrak d_p^{\vis}(\cY_{\nu_2})=1-p^{-3/2}\).  The appearance of
the fractional exponent can also be read directly from the counting
function.  Since the total number of points grows like \(X^{3/2}\),
division of an ambient point by \(p\) changes the asymptotic mass of a
height ball by the factor \(p^{-3/2}\).  Proposition~\ref{prop:cone-local-global}
gives exact local detectability, and Euler's product gives
\(\prod_p(1-p^{-3/2})=1/\zeta(3/2)\).
\end{proof}

The fractional exponent in this example deserves emphasis.  Neither
the ambient dimension \(6\) nor the dimension \(3\) of the affine
Veronese cone appears in the Euler factor.  The exponent \(3/2\) comes
from the relation between the natural parameters and the ambient
height.  Primitive triples \((a,b,c)\) of size \(T\) occur with order
of magnitude \(T^3\), while their Veronese images have height \(T^2\).
Thus a height bound \(X\) corresponds to a parameter bound
\(T=X^{1/2}\), producing \(X^{3/2}\) primitive points.  This example
shows in particular that the zeta exponent attached to a sparse
visibility problem need not be integral.

\section{Polynomial branches}
We now turn to nonhomogeneous polynomial families.  
Let $P(T)\in \Z[T]$ be a monic polynomial with $P(0)=0$ and let
$m=A/B\in \Q_{>0}$ with
$(A,B)=1, B>0.$
Consider the branch
\[
\phi_m(t)
=
\left(t,\frac{A}{B}P(t)\right).
\]
For a point on the branch to be integral, the first coordinate $t$ is integral and since $(A,B)=1$, the second coordinate is integral if and only if $B\mid P(x)$.
\begin{definition}
For $B\geq 1$, define
\[
T_B(P):=\{x\in \Z_{>0}:B\mid P(x)\}.
\]
For a prime $p$, write $e_p:=v_p(B)$. For $x\in T_B(P)$, define
\[
\lambda_p(x)
:=
\min\{v_p(x),\,v_p(A)+v_p(P(x))-e_p\}.
\]
\end{definition}
In this notation, we have, for $x\in T_B(P), 
v_p(\phi_m(x))=\lambda_p(x).$




We can now write the global visibility and p-adic visibility in terms of $\lambda_p$. 
Integral points on the branch are precisely the points $\phi_m(x)$ with
$x\in T_B(P)$.  Since the first coordinate is $x$, the order on the branch
is the usual order on positive integers. Thus $\phi_m(x)$ is globally
visible precisely when there is no element $s\in T_B(P)$ with $0<s<x$,
which is equivalent to $x$ being the smallest element of $T_B(P)$. Further, it is $p$-adically visible if and only if
$
\lambda_p(x)\leq \lambda_p(s)$. This can also be used to write a local defect criterion for polynomial branches. 

\begin{corollary}[Local defect criterion for polynomial branches]
Let $x\in T_B(P)$.  Then $\phi_m(x)$ is locally visible at every prime but
not globally visible if and only if there exists $s\in T_B(P)$ with
$0<s<x$, and for every $u\in T_B(P)$ satisfying $0<u<x$ and every prime
$p$, one has
\[
\lambda_p(x)\leq \lambda_p(u).
\]
\end{corollary}


Theorem \ref{thm:polynomial-local-global} is a classification result on integral polynomials for which visibility is locally-detectable. Using the above criterion, applying Schur's theorem, Hensel lifting, and the Chinese Remainder theorem, we construct locally defect sets for polynomial families proving Theorem \ref{thm:polynomial-local-global}.
\begin{proof}[Proof of Theorem \ref{thm:polynomial-local-global}]
If $P$ is a monomial, then the family is locally detectable by Theorem \ref{thm:local-detectability}.
Suppose that \(P(x)\) is not a monomial. Let \(r\geq 1\) be the
smallest index such that \(a_r\neq 0\). Then
\[
P(x)=x^rR(x),
\]
where $R(x)\in \mathbb Z[x]$, $R(0)=a_r>0$, and \(R(x)\) is nonconstant. We shall construct a point which is not globally visible but is
\(p\)-adically visible for every prime \(p\). To construct such points, we use primes relatively prime to $a_r$. To ensure, further global invisibity properties, we impose additional conditions on such primes. We claim that such primes exist.

\smallskip

\noindent\emph{Claim.} There exist a prime \(\ell\nmid a_r\) and an integer
\(c\geq 1\), with $\gcd(c,a_r)=1$,
such that $\ell^r\mid R(c)$.

\smallskip

\noindent\emph{Proof of the Claim.}
Let $H(x)$ be a primitive non constant irreducible factor of
$R(x)$. By Schur's theorem \cite{Schur1912}, infinitely many primes divide
at least one of the values \(H(m)\), with \(m\in\mathbb Z\). Choose a prime $l$ from this set avoiding the bad primes $\{p \ \text{prime}: p|a_ra_d\text{disc}(H)\}$, where $a_d$ is the leading coefficient of $R(x)$.  Then $H$ has a
root modulo $\ell$. Since $\ell\nmid\text{disc}(H)$,
this root is simple. By Hensel's lemma, it lifts to a root $c_0\in \mathbb Z
$ modulo $\ell^r$. Consequently,
$R(c_0)\equiv 0\pmod{\ell^r}$.
Finally, using the Chinese remainder theorem, we choose a positive representative \(c\geq1\) such that $c\equiv c_0\pmod{\ell^r}$ and $c\equiv 1\pmod p$ for every prime \(p\mid a_r\). Then $\ell^r\mid R(c)$ and $\gcd(c,a_r)=1$. This proves the claim.

\smallskip

Now set $u=c\ell$. We show that the point $P=(lc, R(lc))$ is a primitive globally invisible point. Clearly, $u<c$ and both the points $P$ and $Q=(c, R(c)$ lie on the branch $q=1/u^r$. Moreover, since $l^r|R(c)$, $Q=\phi_q(c)
=\left(c,\frac{R(c)}{\ell^r}\right)$ 
is an integral point on the branch. 
We show that $P$ is primitive. 
If $p\mid cl$. Then  $R(c\ell)\equiv R(0)\pmod p$ and since $\gcd(c,a_r)=1$ and $\ell\nmid a_r$, we have $p\nmid R(c\ell)$. Rpimitivity implies that
$v_p(P)=0$
for every prime $p$. Every earlier integral point \(Q'\) on the same branch
satisfies $v_p(Q')\geq 0$.
Hence $v_p(P)\leq v_p(Q')$
for every prime \(p\) and every earlier integral point \(Q'\). Therefore \(P\)
is \(p\)-adically visible for every prime \(p\).

Thus
\[
P\in \bigcap_p \cB_p^{\operatorname{vis}}(\cY_P)\setminus \cB^{\operatorname{vis}}(\cY_P).
\]
Hence local detectability fails whenever $P(x)$ is not a monomial. This
completes the proof.
\end{proof}

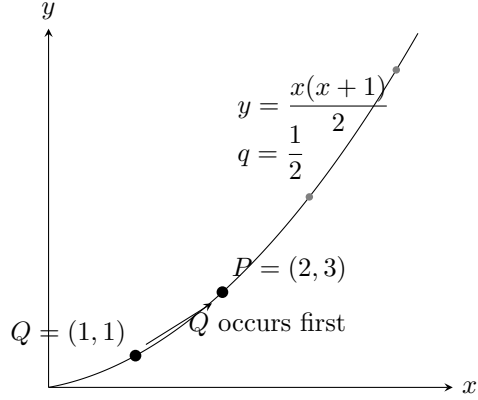
\begin{figure}[t]
\centering
\begin{tikzpicture}[x=1.15cm,y=.42cm,>=stealth]
  \draw[->] (0,0) -- (4.65,0) node[right] {$x$};
  \draw[->] (0,0) -- (0,11.3) node[above] {$y$};

  \draw[domain=0:4.25,samples=100,smooth,variable=\t]
        plot ({\t},{0.5*\t*(\t+1)});

  \foreach \x/\y in {1/1,2/3,3/6,4/10}
       \fill[gray] (\x,\y) circle (1.4pt);

  \fill (1,1) circle (2.2pt)
        node[above left] {$Q=(1,1)$};
  \fill (2,3) circle (2.2pt)
        node[above right] {$P=(2,3)$};

  \draw[->] (1.12,1.35) -- (1.88,2.65)
       node[midway,right] {$Q\text{ occurs first}$};

  \node[align=left] at (3.05,8.2)
       {$y=\dfrac{x(x+1)}2$\\[-1mm]$q=\dfrac12$};
\end{tikzpicture}
\caption{A defect point for \(P(T)=T(T+1)\).  On the branch \(q=1/2\),
the integral point \(Q=(1,1)\) occurs before \(P=(2,3)\), so \(P\) is
not globally visible.  Since \(\gcd(2,3)=1\), the point \(P\) is
nevertheless \(p\)-adically visible for every prime \(p\).}
\label{fig:polynomial-defect}
\end{figure}
Although the local defect set is nonempty, we show that these local-global failures form a sparse subset of $\mathbb Z_{\ge1}^2$, of natural density zero. This is evident directly from the density one result of
Chaubey--Pandey--Regavim \cite{ChaubeyPandeyRegavim2026}. 

\begin{theorem}\label{thm:polynomial-defect-density}
Let
\[
        P(T)=a_dT^d+a_{d-1}T^{d-1}+\cdots+a_1T\in\Z[T],
\]
where $a_d>0$, $a_i\geq0$ and $\gcd(a_d,a_{d-1},\dots,a_1)=1$, and set
\[
\cY_P:=\{(\A^1,\phi_q):q\in\Q_{>0}\},
\qquad
\phi_q(t)=(t,qP(t)).
\]
Then the local defect set
\[
       \cB^{\bad}(\cY_P)
        :=
        \left(\bigcap_p \cB_p^{\operatorname{vis}}(\cY_P)\right)
        \setminus
        \cB^{\operatorname{vis}}(\cY_P)
\]
has density zero relative to \(\cB(\cY_P)\).
\end{theorem}

\begin{proof}
If \(P\) is a monomial, then \(\cY_P\) is a positive-weight homogeneous family,
so local detectability follows from Theorem \ref{thm:local-detectability}. Hence in this case, $\cB^{\bad}(\cY_P)=\emptyset$ and in particular has density $0$.

If \(P\) is not a monomial, then the authors in \cite{ChaubeyPandeyRegavim2026} prove that
\[
        \mathfrak d_{\cB(\cY_P)}\bigl(
        \cB^{\operatorname{vis}}(\cY_P)
        \bigr)
        =
        1.
\]
Thus the non-visible locus $\cB(\cY_P)\setminus
        \cB^{\operatorname{vis}}(\cY_P)$
has density zero. Since
\[
       \cB^{\bad}(\cY_P)
        \subseteq
        \cB(\cY_P)\setminus
        \cB^{\operatorname{vis}}(\cY_P),
\]
the defect set also has density zero.
\end{proof}

\section{Quantitative bounds for the invisible loci and the defect set}
\label{s 6}

\subsection{Bounds for invisible loci}
In the preceding section, we used the authors' work in 
\cite{ChaubeyPandeyRegavim2026} to show that the failure of local
detectability for non-monomial polynomial families is nevertheless a
density-zero phenomenon.  We begin by recalling their stronger
quantitative result.  Let
\[
        P(T)=a_dT^d+a_{d-1}T^{d-1}+\cdots+a_1T\in\Z[T],
        \qquad d\geq2,
\]
where \(a_d>0\), all \(a_i\geq0\), and
\(\gcd(a_d,\dots,a_1)=1\).  As before, let \(\cY_P\) denote the family
\(\phi_q(t)=(t,qP(t))\), with \(q\in\Q_{>0}\), and write
\(U(P):=\cB(\cY_P)\setminus\cB^{\operatorname{vis}}(\cY_P)\) for the
non-visible locus and \(U(P;X):=U(P)\cap[1,X]^2\).

When \(P\) is separable of degree \(d\), it is shown in
\cite[Lemma 2]{ChaubeyPandeyRegavim2026} that, if $P$ is separable, for every
\(\varepsilon>0\), 
\begin{equation}\label{eq:CPR-nonvisible-bound}
        \#U(P;X)
        \ll_{P,\varepsilon}
        X^{\,2-\frac1{2d-1}+\varepsilon}.
\end{equation}

The purpose of this section is to prove the sharpened estimate of Theorem \ref{thm:intro-separable-polynomial}. For \(d\geq3\), the improvement is genuine since
\[
        \frac32+\frac1{2d}
        <
        2-\frac1{2d-1}.
\]
\par The first step is an expression for $\#U(P;X)$ valid without the separability
hypothesis.
\begin{lemma}\label{lem:polynomial-gcd-reduction}
For \(X\geq2\), one has
\begin{equation}\label{eq:polynomial-gcd-reduction}
        \#U(P;X)
        \leq
        X
        \sum_{x\leq X}
        \frac1{P(x)}
        \sum_{1\leq s<x}
        \gcd\bigl(P(x),P(s)\bigr).
\end{equation}
\end{lemma}

\begin{proof}
Fix \(1\leq x,y\leq X\).  The point \((x,y)\) lies on the unique
branch with parameter \(q=y/P(x)\).  It is globally invisible if there is an
integer \(s<x\) for which the earlier point
\(\phi_q(s)=(s,yP(s)/P(x))\) is integral.  This is equivalent to
\(P(x)\mid yP(s)\).

Put \(g=\gcd(P(x),P(s))\), and write \(P(x)=gm\) and \(P(s)=gn\)
with \((m,n)=1\).  The preceding divisibility condition is then
equivalent to \(m\mid y\).  Thus, for fixed \(x\) and \(s\), there are
at most \(Xg/P(x)\) possible values of \(y\).  Taking the union over
\(s<x\) and then summing over \(x\leq X\) gives
\eqref{eq:polynomial-gcd-reduction}.
\end{proof}
Questions on the distribution of expressions concerning common divisors of
polynomial values have been studied from several perspectives.  The
normalized double gcd sum closest to the quantity $\mathcal G_P(X)$ below
appears in the recent work of the authors in
\cite{LP26}. For the distribution and density of
$\gcd(P_1(\mathbf n),\ldots,P_s(\mathbf n))$, see
\cite{PoonenSquarefree}, \cite{WangStanley}, and \cite{BodinDebes}.  Resultant-based questions
concerning the possible size and range of
\(\gcd(F(n),G(n))\) have also been studied in the past. 

We study the following gcd sum:
\begin{equation}\label{eq:def-gcd-average}
        \mathcal G_P(X)
        :=
        \sum_{x\leq X}
        \frac1{P(x)}
        \sum_{s<x}\gcd\bigl(P(x),P(s)\bigr).
\end{equation}
Using Lemma~\ref{lem:polynomial-gcd-reduction}, we have
\(\#U(P;X)\leq X\mathcal G_P(X)\). For \(q\geq1\), let
\[\rho_P(q):=\#\{r\bmod q:P(r)\equiv0\pmod q\}.\]

\begin{lemma}\label{lem:rho-basic-bound}
Suppose that \(P\in\Z[T]\) is separable of degree $d\ge 2$.  There exists
a constant \(C_P>0\) such that
\[
        \rho_P(q)\leq C_Pd^{\omega(q)}
\]
for every \(q\geq1\).
\end{lemma}

\begin{proof} See \cite[Lemma 2.1 (b)]{ChaubeyPandeyRegavim2026}. 
\end{proof}
We next use the following estimates of means of the above counting function. All of these can be obtained using partial summation formulas.

\begin{lemma}
\label{lem:rho-perron}
Let \(P\) be fixed and separable of degree \(d\).  Then
\begin{equation}\label{eq:rho-square-average}
        \sum_{q\leq Z}\rho_P(q)^2
        \ll_P Z(\log(Z))^{d^2-1}
        \quad\text{and}\quad
        \sum_{q\leq Z}\frac{\rho_P(q)^2}{q}
        \ll_P(\log(Z))^{d^2}.
\end{equation}
Moreover, for every \(\varepsilon>0\),
\begin{equation}\label{eq:rho-divisor-bound}
        \sum_{q\mid n}\rho_P(q)
        \ll_{P,\varepsilon}n^\varepsilon.
\end{equation}
\end{lemma}

\begin{proof}
By Lemma~\ref{lem:rho-basic-bound}, there is a constant \(C_P>0\)
such that \(\rho_P(n)\leq C_Pd^{\omega(n)}\) for every \(n\geq1\).
It follows that
\(\rho_P(n)^2\leq C_P^2d^{2\omega(n)}\).  Thus, in order to prove the
first estimate in \eqref{eq:rho-square-average}, we bound
the mean value of \(d^{2\omega(n)}\). 
Let $\tau_k(n)$ be the $k$- fold divisor function, then 
one has,
\[
\tau_k(p^\nu)=\binom{\nu+k-1}{k-1}\geq
\binom{k}{k-1}
=
k,
\]
see \cite[Chapter 13 (13.3)]{Koukoulopoulos}. Therefore, by multiplicativity,
$
\tau_k(n)
\geq
\prod_{p\mid n} k
=
k^{\omega(n)}.
$ 
Now, for $k\ge 2$, we have classically
\[
\sum_{n\leq x} \tau_k(n)
\sim
\frac{x(\log x)^{k-1}}{(k-1)!},
\]
see \cite[Chapter XII]{Titchmarsh}.
Combining all of this with
\(\rho_P(n)^2\ll_Pd^{2\omega(n)}\), we obtain the first estimate:
\[
        \sum_{q\leq Z}\rho_P(q)^2
        \ll_P
        Z(\log(Z))^{d^2-1}.
\]
Furthermore, using this and the partial summation formula, one obtains the estimate of the weighted partial sum in \eqref{eq:rho-square-average}.
It remains to prove \eqref{eq:rho-divisor-bound}. 
For this, we write \(n=\prod_{p^\nu\Vert n}p^\nu\). Then,  $
        \sum_{q\mid n}d^{\omega(q)}
        =
        \prod_{p^\nu\Vert n}(1+d\nu)\leq \prod_{p^\nu\Vert n}(\nu+1)^d=\tau(n)^d.\ll_{d,\epsilon} n^\epsilon$
for every $\epsilon>0$. 
\end{proof}

Large common divisors of \(P(x)\) and \(P(s)\) lead to equations in
which the two polynomial values have a rational ratio.  We now obtain
the uniform estimates for these curves that are needed in the proof of
Theorem~\ref{thm:intro-separable-polynomial}.

\begin{proposition}
\label{prop:generic-ratio-irreducibility}
Let \(P\in\C[T]\) be a separable polynomial of degree \(d\geq2\), and
let
\[
        \mathcal V_P
        :=
        \{P(\alpha):P'(\alpha)=0\}
\]
be its set of finite critical values.  Put
\[
        \mathscr E_P
        :=
        \left\{
        \frac{\beta}{\gamma}:
        \beta,\gamma\in\mathcal V_P
        \right\}
        \subset\C^\times.
\]
If \(c\in\C^\times\setminus\mathscr E_P\), then
\(P(S)-cP(X)\) is irreducible in \(\C[S,X]\).
\end{proposition}

\begin{proof}
We apply \cite[Proposition~3.1(1)]{PakovichSeparatedVariables}.
In the notation of that result, \(C(F)\) denotes the set of all
critical values of the rational function \(F:\mathbf P^1_\C\to
\mathbf P^1_\C\).

For the polynomial \(P\), the finite critical values are precisely
the elements of \(\mathcal V_P\), while \(\infty\) is also a critical
value because \(P\) is totally ramified at infinity.  Thus $C(P)=\mathcal V_P\cup\{\infty\}$ and similarly, $C(cP)=c\mathcal V_P\cup\{\infty\}$. We claim that
\(\mathcal V_P\cap c\mathcal V_P=\varnothing\).
Indeed, suppose that \(u\) belongs to this intersection.  Then there
exist \(\beta,\gamma\in\mathcal V_P\) such that
\(u=\beta=c\gamma\).  Since \(P\) is separable, \(0\) cannot be a
critical value of \(P\): otherwise \(P(\alpha)=P'(\alpha)=0\) for
some \(\alpha\), and \(\alpha\) would be a multiple root of \(P\).
Hence \(\gamma\neq0\), and therefore
\(c=\beta/\gamma\in\mathscr E_P\), contrary to the hypothesis.

It follows that
\[
        C(P)\cap C(cP)=\{\infty\}.
\]
Thus \(P\) and \(cP\) have exactly one common critical value.
Proposition~3.1(1) of \cite{PakovichSeparatedVariables} therefore
implies that the curve
\[
        P(S)-cP(X)=0
\]
is irreducible.  
\end{proof}

For the finitely many exceptional ratios we need to only rule out linear
components.
The argument for this is essentially that of Lobsenz--Phillips
\cite[Proposition~3.3]{LP26}. Their statement treats
$P\in\mathbb Z[T]$ and a multiplier $c=r/s\in\mathbb Q\cap(0,1)$.
The applications in the following are also entirely in this arithmetic setting.
Nevertheless, we state the lemma for $P\in\mathbb C[T]$ and
$c\in\mathbb R_{>0}\setminus\{1\}$, since the same root-permutation
argument works over $\mathbb C$.
\begin{lemma}
\label{lem:no-linear-ratio-component}
Let \(P\in\mathbb{C}[T]\) have at least two distinct roots, and let
$c\in\mathbb R_{>0}\setminus\{1\}$ with $c\neq1$.  Then the polynomial
\(P(S)-cP(X)\) has no linear factor in \(\mathbb{C}[S,X]\).
\end{lemma}

\begin{proof}
We argue by contradiction.  Suppose that \(P(S)-cP(X)\) has a complex linear
factor.  The corresponding irreducible component of the affine curve
\(P(S)=cP(X)\) is therefore a line.

We first observe that this line can be neither vertical nor horizontal.
Indeed, if it were vertical, say \(S=S_0\), then
\(P(S_0)-cP(X)\) would vanish identically as a polynomial in \(X\).
This is impossible because \(c\neq0\) and \(P\) is nonconstant.
Similarly, a horizontal component \(X=X_0\) would force
\(P(S)-cP(X_0)\) to vanish identically as a polynomial in \(S\),
which is again impossible.

Thus the line has nonzero coefficients of both \(S\) and \(X\), and
after solving for \(S\) it may be written in the form
\(S=\lambda X+\mu\), with \(\lambda\in\mathbb{C}^{\times}\) and
\(\mu\in\mathbb{C}\).  Since this line is contained in the curve
\(P(S)-cP(X)=0\), substituting \(S=\lambda X+\mu\) gives the
polynomial identity
\begin{equation}\label{eq:affine-symmetry-P}
        P(\lambda X+\mu)=cP(X).
\end{equation}

Let \(R(P)\) denote the set of distinct roots of \(P\), and define the
affine map \(L:\mathbb{C}\to\mathbb{C}\) by
\(L(z)=\lambda z+\mu\).  If \(\alpha\in R(P)\), then
\(P(\alpha)=0\), and \eqref{eq:affine-symmetry-P} gives
\(P(L(\alpha))=cP(\alpha)=0\).  Hence \(L(\alpha)\in R(P)\).
Thus \(L\) maps the finite set \(R(P)\) into itself.

Since \(\lambda\neq0\), the affine map \(L\) is injective.  Its
restriction to the finite set \(R(P)\) is therefore injective, and
hence bijective.  Consequently \(L\) acts as a permutation of the
roots of \(P\).  Since \(R(P)\) is finite, some positive power of this
permutation is the identity.  Thus there exists an integer \(m\geq1\)
such that \(L^m(\alpha)=\alpha\) for every \(\alpha\in R(P)\). By hypothesis, \(P\) has at least two distinct roots.  Choose
\(\alpha,\beta\in R(P)\) with \(\alpha\neq\beta\).  The affine map
\(L^m\) fixes both \(\alpha\) and \(\beta\).  An affine transformation
of \(\mathbb{C}\) which fixes two distinct points must be the identity:
if \(L^m(z)=az+b\), then the equations
\(a\alpha+b=\alpha\) and \(a\beta+b=\beta\) imply
\((a-1)(\alpha-\beta)=0\), hence \(a=1\), and then \(b=0\).
Therefore $L^m=\operatorname{id}_{\mathbb{C}}$.

\par On the other hand, iterating \eqref{eq:affine-symmetry-P} gives
\(P(L^j(X))=c^jP(X)\) for every \(j\geq1\).  Taking \(j=m\) and using
\(L^m=\operatorname{id}\), we obtain $P(X)=c^mP(X)$. Since \(P\) is not the zero polynomial, it follows that \(c^m=1\).
\par Finally \(c\) is a positive real number, so the only possibility is
\(c=1\).  This contradicts the hypothesis \(c\neq1\). Hence \(P(S)-cP(X)\) cannot have a linear factor.
\end{proof} 
We shall use the following standard form of the Bombieri--Pila
estimate.

\begin{lemma}[Bombieri--Pila]
\label{lem:bombieri-pila}
Let \(F\in\mathbb{C}[X,Y]\) be absolutely irreducible of degree \(m\geq2\).
Then, for every \(\varepsilon>0\),
\begin{equation}\label{eq:BP-bound}
        \#\{(x,y)\in[1,M]^2\cap\Z^2:F(x,y)=0\}
        \ll_{m,\varepsilon}M^{1/m+\varepsilon}.
\end{equation}
If \(F\) has degree at most \(d\) and has no linear factor, then the
same counting function is \(O_{d,\varepsilon}(M^{1/2+\varepsilon})\).
\end{lemma}

\begin{proof}
The first assertion is the theorem of Bombieri--Pila
\cite{BombieriPila}, in its coefficient-uniform form.  For the second,
factor \(F\) into irreducible components.  Since there are no linear
components, every component has degree at least two, and
\eqref{eq:BP-bound} applied to each component gives the required
\(M^{1/2+\varepsilon}\) estimate.  This componentwise argument for
separated-variable curves also appears in
\cite[Lemmas~2.4--2.6]{WangXu}.
\end{proof}
We are now ready to apply these bounds to ratio curves. 
Rather than applying the no-linear-factor estimate uniformly to all ratio curves, we use critical-value criterion of Proposition \eqref{prop:generic-ratio-irreducibility} to prove absolute irreducibility for every ratio outside a finite exceptional set. This gives a full-degree Bombieri–Pila bound for the generic ratios, while the componentwise $M^{1/2+\varepsilon}$ estimate is required only for finitely many exceptional ratios.

For relatively prime integers \(E>K\geq1\), put
\[
        N_{E,K}(M)
        :=
        \#\{(s,x)\in\Z_{\geq1}^2:s<x\leq M,\ EP(s)=KP(x)\}.
\]

\begin{proposition}\label{prop:ratio-curve-count}
Let \(P\in\Z[T]\) be fixed and separable of degree \(d\geq2\).
For every \(\varepsilon>0\), one has
\[
        N_{E,K}(M)
        \ll_{d,\varepsilon}M^{1/d+\varepsilon}
\]
whenever \(K/E\notin\mathscr E_P\).  There are only finitely many
reduced positive ratios \(K/E\in(0,1)\) belonging to
\(\mathscr E_P\), and for each of them
\(N_{E,K}(M)\ll_{P,\varepsilon}M^{1/2+\varepsilon}\).
\end{proposition}

\begin{proof}
For a nonexceptional ratio,
Proposition~\ref{prop:generic-ratio-irreducibility} shows that
\(EP(S)-KP(X)\) is absolutely irreducible of degree \(d\), and
Lemma~\ref{lem:bombieri-pila} gives the first estimate.

Since \(\mathscr E_P\) is finite, there are only finitely many
exceptional reduced positive rational ratios.  For such a ratio we
have \(K/E\neq1\), so Lemma~\ref{lem:no-linear-ratio-component}
shows that the corresponding curve has no linear component.
The second part of Lemma~\ref{lem:bombieri-pila} then gives the
second estimate.
\end{proof}

\begin{proof}[Proof of Theorem~\ref{thm:intro-separable-polynomial}]
Recall that
\begin{equation}\label{eq:def-gcd-average-recalled}
        \mathcal G_P(X)
        :=
        \sum_{x\leq X}
        \frac{1}{P(x)}
        \sum_{1\leq s<x}
        \gcd\bigl(P(x),P(s)\bigr).
\end{equation}
Lemma~\ref{lem:polynomial-gcd-reduction} gives
\(\#U(P;X)\leq X\mathcal G_P(X)\).  It is therefore enough to prove
\begin{equation}\label{eq:general-gcd-target}
        \mathcal G_P(X)
        \ll_{P,\varepsilon}
        X^{\frac12+\frac1{2d}+\varepsilon}.
\end{equation}

We estimate \(\mathcal G_P(X)\) on dyadic intervals.  For \(M\geq1\),
put
\[
        \mathcal G_P(M,2M)
        :=
        \sum_{M<x\leq2M}
        \frac{1}{P(x)}
        \sum_{1\leq s<x}
        \gcd\bigl(P(x),P(s)\bigr).
\]
Since
\(P(T)=a_dT^d+\cdots+a_1T\), with \(a_d>0\) and all \(a_i\geq0\),
there are positive constants \(c_P,C_P\), depending only on \(P\),
such that
$c_PM^d\leq P(x)\leq C_PM^d
        \quad \text{for}\quad M<x\leq2M$. 
 In particular, \(P(x)\asymp_PM^d\) uniformly on the
dyadic interval.
Using the Euler totient function, we express 
\begin{equation}\label{eq:general-gcd-dyadic}
 \mathcal G_P(M,2M)
 =
 \sum_{M<x\leq2M}
 \frac1{P(x)}
 \sum_{s<x}
 \sum_{\substack{q\mid P(x)\\q\mid P(s)}}
 \varphi(q).
\end{equation}
We divide the common divisors \(q\) into three ranges using a real parameter $Q$ satisfying
$M<Q<M^d,$
whose value will be chosen later: 
\begin{equation}\label{eq:GPM-three-range-decomposition}
        \mathcal G_P(M,2M)
        \ll_P
        \Sigma_{\mathrm{small}}(M)
        +
        \Sigma_{\mathrm{mid}}(M)
        +
        \Sigma_{\mathrm{large}}(M),
\end{equation}
where
\begin{equation}\label{eq:Sigma-small-definition}
\Sigma_{\mathrm{small}}(M)
:=
M^{-d}
\sum_{M<x\leq2M}
\sum_{1\leq s<x}
\sum_{\substack{
        q\mid P(x),\ q\mid P(s)\\
        q\leq M}}
\varphi(q),
\end{equation}
\begin{equation}\label{eq:Sigma-mid-definition}
\Sigma_{\mathrm{mid}}(M)
:=
M^{-d}
\sum_{M<x\leq2M}
\sum_{1\leq s<x}
\sum_{\substack{
        q\mid P(x),\ q\mid P(s)\\
        M<q\leq Q}}
\varphi(q), and
\end{equation}
\begin{equation}\label{eq:Sigma-large-definition}
\Sigma_{\mathrm{large}}(M)
:=
M^{-d}
\sum_{M<x\leq2M}
\sum_{1\leq s<x}
\sum_{\substack{
        q\mid P(x),\ q\mid P(s)\\
        q>Q}}
\varphi(q).
\end{equation}

We treat these ranges using different arguments. We first consider \(q\leq M\). 
The integers $n$ satisfying $q\mid P(n)$ lie in precisely
$\rho_P(q)$ residue classes modulo $q$.  Each residue class
contains at most $2M/q+1$ integers in the interval $1\leq n\leq
2M$.  Since $q\leq M$, this is $O(M/q)$.  Thus
$\#\{n\leq2M:q\mid P(n)\}
        \ll
        \frac{M}{q}\rho_P(q).$
For a fixed \(q\), the number of pairs \((s,x)\) occurring in
\eqref{eq:general-gcd-dyadic} for which \(q\) divides both
\(P(s)\) and \(P(x)\) is therefore at most the square of this
quantity.  Using \(1/P(x)\ll_PM^{-d}\) and
\(\varphi(q)\leq q\), we obtain
\[
 \Sigma_{\mathrm{small}}(M)
 \ll_P
 M^{-d}
 \sum_{q\leq M}
 q
 \left(\frac{M\rho_P(q)}q\right)^2
 =
 M^{2-d}
 \sum_{q\leq M}\frac{\rho_P(q)^2}{q}.
\]
Lemma~\ref{lem:rho-perron} now gives
\begin{equation}\label{eq:general-small}
        \Sigma_{\mathrm{small}}(M)
        \ll_P
        M^{2-d}(\log(M))^{d^2}.
\end{equation}

We next consider \(M<q\leq Q\).  Since \(q>M\), a fixed residue
class modulo \(q\) occurs at most twice among the integers
\(1,\ldots,2M\).  Hence
\[
        \#\{s\leq2M:q\mid P(s)\}
        \ll \rho_P(q).
\]
For each fixed \(x\), we may therefore bound the number of possible
\(s<x\) with \(q\mid P(s)\) by \(O(\rho_P(q))\).  Using again
\(1/P(x)\ll_PM^{-d}\), together with
\(\varphi(q)\leq q\leq Q\) and \eqref{eq:rho-divisor-bound}, gives
\begin{equation}
\label{eq:general-middle}
\Sigma_{\mathrm{mid}}(M)
 \ll_P
 QM^{-d}
 \sum_{M<x\leq2M}
 \sum_{q\mid P(x)}
 \rho_P(q)\ll_{P,\varepsilon}
        QM^{1-d+\varepsilon}
\end{equation}
\par It remains to treat the range \(q>Q\).  Although
\(\Sigma_{\mathrm{large}}(M)\) was defined using the factor
\(M^{-d}\), on the dyadic interval \(M<x\leq 2M\) we have
\(P(x)\leq C_PM^d\).  Hence,
we shall use this form of the large-divisor contribution:
\[
 \Sigma_{\mathrm{large}}(M)
 \ll_P
 \sum_{M<x\leq2M}\frac1{P(x)}
 \sum_{s<x}
 \sum_{\substack{q\mid P(x),\ q\mid P(s)\\q>Q}}
 \varphi(q).
\]
Fix a triple \((s,x,q)\) occurring in the preceding sum, and set
$e:=\frac{P(x)}q>0; \ 
        k:=\frac{P(s)}q>0.$
 Moreover $P$ is strictly increasing on
$\R_{>0}$, so $s<x$ implies $P(s)<P(x)$, and hence $k<e$, and $\frac{\varphi(q)}{P(x)} \leq \frac {q}{P(x)} = 1/e.$
The definitions of $e$ and $k$ give
\[
        eP(s)=kP(x).
\]
Since \(q>Q\) and \(P(x)\leq C_PM^d\), we also have
\(e\leq C_PM^d/Q\).  Put
\[
        L:=C_P\frac{M^d}{Q}.
\]

For positive integers \(a>b\), let
\[
 N_{a,b}(Y)
 :=
 \#\{(s,x)\in\Z_{\geq1}^2:
        s<x\leq Y,\ aP(s)=bP(x)\}.
\]
Thus every triple \((s,x,q)\) occurring in the large-divisor sum
determines integers \(1\leq k<e\leq L\) and a pair counted by
\(N_{e,k}(2M)\), with weight at most \(1/e\).  Moreover, for fixed tuple $(e,x)$, the divisor \(q=P(x)/e\) is uniquely determined, so
this passage introduces no additional multiplicity.  Conversely,
\(N_{e,k}(2M)\) may count pairs for which \(P(x)/e\) is not an
integer or does not give a divisor \(q>Q\); allowing such pairs only
enlarges the sum. Therefore
\begin{equation}\label{eq:large-ratio-sum}
        \Sigma_{\mathrm{large}}(M)
        \ll_P
        \sum_{e\leq L}\frac1e
        \sum_{1\leq k<e}
        N_{e,k}(2M).
\end{equation}
Without loss of generality, we assume here that $\gcd(e,k)=1$ since, if $g=\gcd(e,k), e=gE, k=gK$, where $\gcd(E,K)=1$ and also $EP(s)=KP(x)$.
Hence
\[
        N_{e,k}(2M)=N_{E,K}(2M).
\]
Suppose first that \(k/e\notin\mathscr E_P\).  By
Proposition~\ref{prop:ratio-curve-count},
\[
        N_{E,K}(2M)
        \ll_{P,\varepsilon}
        M^{1/d+\varepsilon},
\]
uniformly for primitive integers $e>k\geq1$.
The contribution to \eqref{eq:large-ratio-sum} from all such
nonexceptional ratios is therefore at most
\[
 M^{1/d+\varepsilon}
 \sum_{e\leq L}\frac1e\sum_{k<e}1
 \ll
 LM^{1/d+\varepsilon}
\]
Substituting \(L\),
we obtain
\begin{equation}\label{eq:large-generic}
        \Sigma_{\mathrm{large}}^{\mathrm{gen}}(M)
        \ll_{P,\varepsilon}
        \frac{M^{d+1/d+\varepsilon}}{Q}.
\end{equation}
We now consider the exceptional ratios.  Since \(\mathscr E_P\) is
finite, there are only finitely many reduced positive rational
numbers in
\(\mathscr E_P\cap(0,1)\).  Write them as
\(k_j/e_j\), for \(1\leq j\leq r\), with
\(\gcd(e_j,k_j)=1\) and \(e_j>k_j\).
Proposition
\ref{prop:ratio-curve-count} gives
\[
        N_{e,k}(2M)
        =
        N_{e_j,k_j}(2M)
        \ll_{P,\varepsilon}
        M^{1/2+\varepsilon}.
\]
The total contribution we have is
therefore
\[
 \ll_{P,\varepsilon}
 M^{1/2+\varepsilon}
 \sum_{g\leq L/E_j}\frac1{gE_j}
 \ll_{P,\varepsilon}
 M^{1/2+\varepsilon}\log(L).
\]
There are only finitely many exceptional ratios, with their number
depending only on \(P\).  We therefore
obtain
\begin{equation}\label{eq:large-exceptional}
        \Sigma_{\mathrm{large}}^{\mathrm{exc}}(M)
        \ll_{P,\varepsilon}
        M^{1/2+\varepsilon}.
\end{equation}
Combining the generic \eqref{eq:large-generic} and exceptional parts \eqref{eq:large-exceptional} gives
\begin{equation}\label{eq:general-large}
        \Sigma_{\mathrm{large}}(M)
        \ll_{P,\varepsilon}
        \frac{M^{d+1/d+\varepsilon}}{Q}
        +
        M^{1/2+\varepsilon}.
\end{equation}
We collect estimates from \eqref{eq:general-small}, \eqref{eq:general-middle} and \eqref{eq:general-large} and choose $Q=M^{d-\frac12+\frac1{2d}}< M^d$.
We conclude that
\begin{equation}\label{eq:dyadic-gcd-final}
        \mathcal G_P(M,2M)
        \ll_{P,\varepsilon}
        M^{\frac12+\frac1{2d}+\varepsilon}.
\end{equation}

Finally, decompose the range \(x\leq X\) into dyadic intervals.  Apart
from finitely many initial values of \(x\), which contribute
\(O_P(1)\), we have
\[
        \mathcal G_P(X)
        \leq
        \sum_{\substack{M=2^j\\M\leq X}}
        \mathcal G_P(M,2M).
\]
Since \(1/2+1/(2d)+\varepsilon>0\), the resulting geometric sum is
dominated by its largest term.  Using
\eqref{eq:dyadic-gcd-final}, we obtain
\[
        \mathcal G_P(X)
        \ll_{P,\varepsilon}
        X^{\frac12+\frac1{2d}+\varepsilon},
\]
which is \eqref{eq:general-gcd-target}.

Lemma~\ref{lem:polynomial-gcd-reduction} now gives
\[
        \#U(P;X)
        \leq
        X\mathcal G_P(X)
        \ll_{P,\varepsilon}
        X^{\frac32+\frac1{2d}+\varepsilon}.
\]
This completes the proof.
\end{proof}

\subsection{Bounds for the defect set}
\label{subsec:quantitative-defect}

In this section, we prove Theorem~\ref{thm:intro-quantitative-defect}.  Throughout, let \(D_P(X)\) denote the number of points of
\(\cB^{\bad}(\cY_P)\) lying in \([1,X]^2\).  For the general lower
bound, we assume \(P\) is of degree $d\geq 2$, with non-negative coefficients and $P(T)\neq T^d$.

\begin{lemma}
\label{lem:scale-primitive-defect}
Suppose that \((u,v)\) is a primitive globally invisible point for the
polynomial family associated to \(P\).  Then
\(D_P(X)\gg_{P,u,v}X\).
\end{lemma}

\begin{proof}
For a primitive globally invisible lattice point $(u,v)$, we observe that for every integer $a$ with $\gcd(a,u)=1$, the point $(u,av)$ is also a primitive and globally invisible point. Primitivity is clear. For $(u,av)$ to be globally invisible, note that if for some $s<u$, $(s,qP(s))$ is an integral point before $(u,v)$ on the branch with parameter $q$, then the point $(s,aqP(s))$ is the integral point before $(u,av)$ forcing invisibility. Hence, for every fixed $(u,v)$, the point $(u,av)$ with $(u,v)=1$ is also a defect point. This gives
\[D_{P}(X)\ge \#\{a\le \frac{X}{v}: \gcd(a,u)=1\}\gg_{P}\frac{\phi(u)}{uv}X\] proving the assertion. 
\end{proof}

\begin{proposition}
\label{prop:general-linear-defect-lower}
If $P$ is not a monomial and degree $P \ge 2$,
then
\[
        D_P(X)\gg_P X.
\]
\end{proposition}

\begin{proof}
By the construction in the proof of Theorem 1.6, there exists a primitive globally invisible point $(u,v)\in\cB(\cY_P)$. Applying Lemma \ref{lem:scale-primitive-defect} to this point gives $D_P(X)\gg_P X$.
\end{proof}

For a single primitive globally invisible point, Lemma \ref{lem:scale-primitive-defect}
produces $\gg X$ defect points.  For $
P(T)=T^r(m+nT),$ we get an improvement over this.
We state this in Proposition \ref{prop:linear-cofactor-log-lower} below. 
\begin{proposition}
\label{prop:linear-cofactor-log-lower}
Let \(P(T)=T^r(m+nT)\), where \(r\geq1\), \(m,n\geq1\), and
\((m,n)=1\).  Then
\[
        D_P(X)\gg_P X\log X.
\]
\end{proposition}
In proving this, we use the following asymptotic identity.
\begin{lemma}
\label{lem:defect-totient-progression}
Let \(H\geq1\) be fixed and let \(a\) satisfy \((a,H)=1\).  Then
\[
 \sum_{\substack{c\leq Y\\c\equiv a\pmod H}}
 \frac{\varphi(c)}{c^2}
 =
 \kappa_H\log Y+O_H(1),
\]
where
\(\kappa_H=H^{-1}\zeta(2)^{-1}
\prod_{p\mid H}(1-p^{-2})^{-1}>0\).
\end{lemma}

\begin{proof}
We use the identity
\(\varphi(c)/c=\sum_{d\mid c}\mu(d)/d\).  Since
\(c\equiv a\pmod H\) and \((a,H)=1\), every \(c\) occurring in the
sum is coprime to \(H\). Consequently, for a divisor $d|c, (d,H)=1$. Let $\overline{d}$ denote the inverse of $d$ modulo $H$. Then,
\begin{align*}
    \sum_{\substack{c\leq Y\\c\equiv a\pmod H}}
 \frac{\varphi(c)}{c^2}&=\sum_{\substack{d\leq Y\\(d,H)=1}}
 \frac{\mu(d)}{d^2}
 \sum_{\substack{k\leq Y/d\\
                 k\equiv a\overline d\pmod H}}
 \frac1k=\frac{\log Y}{H}\sum_{\substack{d\leq Y\\(d,H)=1}}
 \frac{\mu(d)}{d^2}+O_H(1)\\&=\frac{1}{H\zeta(2)}\prod_{p\mid H}(1-p^{-2})^{-1}\log Y+O_H(1).
\end{align*}
\end{proof}
\begin{proof}[Proof of Proposition \ref{prop:linear-cofactor-log-lower}]
We fix a prime \(\ell\nmid mn\). As in the proof of
Theorem \ref{thm:polynomial-local-global}, we write
$
H:=\ell^r\prod_{p\mid m}p.
$.
Since $(n,l^r)=1$, the Chinese remainder
theorem gives a reduced residue class $c_0\pmod H$, such that every
positive integer \(c\equiv c_0\pmod H\) satisfies
\begin{equation}\label{eq:admissible-c-log}
\ell^r\mid m+nc,
\qquad
(c,m\ell)=1.
\end{equation}
For every such \(c\), set
$
u_c:=\ell c,  
v_c:=m+n\ell c.$
The construction used in the proof of Theorem \ref{thm:polynomial-local-global}, shows that the point \((u_c,v_c)\) is a primitive globally invisible
point for the family associated to
$
P(T)=T^r(m+nT).
$
Indeed, the earlier integral point occurs at the parameter \(c\), and
\eqref{eq:admissible-c-log} is precisely the required integrality and
coprimality condition.
By the amplification argument of Lemma \ref{lem:scale-primitive-defect}, for every
\(\alpha\geq1\) with
$
(\alpha,u_c)=1,
$
the point
$
\bigl(u_c,\alpha v_c\bigr)
=
\bigl(\ell c,\alpha(m+n\ell c)\bigr)\in B^{\mathrm{bad}}(Y_P)
$. Observe that distinct pairs \((c,\alpha)\)
give distinct lattice points.
Set \(Y:=X^{1/2}\). For sufficiently large \(X\), all the first
coordinates \(\ell c\), with \(c\leq Y\), are at most \(X\).
Consequently,
\[
D_P(X)\geq \mathcal N(X),
\]
where
\begin{align*}
    \mathcal N(X)
&:=
\sum_{\substack{c\leq Y\\c\equiv c_0\;(\mathrm{mod}\;H)}}
\#\left\{
\alpha\leq \frac{X}{m+n\ell c}:
(\alpha,\ell c)=1
\right\}=X
\sum_{\substack{c\leq Y\\c\equiv c_0\;(\mathrm{mod}\;H)}}
\frac{\varphi(\ell c)}
     {\ell c\,(m+n\ell c)}
+
O\left(\sum_{c\leq Y}\tau(\ell c)\right)
\\&=
\frac{1-\ell^{-1}}{n\ell}\,
X
\sum_{\substack{c\leq Y\\c\equiv c_0\;(\mathrm{mod}\;H)}}
\frac{\varphi(c)}{c^2}
+
O_P\!\left(X+Y\log(2Y)\right)\\&=
\frac{1-\ell^{-1}}{2n\ell}\,
\kappa_H X\log X
+
O_P(X),
\end{align*}
where we use Lemma \ref{lem:defect-totient-progression} and since $(c,\ell)=1,
\varphi(\ell c)
=(l-1)\varphi(c),
$
and
$
(m+n\ell c)^{-1}
=(nlc)^{-1}+O_P(c^{-2})
$.  
The leading coefficient is positive. Therefore, for all sufficiently
large \(X\),
\[
D_P(X)\geq \mathcal N(X)\gg_P X\log X.
\]
\end{proof}
We now turn our attention to upper bounds for the defect set.  
The argument proceeds in three steps. Lemma \ref{lem:defect-small-divisor-count} below provides a uniform
congruence estimate for the small-divisor range. We then deduce Proposition \ref{prop:primitive-defect-upper}
in which we use primitivity to force the repeated polynomial factor into the
common divisor, and then treat the remaining small and large
divisors by congruence counting and fixed-ratio curve estimates,
respectively. Finally, Proposition \ref{prop:defect-upper-linear-cofactor} reduces an arbitrary defect
point to a primitive globally invisible point and sums the resulting uniform
bounds over the common divisors.

For positive coprime integers \(M,N\),
let \(Q_{M,N}(T):=T^r(M+NT)\).  Let \(V_{M,N}(Y)\) denote the
number of primitive globally invisible points in \([1,Y]^2\) for the
polynomial family associated to \(Q_{M,N}\).

\begin{lemma}
\label{lem:defect-small-divisor-count}
Fix \(r\geq2\) and \(m\geq1\).  Let \(M\mid m\) and
\((M,N)=1\).  Uniformly in the positive integers \(x\) and \(t\),
\[
 \#\{\,1\leq s<x:x^rt\mid Q_{M,N}(s)\,\}
 \ll_m\tau(x).
\]
\end{lemma}
\begin{proof} For a given $s$, 
let 
$a=(x,s), x=au, s=av,$
where \((u,v)=1\) and \(1\leq v<u\). Since $
x^rt\mid Q_{M,N}(s)=s^r(M+Ns),$
cancelling \(a^r\) gives
$
u^rt\mid v^r(M+Nav).$
In particular,
$
u^r\mid v^r(M+Nav).
$
As \((u,v)=1\), it follows that
$
u^r\mid M+Nav.$ 
For a fixed integer \(a\),  \(u=x/a\) is fixed. Therefore, we count integers
\(v\) with \(1\leq v<u\) satisfying
\[
Na\,v\equiv-M\pmod{u^r}.
\]
Let $
g=(Na,u^r).$
The preceding congruence is solvable only if \(g\mid M\); in that
case its solutions form a single residue class modulo \(u^r/g\).
Since \(M\mid m\), we have \(g\leq m\), and hence the number of
solutions with \(1\leq v<u\) is
\[
\ll 1+\frac{u}{u^r/g}
=
1+\frac{g}{u^{r-1}}
\ll_m 1,
\]
because \(r\geq2\).
Summing over the \(\tau(x)\) possible divisors \(a\mid x\) gives
\[
\#\left\{
1\leq s<x:x^rt\mid Q_{M,N}(s)
\right\}
\ll_m \tau(x).
\]
\end{proof}
We now count the primitive invisible points. 
\begin{proposition}
\label{prop:primitive-defect-upper}
Fix \(r\geq2\) and \(m\geq1\).  Uniformly for positive coprime
integers \(M,N\) with \(M\mid m\), one has
\[
        V_{M,N}(Y)
        \ll_{r,m,\varepsilon}
        N^\varepsilon Y^{7/4+\varepsilon}.
\]
If \(r=2\), the exponent \(7/4\) may be replaced by \(5/3\).
\end{proposition}

\begin{proof}
Let \((x,y)\) be a primitive globally invisible point with
\(Z<x\leq2Z\). We choose an earlier integral parameter \(s<x\). Let 
$
\Delta=(Q_{M,N}(x),Q_{M,N}(s)), \ 
Q_{M,N}(x)=\Delta e.
$
By the argument of Lemma \ref{lem:polynomial-gcd-reduction}, we have \(e\mid y\). Since
\((x,y)=1\), it follows that \((x,e)=1\). As
\[
Q_{M,N}(x)=x^r(M+Nx)=\Delta e,
\]
we deduce that \(x^r\mid\Delta\). We write
$
\Delta=x^rt.
$
Then
\[
t\mid M+Nx,\qquad
e=\frac{M+Nx}{t},
\qquad
x^rt\mid Q_{M,N}(s).
\]
Consequently,
\[
V_{M,N}(Y;Z)
\leq
Y\sum_{Z<x\leq2Z}\frac1{M+Nx}
\sum_{\substack{t\mid M+Nx}}
t\,
\#\{s<x:x^rt\mid Q_{M,N}(s)\}.
\]
Let \(T\) be a parameter satisfying \(1<T\ll NZ\), to be chosen
later. We split the divisors \(t\) according as \(t\leq T\) or
\(t>T\).

For \(t\leq T\), Lemma~\ref{lem:defect-small-divisor-count} bounds
the number of \(s\)'s by \(O_m(\tau(x))\).  Also,
the sum of the divisors \(t\leq T\) of \(M+Nx\) is at most
\(T\tau(M+Nx)\).  Since \(x\asymp Z\), we have
\(M+Nx\geq NZ\), while \(M+Nx\ll_m NZ\).  Using the standard estimate
\(\tau(q)\ll_\varepsilon q^\varepsilon\), we obtain
\[
        V_{\leq T}(Y;Z)
        \ll_{r,m,\varepsilon}
        Y\frac{T}{N}(NZ)^\varepsilon.
\]
To see the factor \(1/N\) explicitly, there are \(O(Z)\) possible
values of \(x\), while the denominator \(M+Nx\) is
\(\gg NZ\); the factors of \(Z\) consequently cancel.

Consider now \(t>T\).  Set \(e=(M+Nx)/t\) as above, and put
\(k=Q_{M,N}(s)/(x^rt)\), which is an integer because
\(x^rt\mid Q_{M,N}(s)\).  We then have
\(eQ_{M,N}(s)=kQ_{M,N}(x)\).  Since \(Q_{M,N}\) is strictly
increasing on \(\R_{>0}\), the inequality \(s<x\) implies \(k<e\).
Furthermore \(t>T\) and \(M+Nx\ll_mNZ\) imply
\(e\ll_mNZ/T\).

For fixed integers \(e>k\geq1\), the relevant pairs \((s,x)\)
lie on the curve
\[
eQ_{M,N}(S)-kQ_{M,N}(T)=0.
\]
The polynomial \(Q_{M,N}(T)=T^r(M+NT)\) has the two distinct
roots \(0\) and \(-M/N\), and \(0<k/e<1\). Hence, applying
Lemma~\ref{lem:no-linear-ratio-component} with
\[
P=Q_{M,N},
\qquad
c=\frac{k}{e},
\]
we find that this curve has no linear component over \(\mathbb C\).
The second assertion of Lemma~\ref{lem:bombieri-pila} therefore
gives, uniformly in \(e,k,M,N\),
\[
\#\left\{
(s,x)\in[1,2Z]^2\cap\mathbb Z^2:
eQ_{M,N}(s)=kQ_{M,N}(x)
\right\}
\ll_{r,\varepsilon} Z^{1/2+\varepsilon}.
\]
If \(r=2\), the curve has degree \(3\). Since a reducible cubic
over \(\mathbb C\) has a linear factor, the curve is absolutely
irreducible, and the first assertion of
Lemma~\ref{lem:bombieri-pila} improves the preceding estimate to
\[
\ll_{\varepsilon} Z^{1/3+\varepsilon}.
\]
The weight attached to such a pair is $
\frac{t}{M+Nx}=\frac1e.
$
Consequently,
\[
 \begin{split}
 V_{>T}(Y;Z)
 &\ll_{r,m,\varepsilon}
 YZ^{1/2+\varepsilon}
 \sum_{e\ll_mNZ/T}\frac1e
 \sum_{1\leq k<e}1  \\
 &\ll_{r,m,\varepsilon}
 Y\frac{NZ^{3/2+\varepsilon}}{T}.
 \end{split}
\]
Indeed, for each \(e\) the inner sum has \(e-1\) terms, cancelling
the factor \(1/e\), and there are \(O_m(NZ/T)\) possible values of
\(e\).

The small- and large-divisor bounds are balanced by taking
\(T=NZ^{3/4}\).  This choice is within the natural range
\(T\ll NZ\) for \(Z\geq2\), and gives
\(V_{M,N}(Y;Z)\ll_{r,m,\varepsilon}
YN^\varepsilon Z^{3/4+\varepsilon}\).  Summing over the
\(O(\log Y)\) dyadic intervals \(Z\leq Y\) does not introduce a new
power of \(Y\), and the factor \(\log Y\) is absorbed into
\(Y^\varepsilon\).  We obtain the asserted
\(N^\varepsilon Y^{7/4+\varepsilon}\) bound.

When \(r=2\), there are
\(O_\varepsilon(Z^{1/3+\varepsilon})\) points on each ratio curve.
Repeating the preceding calculation changes the large-divisor
contribution to
\(O_{m,\varepsilon}(YNZ^{4/3+\varepsilon}/T)\).
Balancing this with the small-divisor contribution requires
\(T=NZ^{2/3}\).  The resulting dyadic estimate is
\(O_{m,\varepsilon}(YN^\varepsilon Z^{2/3+\varepsilon})\), and
dyadic summation gives
\(V_{M,N}(Y)\ll_{m,\varepsilon}
N^\varepsilon Y^{5/3+\varepsilon}\).
\end{proof}

We finally pass from primitive globally invisible points to arbitrary points
of the defect set.

\begin{proposition}
\label{prop:defect-upper-linear-cofactor}
Let \(P(T)=T^r(m+nT)\), where \(r\geq2\), \(m,n\geq1\), and
\((m,n)=1\).  Then
\[
        D_P(X)\ll_{P,\varepsilon}X^{7/4+\varepsilon}.
\]
If \(r=2\), then
\(D_P(X)\ll_{P,\varepsilon}X^{5/3+\varepsilon}\).
\end{proposition}

\begin{proof}
For \(h\geq1\), let
\[
\mathcal D_h(X)
:=
\left\{
(x,y)\in B^{\mathrm{bad}}(Y_P)\cap[1,X]^2:
\gcd(x,y)=h
\right\}.
\]
Take \((x,y)\in\mathcal D_h(X)\), and let
$
x=hx_0, y=hy_0, (x_0,y_0)=1.
$
Since \((x,y)\) is not globally visible, it has an earlier integral
point \((s,z)\) on the same branch.  Its \(p\)-adic visibility for
every prime \(p\) gives
\[
v_p(s,z)\geq v_p(x,y)=v_p(h),
\]
and hence \(h\mid s\) and \(h\mid z\). We write $
s=hs_0, z=hz_0.
$
Since \(s_0<x_0\), and the branch parameters corresponding to the two points are same, we obtain the identity
\[
\frac{y_0}{x_0^r(m+nhx_0)}
=
\frac{z_0}{s_0^r(m+nhs_0)}.
\]
Thus \((x_0,y_0)\) is a primitive globally invisible point for the
polynomial $
P_h(T):=T^r(m+nhT).$
Define
\[
g_h:=(m,h),\qquad
M_h:=\frac{m}{g_h},\qquad
N_h:=\frac{nh}{g_h}.
\]
Then $
(M_h,N_h)=1, M_h\mid m,
$
and
\[
P_h(T)=g_hQ_{M_h,N_h}(T),
\qquad
Q_{M_h,N_h}(T):=T^r(M_h+N_hT).
\]
Observe that multiplication of the polynomial by a positive constant does
not change the associated visibility family, and so, \((x_0,y_0)\) is
counted by
$
V_{M_h,N_h}\left(\frac{X}{h}\right).
$
Moreover, for fixed \(h\), the map
$
(x,y)\longmapsto (x/h,y/h)
$
is injective. Consequently,
\begin{equation}\label{eq:defect-primitive-reduction}
D_P(X)
=
\sum_{h\leq X}\#\mathcal D_h(X)
\leq
\sum_{h\leq X}
V_{M_h,N_h}\left(\frac{X}{h}\right).
\end{equation}
By Proposition \ref{prop:primitive-defect-upper} and the estimate \(N_h\leq nh\),we obtain
\[
\begin{aligned}
D_P(X)
&\ll_{r,m,\varepsilon}
\sum_{h\leq X}
N_h^\varepsilon
\left(\frac{X}{h}\right)^{7/4+\varepsilon} \\
&\ll_{P,\varepsilon}
X^{7/4+\varepsilon}
\sum_{h\leq X}h^{-7/4}
\ll_{P,\varepsilon}
X^{7/4+\varepsilon}.
\end{aligned}
\]
When \(r=2\), the corresponding estimate in Proposition \ref{prop:primitive-defect-upper} gives
\[
D_P(X)
\ll_{P,\varepsilon}
X^{5/3+\varepsilon}
\sum_{h\leq X}h^{-5/3}
\ll_{P,\varepsilon}
X^{5/3+\varepsilon}.
\]
\end{proof}
\begin{proof}[Proof of Theorem~\ref{thm:intro-quantitative-defect}]
 The general lower bound for non-monomial \(P\) follows from
Proposition~\ref{prop:general-linear-defect-lower}.  For
\(P(T)=T^r(m+nT)\), Proposition
\ref{prop:linear-cofactor-log-lower} gives
\(D_P(X)\gg_PX\log X\), while
Proposition~\ref{prop:defect-upper-linear-cofactor} gives
\(D_P(X)\ll_{P,\varepsilon}X^{7/4+\varepsilon}\).  In the cubic case
\(r=2\), the latter proposition gives the improved upper bound
\(D_P(X)\ll_{P,\varepsilon}X^{5/3+\varepsilon}\).
\end{proof}
\section{Mutual visibility}
\label{sec:random-pairs}

In this section we record a two-point version of the visibility problems
studied in this article.  If \(A,B\in\Z^n\), visibility of \(B\) from \(A\)
is measured by the displacement vector $h=B-A$. Thus a visibility datum based at the origin can be translated to any lattice
point \(A\), and the question becomes whether \(B-A\) is visible from the
origin.

\subsection{Translated visibility data}

Let \(\cY=\{(U_\alpha,\phi_\alpha)\}_{\alpha\in I}\) be a visibility datum in
\(\R^n\).  In this section we assume that \(\cY\) covers the ambient lattice:
\[
        \Z^n\setminus\{0\}
        \subset
        \bigsqcup_{\alpha\in I}\phi_\alpha(U_\alpha^+).
\]
Thus every nonzero lattice vector lies on a unique branch.  For \(A\in\Z^n\),
define the translate of \(\cY\) by \(A\) to be
\[
        A+\cY:=\{(U_\alpha,A+\phi_\alpha)\}_{\alpha\in I},
\]
where \[(A+\phi_\alpha)(t)=A+\phi_\alpha(t).\]
For \(A,B\in\Z^n\), with \(A\neq B\), we say that \(B\) is visible from \(A\)
with respect to \(\cY\) if
\[
        B-A\in\cB^{\operatorname{vis}}(\cY).
\]
Similarly, for a prime \(p\), we say that \(B\) is \(p\)-adically visible from
\(A\) if
\[
        B-A\in\cB_p^{\operatorname{vis}}(\cY).
\]
Equivalently, on \(A+\cY\) one uses the relative valuation
\[
        v_{p,A}(P):=\min_i v_p(P_i-A_i).
\]
Let
\[
        \mathcal V(\cY)
        :=
        \{(A,B)\in\Z^n\times\Z^n:A\neq B,\ B-A\in
        \cB^{\operatorname{vis}}(\cY)\},
\]
and
\[
        \mathcal V_p(\cY)
        :=
        \{(A,B)\in\Z^n\times\Z^n:A\neq B,\ B-A\in
        \cB_p^{\operatorname{vis}}(\cY)\}.
\]
Then $
        \mathcal V(\cY)=\bigcap_p\mathcal V_p(\cY) $
if and only if
$
        \cB^{\operatorname{vis}}(\cY)
        =
        \bigcap_p\cB_p^{\operatorname{vis}}(\cY).
$
Thus local detectability is unchanged by translation. 
\par Set $\Omega_X:=\{1,\dots,\lfloor X\rfloor\}^n$ and define the pair density by
\[
        \mathfrak d_{\operatorname{pair}}(\cY)
        :=
        \lim_{X\to\infty}
        \frac{\#\bigl(\mathcal V(\cY)\cap\Omega_X^2\bigr)}
             {|\Omega_X|^2},
\]
provided the limit exists.  Likewise,
\[
        \mathfrak d_{\operatorname{pair},p}(\cY)
        :=
        \lim_{X\to\infty}
        \frac{\#\bigl(\mathcal V_p(\cY)\cap\Omega_X^2\bigr)}
             {|\Omega_X|^2}.
\]

\subsection{The signed weighted homogeneous datum}

\par In this subsection, we fix $w_1,\dots,w_n\in\Z_{\geq1}$ such that $\gcd(w_1,\dots,w_n)=1$ and set $W:=w_1+\cdots+w_n$. Set $\Sigma:=\{-1,0,+1\}^n\setminus\{\mathbf{0}\}$. For \(\sigma=(\sigma_1,\dots,\sigma_n)\in\Sigma\), consider the set of vectors $\alpha=(\alpha_1,\dots,\alpha_n)\in\Q_{\geq0}^n$ satisfying
\[
        \alpha_i=0\Longleftrightarrow \sigma_i=0.
\]
Two such pairs \((\sigma,\alpha)\) and \((\sigma,\beta)\) are declared
equivalent if there exists \(c\in\R_{>0}\) such that
\[
        \beta_i=c^{w_i}\alpha_i
        \qquad
        \text{for every }i.
\]We let \(I\) be the set of equivalence
classes and choose one rational representative \((\sigma,\alpha)\) for each
class. Let $U_{\sigma,\alpha}:=\mathbb{R}$ and set \[\phi_{\sigma,\alpha}(t)
        =
        \bigl(
        \sigma_1\alpha_1t^{w_1},\dots,
        \sigma_n\alpha_nt^{w_n}
        \bigr).
\]
We denote the resulting signed weighted datum by \(\mathscr{Z}_w\).

\begin{lemma}\label{lem:signed-weighted-divisibility}
Let \(h=(h_1,\dots,h_n)\in\Z^n\) be such that \(h_i\neq 0\) for every \(i\).
Then \(h\) is not visible from the origin with respect to
\(\mathscr Z_w\) if and only if there exists an integer \(m\geq2\) such that
\[
        m^{w_i}\mid h_i
        \qquad
        \text{for every }1\leq i\leq n.
\]
\end{lemma}

\begin{proof}
The proof is similar to the proof of Theorem \ref{thm:local-detectability}.
\end{proof}

\begin{theorem}
\label{thm:signed-weighted-pair-visibility}
Let \(\mathscr{Z}_w\) be the signed weighted datum above.  Then
\[
        \mathcal V(\mathscr{Z}_w)
        =
        \bigcap_p\mathcal V_p(\mathscr{Z}_w).
\]
Moreover,
\[
        \mathfrak d_{\operatorname{pair},p}(\mathscr{Z}_w)
        =
        1-\frac1{p^W}
\]
for every prime \(p\), and
\[
        \mathfrak d_{\operatorname{pair}}(\mathscr{Z}_w)
        =
        \prod_p\left(1-\frac1{p^W}\right)
        =
        \frac1{\zeta(W)}.
\]
\end{theorem}

\begin{proof}
Since \(\mathscr Z_w\) is homogeneous for the positive-weight action and
visibility of \(B\) from \(A\) is equivalent, after translation, to
visibility of \(B-A\) from the origin, Theorem
\ref{thm:local-detectability} gives
\[
        \mathcal V(\mathscr Z_w)
        =
        \bigcap_p\mathcal V_p(\mathscr Z_w).
\]
It remains to compute the corresponding local and global densities.

Write \(N=\lfloor X\rfloor\) and
\(\Omega_X=\{1,\dots,N\}^n\), so that
\(|\Omega_X|^2=N^{2n}\).  For \((A,B)\in\Omega_X^2\), put
\(h=B-A=(h_1,\dots,h_n)\).  We first discard the pairs for which
\(h_i=0\) for at least one \(i\).  For each fixed \(i\), the condition
\(A_i=B_i\) gives \(N^{2n-1}\) pairs, and hence the union of these
coordinate hyperplanes contains \(O(N^{2n-1})\) pairs.  It therefore has
pair density zero. Therefore, we restrict to displacement vectors \(h\) such that \(h_i\neq0\) for all $i$, and apply
Lemma~\ref{lem:signed-weighted-divisibility}.

For \(m\geq1\), let
\[
        N_m^\ast(X)
        :=
        \#\left\{
        (A,B)\in\Omega_X^2:
        h_i\neq0\ \forall i,\quad
        m^{w_i}\mid h_i\ \forall i
        \right\}.
\]
For fixed \(q\), the number of pairs
\((a,b)\in\{1,\dots,N\}^2\) satisfying \(a\equiv b\pmod q\) is
\(N^2/q+O_q(N)\).  Applying this in the \(i\)-th coordinate with
\(q=m^{w_i}\), and then removing the zero-coordinate locus, gives, for
each fixed \(m\),
\begin{equation}\label{eq:weighted-Nm-asymptotic}
        N_m^\ast(X)
        =
        \frac{N^{2n}}{m^W}
        +O_m(N^{2n-1}).
\end{equation}
In particular,
\begin{equation}\label{eq:weighted-Nm-limit}
        \lim_{X\to\infty}
        \frac{N_m^\ast(X)}{N^{2n}}
        =
        \frac1{m^W}.
\end{equation}

We now impose visibility by Möbius inversion.  For a displacement vector
\(h\) with all coordinates nonzero, set
\[
        D_w(h)
        :=
        \prod_p
        p^{\,\min_i\lfloor v_p(|h_i|)/w_i\rfloor}.
\]
Then \(m\mid D_w(h)\) if and only if
\(m^{w_i}\mid h_i\) for every \(i\).  By
Lemma~\ref{lem:signed-weighted-divisibility}, \(h\) is visible precisely
when \(D_w(h)=1\).  Thus the standard identity
\(\sum_{m\mid d}\mu(m)=1\) if \(d=1\), and \(0\) otherwise, gives
\[
        \mathbf 1_{\{B\text{ visible from }A\}}
        =
        \sum_{\substack{m\geq1\\
        m^{w_i}\mid B_i-A_i\ \forall i}}
        \mu(m).
\]
Summing over pairs outside the zero-coordinate locus, we obtain
\begin{equation}\label{eq:weighted-visible-mobius}
        \#\bigl(
        \mathcal V(\mathscr Z_w)\cap\Omega_X^2
        \bigr)
        =
        \sum_{m\geq1}\mu(m)N_m^\ast(X)
        +O(N^{2n-1}).
\end{equation}

To justify the passage to the limit, we need a uniform bound in \(m\).
If \(m^{w_i}\mid h_i\) and \(0<|h_i|<N\), then there are at most
\(2N/m^{w_i}\) possibilities for \(h_i\). Hence, the number of possible
displacement vectors is at most \(2^nN^n/m^W\).  Each fixed displacement
vector occurs for at most \(N^n\) ordered pairs, and therefore
\begin{equation}\label{eq:weighted-Nm-uniform}
        N_m^\ast(X)
        \ll
        \frac{N^{2n}}{m^W},
\end{equation}
with an implied constant depending only on \(n\).  Since \(W>1\),
the series \(\sum_{m\geq1}m^{-W}\) converges.  Equations
\eqref{eq:weighted-Nm-limit} and \eqref{eq:weighted-Nm-uniform}
therefore allow us to pass the limit through the sum in
\eqref{eq:weighted-visible-mobius}.  We obtain
\[
\begin{split}
        \mathfrak d_{\operatorname{pair}}(\mathscr Z_w)
        &=
        \sum_{m\geq1}\frac{\mu(m)}{m^W}
        =
        \frac1{\zeta(W)}.
\end{split}
\]

It remains to identify the local factor.  Fix a prime \(p\).  By the prime
formulation of Lemma~\ref{lem:signed-weighted-divisibility}, the local
obstruction is precisely that
\(p^{w_i}\mid B_i-A_i\) for every \(i\).  In one coordinate, the
congruence \(B_i\equiv A_i\pmod{p^{w_i}}\) has density \(p^{-w_i}\)
among ordered pairs.  Since the coordinates are independent, the
simultaneous obstruction has density
\[
        \prod_{i=1}^n p^{-w_i}
        =
        p^{-W}.
\]
Consequently,
\[
        \mathfrak d_{\operatorname{pair},p}(\mathscr Z_w)
        =
        1-\frac1{p^W}.
\]
Taking the product over all primes, we obtain
\[
        \prod_p
        \mathfrak d_{\operatorname{pair},p}(\mathscr Z_w)
        =
        \prod_p\left(1-\frac1{p^W}\right)
        =
        \frac1{\zeta(W)}
        =
        \mathfrak d_{\operatorname{pair}}(\mathscr Z_w),
\]
which completes the proof.
\end{proof}

\subsection{Examples}

We first recover the classical visibility between two lattice points.  Take all
weights equal to \(1\).  The corresponding signed weighted datum is the
rational straight-line datum, whose branches may be written as
\(\phi_{\sigma,\alpha}(t)=(\sigma_1\alpha_1t,\dots,
\sigma_n\alpha_nt)\), where each \(\sigma_i\) belongs to
\(\{-1,0,+1\}\), each \(\alpha_i\) is a nonnegative rational number, and
\(\alpha_i=0\) precisely when \(\sigma_i=0\).  For distinct lattice points
\(A\) and \(B\), visibility of \(B\) from \(A\) is equivalent to the
primitivity of the displacement vector \(B-A\); that is,
\(\gcd(B_1-A_1,\dots,B_n-A_n)=1\).  At a prime \(p\), the local obstruction
is simply that \(A\) and \(B\) are congruent coordinatewise modulo \(p\).
Thus Theorem~\ref{thm:signed-weighted-pair-visibility} gives the classical
density \(1/\zeta(n)\) for visible ordered pairs.

We next consider signed generalized lines of sight.  Here \(n=2\) and the
weights are \(1\) and \(b\).  The branches may be written as
\(\phi_{\sigma,\tau,\alpha}(t)
=(\sigma\alpha_1t,\tau\alpha_2t^b)\), where
\(\sigma,\tau\in\{-1,0,+1\}\), the parameters
\(\alpha_1,\alpha_2\) are nonnegative rationals, and a parameter is zero
precisely when the corresponding sign is zero.  If
\(B-A=(u,v)\) has both coordinates nonzero, then \(B\) is visible from
\(A\) precisely when there is no prime \(p\) for which \(p\) divides \(u\)
and \(p^b\) divides \(v\).  The local obstruction at \(p\) therefore has
density \(p^{-(b+1)}\), and the density of visible ordered pairs is
\(1/\zeta(b+1)\).  When \(b=1\), this specializes to the classical planar
density \(1/\zeta(2)\).

\bibliographystyle{alpha}
\bibliography{references}
\end{document}